\documentclass[11pt,reqno]{amsart}
\usepackage[T1]{fontenc}
\usepackage{lmodern,amsmath,amssymb,amsthm,mathtools,mathrsfs}
\usepackage[a4paper,margin=27mm]{geometry}
\usepackage{microtype,enumitem,needspace}
\usepackage[hidelinks,pdfusetitle]{hyperref}
\hypersetup{
 pdftitle={Hyperbolic Area Methods in Meromorphic Dynamics and Elliptic Polynomial Skew Products},
 pdfsubject={Hyperbolic Area Methods in Meromorphic Dynamics and Elliptic Polynomial Skew Products},
 pdfkeywords={wandering domains, hyperbolic area, singular orbit, meromorphic map, polynomial skew product}
}
\newcommand{\proofstep}[1]{\par\smallskip\noindent\emph{#1.}\enspace}
\setlist[enumerate]{leftmargin=2em,itemsep=2pt}
\newcommand{\C}{\mathbb C}
\newcommand{\D}{\mathbb D}
\newcommand{\R}{\mathbb R}
\newcommand{\Z}{\mathbb Z}
\newcommand{\sphere}{\widehat{\mathbb C}}
\newcommand{\Fat}{\mathcal F}
\newcommand{\Jul}{\mathcal J}
\newcommand{\dd}{\,\mathrm d}
\newtheorem{theorem}{Theorem}[section]
\newtheorem{maintheorem}{Theorem}

\newtheorem{proposition}[theorem]{Proposition}
\newtheorem{lemma}[theorem]{Lemma}
\newtheorem{corollary}[theorem]{Corollary}
\theoremstyle{definition}
\newtheorem{definition}[theorem]{Definition}
\newcounter{introductionquestion}
\newenvironment{question}[1]{\par\medskip\noindent\refstepcounter{introductionquestion}\textbf{Question~\theintroductionquestion}\ \textnormal{(#1).}\ \itshape}{\par\medskip}
\newtheorem{example}[theorem]{Example}
\numberwithin{equation}{section}
\title[Hyperbolic Area Methods and Skew Products]{Hyperbolic Area Methods in Meromorphic Dynamics and Elliptic Polynomial Skew Products}

\author{Zihao Ye}
\address{Zihao Ye: School of Mathematical Sciences, East China Normal University, Shanghai 200241, China}
\email{52290155037@stu.ecnu.edu.cn}

\date{September 24, 2026}

\newcommand{\Chat}{\sphere}
\newcommand{\F}{\mathcal F}
\newcommand{\N}{\mathcal N}
\newcommand{\RR}{\mathcal R}
\newcommand{\SSS}{\mathcal S}
\DeclareMathOperator{\dist}{dist}
\DeclareMathOperator{\Rea}{Re}
\DeclareMathOperator{\Ima}{Im}
\DeclareMathOperator{\Int}{Int}
\DeclareMathOperator{\Log}{Log}
\theoremstyle{remark}

\begin{document}
\begin{abstract}
We use hyperbolic area and covering geometry to study wandering Fatou components and singular orbits.
We give proofs without quasiconformal deformation of the rational no-wandering theorem and the known no-wandering results in Bergweiler's Question~9.
The framework also excludes wandering for entire functions with compact singular sets whose derived sets lie in the Fatou set.
For Bergweiler's Question~8, every wandering component of a transcendental meromorphic map has a subsequence of iterates converging locally uniformly to infinity.
For Bergweiler's Question~4, we establish hyperbolic separation near singular-free Baker cycles after deleting any finite set of original singular values as orbit generators.
For Bergweiler's Question~11, corresponding finite-deletion boundary estimates hold along wandering tails.
In higher dimension, a relative-area theorem excludes bounded wandering components of polynomial skew products over bounded bases carrying an absolutely continuous invariant probability of full support.
Applied to elliptic polynomial skew products over Siegel rotation domains, it yields a local no-wandering theorem without conditions on the fiber critical orbits or continuity of fiber Julia sets.
A global criterion covers regular skew products on $\mathbb P^2$ whose bounded base Fatou components eventually enter periodic Siegel disks, including $(\lambda z+z^d,q(z,w))$ for Brjuno $\lambda$ and polynomial $q$ of total degree at most $d$ with a nonzero $w^d$ coefficient.
\end{abstract}
\maketitle
{\small\tableofcontents}

\section{Introduction}

Sullivan's theorem \cite{Sullivan} settled the Fatou--Julia question of wandering domains for rational maps.
Eremenko and Lyubich \cite{EL} and Goldberg and Keen \cite{GoldbergKeen} obtained the corresponding conclusion for entire functions with finitely many inverse singular values; further meromorphic no-wandering theorems appear in \cite{BakerKotusLu,Stallard,BergweilerNewton,BergweilerTerglane}.
Eremenko and Lyubich \cite{EremenkoLyubichProblems} emphasized the role of finite-dimensional quasiconformal deformation in the finite-type setting.
Bergweiler subsequently asked for proofs of the rational and meromorphic theorems that avoid quasiconformal mappings \cite[Section~4.6, Question~9]{Bergweiler}.
Here the alternative mechanism is geometric: covering maps preserve hyperbolic area, adding finitely many punctures has a finite area cost, and a wandering tail gives rise to disjoint regions of unbounded intrinsic area.
We develop this mechanism into covering and area estimates and then study wandering orbits and singular values by related methods.

Write $\Fat(f)$ and $\Jul(f)$ for the Fatou and Julia sets, and $\sphere$ for the Riemann sphere.
By a forward Fatou component $U_n$ we mean the \emph{full} component containing $f^n(U_0)$; $U_0$ is wandering if these components are pairwise distinct.
For a transcendental entire function, $S(f)$ is the closure in $\C$ of its finite critical and asymptotic values.
For a transcendental meromorphic function, $S(f)$ denotes the closed spherical inverse-singular set; the latter convention records a singularity over infinity when appropriate.
All iterates are taken only where they are defined, and an orbit reaching infinity stops there.
Bergweiler's $\operatorname{sing}(f^{-1})$ consists of finite critical and asymptotic values and their finite limit points \cite[Section~4.3]{Bergweiler}; in the questions below we retain this historical notation and use our spherical $S(f)$ in the theorems.
A prime denotes the spherical derived set.

\begin{question}{cf. Bergweiler, Question~9 \cite[Section~4.6]{Bergweiler}}\label{question:q9}
Can the no-wandering results called Theorems~11 and~12 in Bergweiler's article be proved without quasiconformal mappings?
\end{question}

The two historical theorems concern rational maps and four classes of transcendental meromorphic maps, respectively.
With $\rho(f)$ the Nevanlinna order, the classes in Bergweiler's statement are
\begin{align*}
 \SSS&=\{f:|S(f)|<\infty\},\\
 \F&=\{f(z)=z+R(z)e^{Q(z)}:
               R\text{ rational},\ Q\text{ polynomial}\},\\
 \N&=\{f:\rho(f)<\infty,\ f'=r e^p(f-z),
               \ r\text{ rational},\ p\text{ polynomial}\},\\
 \RR_1&=\{f:f'=r(f-z)^2,\ r\text{ rational}\},\\
 \RR_2&=\{f:f'=r(f-z)(f-\tau),
               \ r\text{ rational},\ \tau\in\C\},
 \qquad\RR=\RR_1\cup\RR_2.
\end{align*}
Only the definition of $\N$ explicitly imposes finite order.

\begin{maintheorem}\label{thm:A-q9}\label{core:sullivan}\label{thm:main}
The following no-wandering conclusions admit proofs without quasiconformal mappings.
\begin{enumerate}[label=\textup{(\roman*)}]
\item Every Fatou component of a rational map
$R:\sphere\to\sphere$ of degree at least two is preperiodic.
\item No transcendental meromorphic function in
$\SSS\cup\F\cup\N\cup\RR$ has a wandering Fatou component.
\end{enumerate}
\end{maintheorem}

The rational assertion is Sullivan's theorem \cite{Sullivan}; degree-at-most-one maps have the same conclusion by elementary classification.
For the entire subclass of $\SSS$, the classical no-wandering result was proved by Eremenko and Lyubich \cite{EL} and independently by Goldberg and Keen \cite{GoldbergKeen}; Baker, Kotus and L\"u proved the meromorphic $\SSS$ result \cite{BakerKotusLu}.
The remaining assertions were obtained by Stallard for $\F$ \cite{Stallard}, Bergweiler for $\N$ \cite{BergweilerNewton}, and Bergweiler and Terglane for $\RR$ \cite{BergweilerTerglane}.
The contribution of Theorem~\ref{thm:A-q9} is to prove the full stated scope by area estimates, covering arguments and differential equations, without the quasiconformal deformations used in the classical proofs.
All cases in Question~\ref{question:q9} are proved in Section~\ref{sec:q9} by this hyperbolic-area method.
In particular, its $\RR_2$ assertion does not add a finite-order assumption; the required growth bound is derived from the Riccati equation in its proof.
A related result established in the first version of this paper is retained here as a supplement:

\begin{theorem}\label{core:singular-accumulation}
Let $f:\C\to\C$ be transcendental entire.
If $S(f)$ is compact and its derived set satisfies $S(f)'\subset\Fat(f)$, then $f$ has no wandering Fatou components.
\end{theorem}

This supplement permits infinitely many singular values, with only finitely many lying in the Julia set.
It imposes no condition on the forward orbits of those Julia singular values.
It is a no-wandering result outside the four classes as stated above, not an additional Bergweiler question.

The same estimates also bear on the orbit of an individual wandering component.
Eremenko and Lyubich \cite[Example~1]{EremenkoLyubichExamples} constructed an entire wandering domain with infinitely many finite limit functions; infinity is a limit function in their example as well.
Bergweiler formulated the following question for all transcendental meromorphic functions, with infinity required only as a subsequential limit.

\begin{question}{cf. Bergweiler, Question~8 \cite[Section~4.5]{Bergweiler}, for the meromorphic case}\label{question:q8}
If $U$ is a wandering domain of a transcendental meromorphic function $f$, can a subsequence of $f^n|_U$ converge to infinity?
\end{question}

The entire-function case was raised in related forms.
Hayman and Lingham's revised Problem~2.77, attributed there to Baker, Herman and Kra (communicated by Hamilton), and their Problem~2.87, attributed to Herman, Eremenko and Lyubich, ask whether all forward images of an entire wandering domain can remain in one bounded set \cite[Problems~2.77 and~2.87]{HaymanLingham}.
The qualifier ``uniformly'' in Problem~2.77 was added to the earlier formulation \cite[Update~2.77]{HaymanLingham}; see also the 1989 list \cite{BrannanHayman}.
Eremenko and Lyubich asked whether the orbit of an entire wandering domain can be bounded \cite{EremenkoLyubichProblems}.
Eremenko's distinct Problem~2.67 asks whether there can be an infinite bounded set of constant limit functions \cite[Problem~2.67]{HaymanLingham}.
A common bound on whole forward images, a bounded orbit of one point, an infinite bounded set of limit functions, and the locally uniform subsequence in Question~\ref{question:q8} are different formulations.

Here convergence means locally uniform convergence in the spherical metric.
The word ``subsequence'' matters: the whole forward orbit need not converge to infinity, as examples of Pardo-Sim\'on and Sixsmith show \cite{PardoSimonSixsmith2024}.

\begin{maintheorem}\label{q8r2:headline}
Let $f:\C\to\sphere$ be transcendental meromorphic and let $U$ be a wandering Fatou component.
There is a sequence $n_k\to\infty$ such that $f^{n_k}|_U\to\infty$ locally uniformly in the spherical metric.
Equivalently, no point of $U$ has a bounded forward orbit.
\end{maintheorem}

Prochorov, Rempe and Waterman \cite[Theorem~1.2]{PRW} obtained this conclusion for entire functions.
Drawing on further ideas of their own, they also obtained a valuable local extension to meromorphic mappings \cite[Theorem~1.3(1)]{PRW}.
Their local theorem assumes simple connectivity of the relevant trapped components; the assertion here places no connectivity or singular-value condition on $U$.
The proof uses a bounded-window area estimate and, for multiply connected wandering tails, divergence of injectivity radii.
The conclusion of Theorem~\ref{q8r2:headline} was obtained independently before the author became aware of their preprint, which was submitted to arXiv on 23 September 2026.
The first version of the present paper introduced the hyperbolic-area framework used here; both treatments of Bergweiler's Question~8 use local area estimates.

The last two questions concern inverse singularities near the boundaries of Baker cycles and wandering domains.
Their original formulations ask for a relation; they do not specify the quantitative finite-deletion assertions below.
To state those assertions, put $S=S(f)$ for a transcendental meromorphic $f$ and, for every finite set $T\subset S$ of \emph{original singular-value generators}, set
\[
 \mathcal O_T=\{f^k(s):s\in S\setminus T,\ k\ge0,\
                          f^k(s)\text{ is defined}\},
 \qquad
 A_T=\overline{\mathcal O_T}^{\,\sphere}.
\]
The value infinity is included when the corresponding iterate is defined, but $f(\infty)$ is never evaluated.
Deletion occurs before forward iteration: $A_T$ is generally different from the full postsingular closure with $T$ subtracted.
If retained generators reach or approach a deleted value, that value can still lie in $A_T$.

Bergweiler observed that, unlike attracting and parabolic cycles, a cycle of Baker domains need not contain an inverse singularity \cite[Theorem~7 and Section~4.3]{Bergweiler}.
He asked what can nevertheless be detected at its boundary:

\begin{question}{cf. Bergweiler, Question~4 \cite[Section~4.3]{Bergweiler}}\label{question:q4}
Let $f$ be a meromorphic function with a cycle of Baker domains containing no point of $\operatorname{sing}(f^{-1})$.
How are the inverse singularities related to the boundaries of the domains in the cycle?
\end{question}

\begin{maintheorem}\label{q4:main}
Let $f:\C\to\sphere$ be transcendental meromorphic.
Let $U_0,\ldots,U_{p-1}$ be a Baker cycle with $U_i\cap S=\varnothing$,
and let $a_i\in\partial U_i$ be the locally uniform spherical limit
of $f^{np}|_{U_i}$.
For every finite $T\subset S$ and every $0\le i<p$, one has $a_i\in A_T'$.
In fact, there are pairwise distinct $s_k\in S\setminus T$ whose defined forward images $f^{m_k}(s_k)$ are also pairwise distinct and tend to $a_i$.
If $A_T$ avoids the cycle, then $Z_T=\sphere\setminus A_T$ is hyperbolic and
\begin{equation}\label{eq:baker-head-separation}
 d_{Z_T}\bigl(f^n(z),Z_T\setminus U_{n\bmod p}\bigr)
           \longrightarrow\infty\qquad(z\in U_0),
\end{equation}
locally uniformly in $z$.
\end{maintheorem}

Already in the original survey, Bergweiler proved that infinity is an accumulation point of the first $p$ forward images of inverse singular values for every Baker cycle \cite[Theorem~16]{Bergweiler}.
Baker, Dom\'inguez and Herring proved a singular-accumulation obstruction for Baker domains \cite[Theorem~F]{BakerDominguezHerring}; see also Baker's formulation for maps meromorphic outside a small set \cite[Theorem~2]{Baker2002}.
Zheng \cite[Theorem~2.2 and p.~30]{Zheng2005} located the individual phase limits in the derived singular set of the return iterate.
The elementary finite-iterate singular-value relation then gives the finite-deletion accumulation assertion of Theorem~\ref{q4:main}; the proof below also derives it directly from the area framework.
Earlier full-postsingular separation and proximity results appear in \cite[Theorem~3]{Bergweiler1995} and \cite[Propositions~7.4 and~7.5]{MBRG}; see also \cite[Theorem~A]{BFJK}.
The additional conclusion here is that hyperbolic separation persists when any finite set of original singular values is removed as a source of forward orbits.

Bergweiler also asked whether wandering domains require infinitely many inverse singularities and, in particular, what their boundaries must detect when every forward component avoids the singular set \cite[Section~4.6,
Questions~10--11]{Bergweiler}:

\begin{question}{cf. Bergweiler, Question~11 \cite[Section~4.6]{Bergweiler}}\label{question:q11}
Let $f$ be a meromorphic function with a wandering domain $U$, and let $U_n$ be the Fatou component containing $f^n(U)$.
If $U_n\cap\operatorname{sing}(f^{-1})=\varnothing$ for every $n\ge0$,
what is the relation between $\partial U_n$ and the inverse singularities?
\end{question}

For a singular-free tail choose compatible universal covers $\pi_n:\D\to U_n$ with $\pi_{n+1}=f\circ\pi_n$.
Their deck groups satisfy $\Gamma_n\subset\Gamma_{n+1}$.
The tail is \emph{thick} if $\bigcup_n\Gamma_n$ is discrete and \emph{thin} otherwise; the covering classification is given in
Section~\ref{sec:covering-loss}.  Distances to $\partial U_n$
and $A_T\cap\C$ in the following statement are Euclidean.

\begin{maintheorem}\label{q11all:main}
Let $f:\C\to\sphere$ be transcendental meromorphic and let $(U_n)$ be the full forward components of a wandering domain, with $U_n\cap S=\varnothing$ for all $n$.
Fix any finite $T\subset S$.
If $A_T$ meets a forward component, there exist $N$ and finite points $p_n\in\mathcal O_T\cap U_n$ for $n\ge N$ such that $f(p_n)=p_{n+1}$.

Otherwise, for each $z\in U_0$, set $z_n=f^n(z)$ and
$d_n=\dist(z_n,\partial U_n)$.  Then
\begin{equation}\label{q11all:distance-ratio}
 \frac{\dist(z_n,A_T\cap\C)}{d_n}\longrightarrow1,
\end{equation}
and there are actual finite singular-orbit points $p_n\in\mathcal O_T$ such that
\begin{equation}\label{q11all:actual-witness}
 \frac{|p_n-z_n|}{d_n}\longrightarrow1,
 \qquad
 \frac{\dist(p_n,U_n)}{d_n}\longrightarrow0.
\end{equation}
In this second case, if the tail is thick, then $Z_T=\sphere\setminus A_T$ is hyperbolic and
\begin{equation}\label{q11all:thick-hyperbolic}
 d_{Z_T}(z_n,Z_T\setminus U_n)\longrightarrow\infty
\end{equation}
locally uniformly for $z\in U_0$.

In all cases, every constant normal limit belongs to $A_T$.
For a thick tail all such limits belong to $A_T'$; for a thin tail at least one belongs to $(S\setminus T)'=S'$.
\end{maintheorem}

For the full postsingular set, wandering-domain limit relations were obtained for entire functions in \cite{BHKMT1993} and for meromorphic functions in \cite[Theorem~1]{Baker2002}, \cite{Zheng2003} and \cite[Proposition~7.3]{MBRG}.
Prochorov, Rempe and Waterman \cite[p.~3]{PRW} also announced an unconditional derived-singular-set limit statement for entire wandering domains; its proof is deferred to a later revision of their preprint.
When $T=\varnothing$ and the full postsingular set avoids the forward components, the distance ratio and the actual-point approximation follow already from \cite[Lemma~2.1 and the proof of Theorem~B]{BFJK}.
Full-postsingular hyperbolic separation in this setting also follows from \cite[Proposition~7.4]{MBRG}.
Theorem~\ref{q11all:main} retains the Euclidean estimates and actual orbit witnesses after deleting \emph{any} finite set of original singular values as orbit generators, even if the deleted values have orbits entering the forward domains.
In thick tails hyperbolic separation survives the same deletion; for thin tails we assert only the stated Euclidean estimates and limit-set conclusions.
The ratio $\dist(p_n,U_n)/d_n\to0$ does not by itself imply absolute Euclidean convergence to the boundary when $d_n$ is unbounded.

The area argument also applies to polynomial skew products.
Raissy asked whether wandering Fatou components can occur near an invariant fiber \cite[Question~2]{RaissySurvey}; the question concerns all types of invariant fibers and has positive examples in the parabolic case \cite[Section~3]{RaissySurvey}.
We treat the linearizable elliptic case for polynomial fibers of constant degree.
The relative theorem in Section~\ref{sec:elliptic-skew} handles the bounded dynamics over an invariant rotation domain; together with the dynamics at infinity, it gives the following global criterion.

A polynomial skew product $F(z,w)=(p(z),q(z,w))$ is \emph{regular of degree $d$} if the homogeneous degree-$d$ parts of its two coordinates have no common nonzero zero.
Equivalently, $F$ extends to a holomorphic endomorphism of $\mathbb P^2$ of algebraic degree $d$.

\begin{maintheorem}\label{thm:elliptic-global}
Let $F(z,w)=(p(z),q(z,w))$ be a regular polynomial skew product of degree $d\ge2$.
Suppose that every bounded Fatou component of $p$ eventually maps to a periodic Siegel disk.
Then $F$ has no wandering Fatou component in $\mathbb P^2$.
\end{maintheorem}

The analytic input, Theorem~\ref{thm:elliptic-skew}, uses a bounded base carrying an absolutely continuous invariant probability of full support.
Under uniform bounds on the fiber coefficients and a leading coefficient bounded away from zero, it excludes wandering components of the bounded-orbit set.
Every bounded Siegel disk carries the required measure in its linearizing coordinate.
For a unitary base on a ball, with fiber coefficients holomorphic near the closed ball and a leading coefficient nonzero there, the common escape region also excludes wandering in the compactified relative skew product.
This settles the local no-wandering question in the specified linearizable elliptic setting without assumptions on the critical orbits of the central fiber.
Peters and Raissy \cite[Theorem~1]{PetersRaissy} proved both no wandering and a stronger bulging description under a Brjuno condition and a hypothesis that all central-fiber critical points lie in attracting or parabolic basins.
The present result does not assert that description or cover nonlinearizable elliptic bases.

In particular, Theorem~\ref{thm:elliptic-global} applies to
\[
 F(z,w)=(\lambda z+z^d,q(z,w))
\]
for $d\ge2$, irrational Brjuno $\lambda$, and any polynomial $q$ of total degree at most $d$ with nonzero $w^d$ coefficient.
Example~\ref{ex:double-siegel-global} has central-fiber critical points in the Julia set, outside the assumptions of \cite[Theorem~1]{PetersRaissy}.
Boc Thaler, Forn\ae ss and Peters \cite[Theorem~7 and Remark~2.1]{BocThalerFornaessPeters2015} constructed a regular cubic skew product with an invariant Fatou component having a punctured-disk limit set,
while the Siegel disk in its central fiber lies entirely in the two-dimensional Julia set.
For a bounded-type choice of their rotation number, our global criterion also excludes wandering components for this same map; see Example~\ref{ex:bfp-global}.
The global proof uses Lilov's theorem at infinity as stated in \cite[Theorem~1.1]{PetersVivas2016}.
The attracting-base results of Peters and Smit \cite{PetersSmit2018} and Ji and Shen \cite{JiShen2026} concern a different transverse regime.
Boc Thaler's elliptic-fixed-point wandering example in complex dimension three \cite[Section~6]{BocThaler2026} has a two-dimensional complex fiber over its rotation coordinate and therefore lies outside the one-dimensional fiber setting here.

All main theorems in this paper were obtained before the first version was posted as arXiv:2609.23834 \cite{Ye2026v1}.
The author subsequently completed the organization and writing of the manuscript.

Section~\ref{sec:background} records standard facts and notation.
Section~\ref{sec:area-framework} develops the covering and area tools.
Section~\ref{sec:q9} proves Theorem~\ref{thm:A-q9} and the entire-function supplement.
Section~\ref{sec:q8} proves Theorem~\ref{q8r2:headline}.
Sections~\ref{sec:q4} and \ref{sec:q11} address Questions~\ref{question:q4} and~\ref{question:q11} on Baker cycles and wandering boundaries, respectively.
Sections~\ref{sec:elliptic-skew} and \ref{sec:elliptic-global} give the relative and global skew-product applications.

\section{Background}\label{sec:background}\label{sec:prelim}
\subsection{Hyperbolic area}\label{sec:area}

\begin{definition}\label{def:area}
On a hyperbolic open set $G\subset\sphere$, write the complete Poincar\'e metric of curvature $-4$ as $\rho_G(z)|dz|$ in local coordinates.
Its induced distance $d_G$ and normalized area measure $\alpha_G$ are defined componentwise, with
\[
 \dd\alpha_G=\frac2\pi\rho_G^2\,\dd x\dd y,
 \qquad \rho_\D(z)=\frac1{1-|z|^2}.
\]
\end{definition}

For a holomorphic map $\psi:G\to H$ between hyperbolic open sets, Schwarz--Pick gives $\psi^*\dd\alpha_H\le\dd\alpha_G$, with equality when $\psi$ is a covering.
In particular, densities increase under domain restriction.
For a finite set $F\subset\sphere$ with $|F|\ge3$, Gauss--Bonnet gives $\alpha_{\sphere\setminus F}(\sphere\setminus F)=|F|-2$.
Every simply connected hyperbolic domain has infinite intrinsic area.

\begin{lemma}\label{core:kernel}
Let $F_j\subset\sphere$ be increasing finite sets with $|F_1|\ge3$, and put $F_\infty=\overline{\bigcup_jF_j}$.
On each component $D$ of $\sphere\setminus F_\infty$, the densities $\rho_{\sphere\setminus F_j}$ increase locally uniformly to $\rho_D$.
\end{lemma}
\begin{proof}
Fix a base point in $D$.
Every neighborhood of a point of $F_\infty$ meets some $F_j$, so the domains $\sphere\setminus F_j$ and all their subsequences have Carath\'eodory kernel $D$.
Hejhal's kernel theorem \cite{Hejhal} gives convergence of the normalized universal covers.
At the base point, the density is the reciprocal of the absolute value of the covering map's derivative; hence the densities converge pointwise to $\rho_D$.
Monotonicity and Dini's theorem give local uniform convergence.
\end{proof}

\subsection{Repelling markings}

For each rational map of degree at least two, or transcendental entire or meromorphic map considered below, fix increasing finite unions $P_j$ of complete repelling periodic orbits such that
\begin{equation}\label{core:marks}
 f(P_j)=P_j,\qquad |P_1|\ge3,\qquad
 \overline{\bigcup_jP_j}=\Jul(f).
\end{equation}
Their existence follows from the density of repelling periodic points \cite[Theorem~4]{Bergweiler}.
For the meromorphic results, this is the choice referred to as the \emph{repelling marking}.
\begin{definition}\label{def:marking}
The sets $(P_j)$ satisfying \eqref{core:marks} are repelling markings.
\end{definition}

For a disk map $H:\D\to\D$ with $H(0)=0$ and $H'(0)=a\in(0,1]$, Schwarz--Pick applied to $H(\zeta)/\zeta$ gives
\begin{equation}\label{q8new:near-id}
 \sup_{|\zeta|\le r}|H(\zeta)-\zeta|
 \le\frac{r(1+r)}{1-r}(1-a),\qquad 0<r<1.
\end{equation}

\subsection{Inverse singularities and covering transitions}

In the results on meromorphic maps below, $f:\C\to\sphere$ denotes a transcendental meromorphic function unless stated otherwise.
We write $S=S(f)$ for its closed spherical set of inverse singular values and $S_\C=S\cap\C$.
Thus $S_\C$ agrees with the singular set used in the introduction when $f$ is entire.
The spherical convention also records a singularity over infinity when one is present.
All iterates are taken only where they are defined; an orbit reaching infinity stops there.
In particular, Fatou components contain no pre-poles, and their iterates are defined throughout.
We use full Fatou components: $U_n$ is the component containing $f^n(U_0)$.

\begin{lemma}\label{q4:transitions}
If $U_{n+1}\cap S=\varnothing$, then $f:U_n\to U_{n+1}$ is a holomorphic covering onto $U_{n+1}$.
\end{lemma}
\begin{proof}
The inverse-singularity covering theorem makes every component of $f^{-1}(U_{n+1})$ a covering of $U_{n+1}$.
The component containing $f^n(U_0)$ is contained in $U_n$ by backward invariance of the Fatou set.
Conversely, $f(U_n)\subset U_{n+1}$, so connectedness puts all of $U_n$ in that inverse component.
\end{proof}

Choose repelling markings $P_j$ as in Definition~\ref{def:marking}.
If $A\subset\sphere$ is closed and avoids $U_n$, Lemma~\ref{core:kernel}, applied to finite subsets of $A\cup\Jul(f)$ and then by monotonicity, gives
\begin{equation}\label{eq:recovery}
 \rho_{\sphere\setminus(A\cup P_j)}\uparrow\rho_{U_n}
 \quad\hbox{locally uniformly on }U_n.
\end{equation}
Indeed, the component of $\sphere\setminus(A\cup\Jul(f))$ containing $U_n$ is exactly $U_n$.
The same assertion holds for plane complements by adjoining infinity to $A$.

\section{Hyperbolic area and covering framework}\label{sec:area-framework}\label{sec:covering-loss}

The common mechanism is finite loss of hyperbolic area along a covering orbit.
Finite punctures or compact holes have bounded area cost, while a thick wandering tail supplies disjoint packets of arbitrarily large intrinsic area.
In the thin case annular coverings and their shrinking Jordan sides replace the large-area packets.
We record the quantitative tools before applying them to any of the four questions.

\subsection{Area cost of punctures and compact holes}
\begin{lemma}\label{core:puncture}
If $G\subset\sphere$ is hyperbolic and $E\subset G$ is finite, then
\[
 0\le\int_{G\setminus E}
       (\dd\alpha_{G\setminus E}-\dd\alpha_G)\le |E|.
\]
\end{lemma}
\begin{proof}
Choose increasing finite sets $F_j\subset\sphere\setminus G$, containing a fixed triple, whose union is dense in $\sphere\setminus G$.
Set $G_j=\sphere\setminus F_j$ and $H_j=G_j\setminus E$.
Gauss--Bonnet gives
\[
 \int_{H_j}(\dd\alpha_{H_j}-\dd\alpha_{G_j})=|E|.
\]
On $G\setminus E$, Lemma~\ref{core:kernel} gives pointwise convergence of both Poincar\'e densities to their intrinsic limits.
The differences of the area densities are nonnegative by domain monotonicity, so Fatou's lemma proves the bound, also when $G$ is disconnected.
\end{proof}

\begin{lemma}\label{core:holes}
Let $F_j\subset\sphere$ be finite sets containing a fixed three-point set $T$.
Let $K$ and $N$ be finite unions of pairwise disjoint closed smooth Jordan disks, with $K\Subset\operatorname{Int}N$.
Suppose a fixed neighborhood of $N$ avoids every $F_j$.
For $G_j=\sphere\setminus F_j$ and $H_j=G_j\setminus K$, there is a constant $C<\infty$, independent of $j$, such that
\begin{equation}\label{core:holes-bound}
 \int_{G_j\setminus N}
       (\dd\alpha_{H_j}-\dd\alpha_{G_j})\le C.
\end{equation}
\end{lemma}
\begin{proof}
The case $K=\varnothing$ is immediate.
Otherwise put $u_j=\log(\rho_{H_j}/\rho_{G_j})\ge0$.
The curvature equation gives $\Delta u_j=4(\rho_{H_j}^2-\rho_{G_j}^2)\ge0$.
Set $G_*=\sphere\setminus T$ and $\varepsilon=d_{G_*}(K,G_*\setminus\operatorname{Int}N)>0$.
For $z\in G_j\setminus N$, domain monotonicity gives $d_{G_j}(z,K)\ge d_{G_*}(z,K)\ge\varepsilon$.
A universal covering $\phi:\D\to G_j$ with $\phi(0)=z$ maps $\{|w|<\tanh\varepsilon\}$ into $H_j$.
Schwarz--Pick therefore gives $u_j(z)\le\log\coth\varepsilon$ (cf.\ \cite[Proposition~3.4]{MBRG}).
Thus $u_j$ is bounded above near each point of $F_j$ and extends to a nonnegative subharmonic function on $\sphere\setminus K$.

On the fixed open set $L=\operatorname{Int}N\setminus K$, the inclusions $L\subset H_j$ and $G_j\subset G_*$ give
\begin{equation}\label{core:holes-comparison}
 0\le u_j\le u_*:=\log\frac{\rho_L}{\rho_{G_*}}.
\end{equation}
Choose a smooth conformal metric $\sigma$ on the sphere and a smooth cutoff $\chi\in C_c^\infty(\sphere\setminus K)$ such that $0\le\chi\le1$, $\chi=1$ near $\sphere\setminus N$, and $\operatorname{supp}\Delta_\sigma\chi\Subset L$.
Here $\Delta_\sigma$ and $\dd A_\sigma$ denote the Laplacian and area element of $\sigma$.
The extended distributional Laplacian of $u_j$ is positive, so integration by parts gives
\begin{align}\label{core:holes-cutoff}
 \int_{G_j\setminus N}(\dd\alpha_{H_j}-\dd\alpha_{G_j})
 &\le\frac1{2\pi}\langle\Delta_\sigma u_j,\chi\rangle
 =\frac1{2\pi}\int u_j\Delta_\sigma\chi\,\dd A_\sigma\notag\\
 &\le\frac1{2\pi}
 \sup_{\operatorname{supp}\Delta_\sigma\chi}u_*
 \int|\Delta_\sigma\chi|\,\dd A_\sigma.
\end{align}
The last quantity is finite and independent of $j$.
\end{proof}

\subsection{The limiting covering surface}
Suppose that $(U_n)$ is a singular-free wandering tail.
Compatible universal covers $p_n:\D\to U_n$ can be chosen with $p_{n+1}=f p_n$.
Their deck groups satisfy $\Gamma_n\subset\Gamma_{n+1}$.
We call the tail \emph{thick} if $\Gamma=\bigcup_n\Gamma_n$ is discrete and \emph{thin} otherwise.
These are the thick and thin alternatives for the covering system: its transition maps are full coverings and hence are eventually isometric in the sense of Ferreira--Rempe \cite[Lemma~4.3 and Corollary~4.4]{FerreiraRempe}.
In the thin case their classification gives an eventual tail of doubly connected components, with finite covering degrees $e_n$.
Such a Fatou component cannot be a punctured disk: a univalent parametrization at its puncture would give an isolated Julia point.
The annuli therefore have finite modulus.
The cumulative degree products tend to infinity (otherwise the injectivity radius would eventually be a positive constant), and
\begin{equation}\label{eq:annular-degree}
 \operatorname{mod}U_n=(e_0\cdots e_{n-1})\operatorname{mod}U_0
 \longrightarrow\infty
\end{equation}
after reindexing.
Only the covering and metric conclusions of this classification enter the argument.

\begin{lemma}\label{lem:large-disks}
In the thick case, for every $M<\infty$ there are finitely many pairwise disjoint disks $D_i\Subset U_0$ and compact disks $Q_i\Subset D_i$ such that, with $Q=\bigcup_iQ_i$,
\[
 \alpha_{U_0}(Q)>M.
\]
Every $f^n$ is injective on $\bigcup_iD_i$, including between distinct disks, and their images at different times are disjoint.
\end{lemma}
\begin{proof}
Let $V=\D/\Gamma$.
Subgroup inclusion gives compatible coverings $\pi_n:U_n\to V$ with $\pi_{n+1}f=\pi_n$.
The complete area of $V$ is infinite.
Otherwise its finite-area Fuchsian group would be finitely generated, so $\Gamma=\Gamma_N$ for some $N$.
Then $V=U_N$ would be a finite-area planar hyperbolic surface, hence a sphere with finitely many punctures.
Its univalent inclusion into $\sphere$ extends across those punctures to a M\"obius map.
This would give a finite complement, whereas the complement contains the infinite perfect Julia set.

Choose finitely many disjoint compact coordinate disks in $V$ of total area greater than $M$, with slightly larger disjoint disk neighborhoods.
Lift one copy of each neighborhood through $\pi_0$.
Covering invariance gives the stated area.
Equality of two image points under $f^n$ would, after projection by $\pi_n$, identify points in the selected disjoint coordinate disks.
Thus $f^n$ is jointly injective on the selected lifts.
Different times lie in distinct Fatou components.
\end{proof}

\subsection{Positive area and limiting metrics}
\begin{lemma}\label{lem:finite-loss}
Let $X\subset\sphere$ be hyperbolic and omit a fixed triple, let $E\subset X$ be finite, and put $Y=X\setminus E$.
Suppose $V=f^{-1}(Y)\subset X$ and $f:V\to Y$ is a covering.
Let $D_i$ be finitely many disks in a Fatou component such that every $f^n$ is injective on $\bigcup_iD_i$, and all these images, at distinct times, are disjoint subsets of $Y$.
If compact disks $Q_i\Subset D_i$ satisfy $\alpha_X(Q)>|E|$, where $Q=\bigcup_iQ_i$, then:
\begin{enumerate}
\item $\alpha_X(f^n(Q))\ge\alpha_X(Q)-|E|>0$ for every $n$;
\item on at least one $D_i$, the pullback densities $(f^n)^*d\alpha_X$
converge smoothly on compact subsets to a positive density;
\item at some $z\in Q$, writing $z_n=f^n(z)$, the quantities
\[
 \epsilon_n=\frac{\rho_Y(z_n)^2}{\rho_X(z_n)^2}-1,
 \qquad
 \delta_n=\frac{\rho_V(z_n)^2}{\rho_X(z_n)^2}-1
\]
satisfy $\sum_n(\epsilon_n+\delta_n)<\infty$.
\end{enumerate}
\end{lemma}
\begin{proof}
Put $\mu=\alpha_X$ and $\nu=\alpha_Y-\alpha_X$.
Lemma~\ref{core:puncture} gives $\nu(Y)\le |E|$, while covering invariance and $V\subset X$ give $f^*(\mu+\nu)\ge\mu$.
On each source disk write $(f^n)^*d\mu=\sigma_n\,dx\,dy$ and $(f^n)^*d\nu=t_n\,dx\,dy$.
Then
\begin{equation}\label{eq:area-compare}
 \sigma_{n+1}+t_{n+1}\ge\sigma_n,
 \qquad
 \sum_{n\ge0}\int_Qt_{n+1}
 =\sum_{n\ge0}\nu(f^{n+1}(Q))\le |E|.
\end{equation}
The same bound holds on arbitrary compact subsets of the union of the source disks.
Schwarz--Pick bounds $\sigma_n$ above by their complete area densities.
Summable negative variation and telescoping therefore give summable total variation on each compact subset.
Consequently $\sigma_n$ has an almost-everywhere and $L^1_{\rm loc}$ limit $\sigma$, and
\begin{equation}\label{eq:packetfloor}
 \int_Q\sigma\ge\alpha_X(Q)-|E|>0.
\end{equation}
The same telescoping inequality at each finite time proves the first claim.

On a disk where the limiting mass is positive, lift $f^n$ through a universal disk of its component of $X$, normalizing the lift $h_n$ to vanish at a fixed source point.
Its density is $(2/\pi)|h_n'|^2/(1-|h_n|^2)^2$.
A constant normal limit would be zero and would force the mass to vanish.
Every other limit takes values in $\D$ and has nonzero derivative by Hurwitz.
Its density agrees almost everywhere with the unique $L^1$ limit.
Normality and this uniqueness imply smooth convergence of the densities to a positive density on that disk.

Choose $z\in Q$ with $\sigma(z)>0$ and $\sum_nt_n(z)<\infty$.
At this point the exact identities are
\[
 \epsilon_n=t_n/\sigma_n,
 \qquad
 1+\delta_n=(\sigma_{n+1}/\sigma_n)(1+\epsilon_{n+1}).
\]
The first gives $\sum\epsilon_n<\infty$.
Taking logarithms in the second and telescoping gives $\sum\log(1+\delta_n)<\infty$.
Since $\delta_n\ge0$, this is equivalent to $\sum\delta_n<\infty$.
\end{proof}

\begin{proposition}\label{lem:carrier-discrete}\label{prop:carrier}
Suppose $A\subset\Fat(f)$ is closed in $\C$, $f(A)\subset A$, every Fatou component meeting $A$ is preperiodic, and $E=S_\C\setminus A$ is finite.
If a wandering tail is thick, there are a disk $D\Subset U_0$ and a finite forward invariant repelling marking $P$ such that, for $X=\C\setminus(A\cup P)$,
\begin{equation}\label{eq:derivativeone}
 f^n|_D\longrightarrow\infty,
 \qquad
 \frac{|f'(f^n(z))|\rho_X(f^{n+1}(z))}{\rho_X(f^n(z))}
 \longrightarrow1
\end{equation}
locally uniformly on $D$.
This applies in particular to an eventually simply connected wandering tail.
\end{proposition}
\begin{proof}
A wandering tail avoids $A$ and, after discarding finitely many components, avoids $E$.
Lemma~\ref{q4:transitions} supplies the covering restrictions.
Take the disks of Lemma~\ref{lem:large-disks} with intrinsic area greater than $|E|+2$, and use \eqref{eq:recovery} to choose $P$ with $\alpha_X(Q)>|E|+1$.
Forward invariance gives $f^{-1}(X\setminus E)\subset X$.
Lemma~\ref{lem:finite-loss} now supplies a positive smooth limiting density on one disk.
Ratios of consecutive densities give the derivative limit in \eqref{eq:derivativeone}.

All normal limits on a wandering component are constant: a nonconstant limit and Rouch\'e's theorem would put one fixed target disk in infinitely many disjoint forward components.
A finite constant limit belongs to $\Jul(f)$.
Near such a point $a$, the closed set $A\subset\Fat(f)$ is absent and the finite marking has at most the single point $a$.
A smaller neighborhood has finite $\alpha_X$-area, including when $a$ is a cusp.
A compact disk with positive limiting pullback mass would have infinitely many disjoint forward images in this neighborhood, each of area bounded below.
This is impossible.
Thus the only normal limit is infinity.
In the simply connected case all deck groups are trivial, so the same argument applies.
\end{proof}

\subsection{Summable metric defects}
\begin{lemma}\label{q4:summable-defects}
Let $X\subset\Chat$ be hyperbolic and omit a fixed triple.
Let $E\subset X$ be finite, put $Y=X\setminus E$, and suppose $V=f^{-1}(Y)\subset X$ and $f:V\to Y$ is a holomorphic covering.
Let $z_n=f^n(z_0)\in V\cap Y$, and set
\[
 \epsilon_n=\frac{\rho_Y(z_n)^2}{\rho_X(z_n)^2}-1,
 \qquad
 \delta_n=\frac{\rho_V(z_n)^2}{\rho_X(z_n)^2}-1.
\]
If $\sum_n(\epsilon_n+\delta_n)<\infty$, then for every $R<\infty$ and all sufficiently large $n$, the ball $B_X(z_n,R)$ lies in $\Fat(f)$.
In particular, if $z_n$ lies in the Fatou component $U_n$, then
\[
 d_X(z_n,X\setminus U_n)\longrightarrow\infty.
\]
\end{lemma}
\begin{proof}
The disk estimate \eqref{q8new:near-id} gives, for $0<r<1$,
\begin{equation}\label{q4:near-identity-disk}
 \sup_{|\zeta|\le r}|h(\zeta)-\zeta|
 \le\frac{r(1+r)}{1-r}(1-a).
\end{equation}

Choose universal covers $\phi_n:\D\to X$ centered at $z_n$, with arguments chosen recursively.
Once $\phi_n$ is chosen, take a universal cover $v_n:\D\to V$ centered at $z_n$ whose lift $a_n$ through $\phi_n$ has positive derivative at zero.
Then $f v_n$ is a universal cover of the corresponding component of $Y$.
Choose the argument of $\phi_{n+1}$ so its lift $b_{n+1}$ also has positive derivative at zero.
Thus $a_n,b_{n+1}:\D\to\D$ fix zero and
\[
 \phi_n a_n=v_n,
 \qquad f\phi_n a_n=\phi_{n+1}b_{n+1},
\]
\[
 a_n'(0)=(1+\delta_n)^{-1/2},
 \qquad b_{n+1}'(0)=(1+\epsilon_{n+1})^{-1/2}.
\]
These lifts need not be globally injective.

Fix $0<r_0<r_1<r_2<1$.
Equation~\eqref{q4:near-identity-disk} gives
\[
 \|a_n-\mathrm{id}\|_{\overline{\D}_{r_2}}\le C\delta_n,
 \qquad
 \|b_{n+1}-\mathrm{id}\|_{\overline{\D}_{r_2}}
 \le C\epsilon_{n+1}.
\]
For all sufficiently large $n$, Rouch\'e's theorem on $|w|=r_2$ shows that $a_n(w)=\zeta$ has precisely one solution in $\D_{r_2}$, counting multiplicity, for every $\zeta\in\D_{r_1}$.
Its inverse $A_n:\D_{r_1}\to\D_{r_2}$ is therefore holomorphic and $\|A_n-\mathrm{id}\|_{\D_{r_1}}\le C\delta_n$.
Consequently $G_n=b_{n+1}A_n$ satisfies
\begin{equation}\label{q4:coordinate-increment}
 f\phi_n=\phi_{n+1}G_n\quad\hbox{on }\D_{r_1},
 \qquad
 \|G_n-\mathrm{id}\|_{\D_{r_1}}
 \le C(\delta_n+\epsilon_{n+1}).
\end{equation}
Also $\phi_n(\D_{r_1})\subset V$, since $\phi_n=v_n A_n$ there.

Choose $N$ so the sum of the right-hand errors for $n\ge N$ is less than $r_1-r_0$.
Every future composition of the $G_n$ is then defined on $\D_{r_0}$, with all intermediate images in $\D_{r_1}$.
Equation~\eqref{q4:coordinate-increment} implies that all iterates of $f$ are holomorphic on $\phi_N(\D_{r_0})$ and omit the fixed triple outside $X$.
Every intermediate image lies in $V\subset\C$, so no iterate takes the value infinity; a pole at any stage is also excluded by the existence of the next finite image.
Montel's theorem gives normality on $\phi_N(\D_{r_0})$; local biholomorphic inverse charts of $\phi_N$ justify passage from the disk to its image.

For curvature $-4$, the covering image $\phi_N(\D_{r_0})$ is precisely $B_X(z_N,\operatorname{artanh}r_0)$, by path lifting for the complete metric.
The argument holds for every sufficiently large $N$, and $r_0$ may be arbitrarily close to one.
Connectedness of each ball then places it in $U_N$ whenever $z_N\in U_N$.
\end{proof}

\begin{lemma}
\label{q11geo:remove-marks}
Let $A\subset\Chat$ be closed with at least three points, let $Z=\Chat\setminus A$, and let $P\subset Z$ be finite.
Put $X=Z\setminus P$.
Suppose $w_n\in W_n\subset X$, where $W_n$ are connected open sets, every spherical cluster point of $(w_n)$ belongs to $A$, and
\[
 d_X(w_n,X\setminus W_n)\longrightarrow\infty.
\]
Then $d_Z(w_n,Z\setminus W_n)\to\infty$.
The points $w_n$ and the sets $W_n$ need not belong to one fixed component of $Z$.
\end{lemma}
\begin{proof}
For every fixed $L<\infty$, the closed complete hyperbolic ball of radius $L$ about each $p\in P$ is compact in the component of $Z$ containing $p$.
The union of these finitely many balls is therefore compact in $Z$.
Since every spherical cluster point of $w_n$ lies in $A$, it follows that
\[
 \ell_n:=d_Z(w_n,P)\longrightarrow\infty.
\]
Distances to marks in other components are understood as infinity.
Fix $R<\infty$.
Each point of $B_Z(w_n,R)$ has $Z$-distance at least $\ell_n-R$ from $P$.
A centered universal-cover disk of any smaller radius avoids $P$, so Schwarz--Pick on that disk gives
\[
 1\le\frac{\rho_X}{\rho_Z}
 \le\coth(\ell_n-R)=1+o(1)
 \quad\hbox{on }B_Z(w_n,R),
\]
with the evident interpretation if $\ell_n=\infty$.
Every $Z$-path of length at most $R$ from $w_n$ thus lies in $X$ and has $X$-length at most $(1+o(1))R$.
The assumed separation in $X$ places the whole $Z$-ball inside $W_n$ for large $n$.
Since $R$ was arbitrary, the conclusion follows.
\end{proof}

For a finite $T\subset S$, let $\mathcal O_T$ be the set of all defined forward images of points of $S\setminus T$, including their zeroth images and any terminal value infinity.
Set
\begin{equation}\label{eq:deleted-postsingular}
 A_T=\overline{\mathcal O_T}^{\,\sphere}.
\end{equation}
Then $f(A_T\cap\C)\subset A_T$ and $S\setminus A_T\subset T$.
The set $T$ is omitted only as a set of starting singular generators: an orbit point of a generator in $T$ may still lie in $A_T$ if it is reached or approximated by orbits of retained generators.
A prime denotes the spherical derived set.

\subsection{Annular wandering tails}
\label{q11r:section}

The inverse continuation over a disk with one singular value follows the argument of Baker \cite[Theorem~2, proof]{Baker2002}; the exclusion of its exterior inverse side uses the periodic-point argument of \cite[Theorem~1,
proof]{Baker2002}.
We record both steps because the finite-deletion conclusion requires their iteration at every large time.

For the annular alternative in Section~\ref{sec:covering-loss}, let $(U_n)$ be a singular-free wandering tail of finite-modulus annuli with finite-degree covering transitions.
Keep $\mathcal O_T$ and $A_T$ as in \eqref{eq:deleted-postsingular}.

Choose an annular uniformization of $U_0$ and let $\gamma_0$ be any circle concentric in this coordinate.
Write $\gamma_n$ for the underlying curve of $f^n(\gamma_0)$.
Finite coverings of annuli have the form $\zeta\mapsto c\zeta^{\pm d}$ in annular coordinates.
Consequently every $\gamma_n$ is again a coordinate circle, in particular a Jordan curve, with its parametrization traversed according to the cumulative degree.
This choice includes the closed hyperbolic geodesic but is not restricted to it; it does not allow an arbitrary analytic essential curve.
All normal limits on a wandering Fatou component are constant.
Hence
\begin{equation}\label{q11r:shrink}
 \operatorname{diam}_{\#}\gamma_n\longrightarrow0.
\end{equation}
For all sufficiently large $n$, there is a unique Jordan side $B_n$ of $\gamma_n$ with small spherical diameter, and $\operatorname{diam}_{\#}B_n\to0$.
To see this last statement, place $\gamma_n$ in a spherical ball of radius $\operatorname{diam}_{\#}\gamma_n$.
The complementary large ball is connected, so one Jordan side lies in the small ball.

\begin{lemma}
\label{q11r:one-singular-disk}
Let $D\subset\sphere$ be a Jordan disk such that
$\partial D\cap S=\varnothing$ and $|D\cap S|\leq1$.
Every component of $f^{-1}(D)$ is simply connected.
\end{lemma}
\begin{proof}
If $D\cap S=\varnothing$, each inverse component covers the disk and is simply connected.
Otherwise write $D\cap S=\{s\}$, let $C$ be an inverse component, and put $C^*=C\setminus f^{-1}(s)$.
Removing a discrete set does not disconnect a plane domain.
Thus $C^*$ is connected, and the inverse covering theorem gives a connected unbranched covering
\[
 f:C^*\longrightarrow D\setminus\{s\}.
\]
The covering is either universal or a finite-degree punctured-disk cover.
If $C$ contains a point over $s$, a small loop about that point has nonzero winding in $C^*$, excluding the universal case.
There can be at most one such point: loops around two distinct removed finite points give two independent classes in first homology, whereas a punctured disk has cyclic first homology.
Filling the sole puncture then gives a disk.

Suppose instead there is no point of $C$ over $s$.
The universal case is already simply connected.
In the finite-degree case choose a univalent parametrization $j:\D^*\to C^*$ and a target coordinate $\phi$ centered at $s$, so that $\phi(f(j(\zeta)))=\zeta^d$.
Univalence excludes an essential singularity of $j$ at zero.
A finite removable value would give, by continuity, a point of the same inverse component over $s$, a contradiction.
If $j$ has a pole at zero, its image contains a full neighborhood of infinity and $f$ tends to $s$ throughout that neighborhood.
The function $f$ would therefore extend meromorphically at infinity, contrary to transcendence.
\end{proof}

\Needspace{6\baselineskip}
\begin{lemma}
\label{q11r:nonreset-step}
There is $N_0$ such that, for $n\geq N_0$, if $|B_{n+1}\cap S|\leq1$, then
\begin{equation}\label{q11r:propagate}
 B_n\subset\C,\qquad f(B_n)\subset B_{n+1}.
\end{equation}
\end{lemma}
\begin{proof}
An annular collar of $\gamma_n$ has one full collar side mapping into $B_{n+1}$.
Let $C_n$ be the inverse component containing that side.
Since $\gamma_n$ maps into $\partial B_{n+1}$, the domain $C_n$ is
disjoint from $\gamma_n$ and has this whole curve in its boundary.
Lemma~\ref{q11r:one-singular-disk} makes $C_n$ simply connected.
It follows that $C_n$ is the entire Jordan side containing its collar: otherwise an omitted point in that side would be separated from the opposite closed side by the full collar, disconnecting the spherical complement of $C_n$.

Fix three finite regular points $a_1,a_2,a_3$ with pairwise distinct images under $f$.
For large $n$, the disk $B_n$ contains at most one of these points, by its shrinking diameter.
None lies on $\gamma_n$ after finitely many indices, since the $U_n$ are pairwise disjoint.
If $C_n$ were the large side, it would contain at least two of them, and their two distinct images would lie in the shrinking disk $B_{n+1}$.
This is impossible for large $n$.
Hence $C_n=B_n$.
As inverse components are subsets of $\C$, this also proves that $\infty\notin B_n$.
\end{proof}

\begin{theorem}
\label{q11r:reset-propagation}
Call $n$ a reset time if $|B_n\cap S|\geq2$.
There are infinitely many reset times.
For every finite $T\subset S$, there exist $N_T$ and points
\begin{equation}\label{q11r:witnesses}
 q_n\in B_n\cap \mathcal O_T\qquad(n\geq N_T).
\end{equation}
Consequently every constant normal limit $a$ on $U_0$ belongs to $A_T$, for every finite $T\subset S$.
Moreover at least one such limit belongs to $S'$, and thus to $(S\setminus T)'$ for every finite $T$.
\end{theorem}
\begin{proof}
If there were only finitely many reset times, Lemma~\ref{q11r:nonreset-step} would give $f(B_n)\subset B_{n+1}$ for all sufficiently large $n$.
Every iterate on a fixed late $B_N$ would then be defined, with $f^k(B_N)\subset B_{N+k}$.
Their image diameters tend to zero, so the iterates form a normal family on $B_N$: any sequence has a subsequence of constant spherical limiting values, and the shrinking image diameter makes the convergence uniform on $B_N$.
Thus $B_N\subset \Fat(f)$.
Its collar intersects $U_N$, so maximality of the Fatou component gives $B_N\subset U_N$, filling the essential curve $\gamma_N$ in an annulus.
This contradiction proves the existence of infinitely many resets.

Since $T$ is finite and the diameters of $B_n$ tend to zero, every sufficiently late $B_n$ contains at most one point of $T$.
Start at a reset time $N_T$ later than this threshold and $N_0+1$.
At any reset time $n\geq N_T$, choose
\[
 q_n\in (S\setminus T)\cap B_n,
\]
which is possible because $B_n$ contains at least two points of $S$.
At a nonreset time $n>N_T$, define $q_n=f(q_{n-1})$.
This is legitimate: Lemma~\ref{q11r:nonreset-step} gives $B_{n-1}\subset\C$ and $f(B_{n-1})\subset B_n$, so $q_{n-1}$ is finite and $q_n\in B_n$.
In particular, if a previous witness equals infinity, the next time must be a reset and the construction chooses a new initial generator; it never attempts to evaluate $f(\infty)$.
The induction proves \eqref{q11r:witnesses}.

If $f^{n_j}$ converges locally uniformly on $U_0$ to a constant $a$, then $\gamma_{n_j}$ and $B_{n_j}$ converge spherically to $a$.
The points $q_{n_j}\in \mathcal O_T\cap B_{n_j}$ consequently tend to $a$, proving $a\in A_T$.

Finally choose an infinite sequence of reset times and a normal subsequence with constant limit $a$.
Each corresponding $B_n$ contains two distinct points of $S$, both tending to $a$; thus $a\in S'$.
Removing finitely many points does not change a derived set, so $S'=(S\setminus T)'$ for every finite $T$.
\end{proof}

\subsection{A round-ring estimate}
\label{q11e:section}

\begin{lemma}
\label{q11e:round-ring}
Let $H_j$ be bounded plane annuli separating $0$ from $\infty$, with bounded and unbounded complementary continua $K_j,L_j$.
Set
\[
 R_j=\max_{w\in K_j}|w|,
 \qquad r_j=\min_{w\in L_j\cap\C}|w|.
\]
If $\operatorname{mod}H_j\to\infty$, then $r_j/R_j\to\infty$.
Consequently $\{R_j<|w|<r_j\}\subset H_j$ for large $j$.
The same assertion holds with any varying finite center in the bounded complementary component in place of zero.
\end{lemma}
\begin{proof}
If the ratio were bounded by $M$ along a subsequence, choose an inner point of modulus $R_j>0$ and normalize by a similarity so that the inner continuum contains $0,-1$.
The outer continuum then contains $\infty$ and a point of modulus at most $M$.
Both continua have a uniformly positive lower bound $\eta$ for their round spherical diameters.

Every essential Jordan curve in the normalized ring has spherical length at least $\eta$.
Indeed, a curve of length $\ell<\pi$ lies in a spherical ball of radius $\ell/2$ about one of its points, by using the shorter subarc to every other point.
One of its two sides lies in that ball, so the continuum on that side has diameter at most $\ell$.
For $\ell\ge\pi$ the claim is automatic.
Testing extremal length of essential closed curves with the round spherical metric, of total area $4\pi$, gives
\[
 \operatorname{mod}H_j\le4\pi/\eta^2,
\]
because that extremal length is the reciprocal of ring modulus.
This contradicts the hypothesis.
The round-ring inclusion follows from the definitions of $R_j,r_j$, and translation permits varying centers.
\end{proof}

\section{Hyperbolic Area Proofs of No-Wandering Theorems}\label{sec:q9}

The two parts of Theorem~\ref{thm:A-q9} correspond to Bergweiler's historical Theorems~11 and~12, respectively.
We prove the rational case first, and then treat each transcendental class under its exact stated assumptions.

\subsection{Rational maps}\label{sec:rational}

The degree-at-most-one case is elementary: a constant map has a single Fatou component, and a M\"obius map is conjugate to a multiplication or translation, with invariant Fatou components.
For the degree-at-least-two case, we give the proof.

For a rational map $R$, write $\operatorname{CV}(R)$ for its critical values on $\sphere$.
All domains and closures in this section are spherical.

\begin{proof}[Proof of Theorem~\ref{core:sullivan}\textup{(i)}]
Suppose that $R$ has a wandering Fatou component.
Baker's reduction \cite[Lemma~5.34]{McMullen} and the finiteness of $\operatorname{CV}(R)$ give a simply connected tail $(V_n)_{n\ge0}$ avoiding $\operatorname{CV}(R)$.
Complete invariance of the Fatou set and properness of $R$ show that $R:V_n\to V_{n+1}$ is a covering, hence a conformal isomorphism.
Thus $R:\Omega\to\Omega\setminus V_0$ is a conformal isomorphism, where $\Omega=\bigsqcup_{n\ge0}V_n$.

Choose finite unions $Q_j\subset\Jul(R)$ of repelling periodic orbits such that $Q_j\subseteq Q_{j+1}$, $|Q_1|\ge3$, and $\bigcup_jQ_j$ is dense in $\Jul(R)$ \cite[Theorem~4]{Bergweiler}.
Set $A_j=\sphere\setminus Q_j$ and $B_j=A_j\setminus\operatorname{CV}(R)$, and define measures by
\[
 \dd\beta_j=\dd\alpha_{A_j}\quad\text{on }A_j,\qquad
 \dd\delta_j=\dd\alpha_{B_j}-\dd\alpha_{A_j}\ge0
 \quad\text{on }B_j.
\]
Gauss--Bonnet gives
\begin{equation}\label{eq:rational-cost}
 \delta_j(B_j)=|\operatorname{CV}(R)\setminus Q_j|
 \le2\deg R-2.
\end{equation}
The periodic-orbit construction gives $R(Q_j)=Q_j$, hence $R^{-1}(B_j)\subset A_j$.
Since $B_j$ avoids $\operatorname{CV}(R)$, $R:R^{-1}(B_j)\to B_j$ is an unramified covering.
Covering invariance and domain monotonicity from Section~\ref{sec:area} give, respectively,
\[
 R^*(\dd\beta_j+\dd\delta_j)
 =\dd\alpha_{R^{-1}(B_j)}\ge\dd\beta_j
 \quad\text{on }R^{-1}(B_j).
\]
Since $\Omega\subset\Fat(R)$ avoids $\operatorname{CV}(R)$ and $Q_j\subset\Jul(R)$, we have
\[
 R(\Omega)=\Omega\setminus V_0\subset\Omega\subset B_j.
\]
Thus $\Omega\subset R^{-1}(B_j)$.
Integrating over $\Omega$ and changing variables under $R:\Omega\to\Omega\setminus V_0$ yields
\begin{align*}
 \beta_j(V_0)+\beta_j(\Omega\setminus V_0)
 &=\beta_j(\Omega)\\
 &\le\beta_j(\Omega\setminus V_0)+\delta_j(\Omega\setminus V_0).
\end{align*}
As $\beta_j(A_j)=|Q_j|-2<\infty$, we may subtract the finite term $\beta_j(\Omega\setminus V_0)$.
Nonnegativity of $\delta_j$ and \eqref{eq:rational-cost} then give
\[
 \beta_j(V_0)\le\delta_j(\Omega\setminus V_0)
 \le\delta_j(B_j)\le2\deg R-2.
\]
On the other hand, Lemma~\ref{core:kernel} and monotone convergence give
\[
 \beta_j(V_0)\uparrow\alpha_{V_0}(V_0)=\infty,
\]
a contradiction.
\end{proof}

\Needspace{5\baselineskip}
\subsection{Finite inverse singular sets and thin tails}
\subsubsection{Finite singular sets}\label{sec:finite-type}
\begin{theorem}\label{thm:finite-type}
A transcendental meromorphic function with finitely many inverse singular values has no wandering Fatou component.
\end{theorem}
\begin{proof}
A hypothetical wandering orbit eventually avoids the finite singular set, so its full transitions are coverings.
If the tail is thin, the annular reduction in Section~\ref{sec:covering-loss} and Theorem~\ref{q11r:reset-propagation} give infinitely many shrinking Jordan sides containing two distinct singular values.
This is impossible for a finite set.

In the thick case use Lemma~\ref{lem:large-disks} to choose $Q$ with $\alpha_{U_0}(Q)>|S|+2$, and choose a finite marking $P_j$ with $\alpha_{X_j}(Q)>|S|+1$, where $X_j=\C\setminus P_j$.
Set $Y_j=X_j\setminus S_\C$.
Its inverse is contained in $X_j$ and covers $Y_j$.
Lemma~\ref{lem:finite-loss} gives $\alpha_{X_j}(f^n(Q))>1$ for all $n$.
The sets $f^n(Q)$ are pairwise disjoint, whereas $X_j$ has finite area.
This contradiction proves the theorem.
\end{proof}

\subsubsection{A criterion excluding thin tails}
\begin{lemma}\label{lem:bounded-basin-fixed}
Let $f$ be meromorphic on $\C$, and let $B$ be a bounded immediate
basin of an attracting fixed point $\zeta$. Then $\partial B$ contains a
finite fixed point of $f$ that is not attracting.
\end{lemma}
\begin{proof}
The map is holomorphic on a neighborhood of $\overline B$: a pole on $\overline B$ would contradict $f(B)\subset B$ and boundedness of $B$.
The restriction $f:B\to B$ is proper, since a boundary point whose image belongs to $B$ is itself in the attracting basin and hence, by maximality, cannot be a boundary point.
Its degree is finite and at least two; a one-sheeted proper self-map would be an automorphism fixing $\zeta$, which cannot have an attracting multiplier.

There are finitely many critical points in $B$, and their forward orbits accumulate only at $\zeta$.
Choose $z_0,z_1\in B\setminus\{\zeta\}$ with $f(z_1)=z_0$, and a simple arc $\gamma$ joining them, avoiding the closure of the forward critical orbits.
The choices can also avoid all forward critical images.
Take a simply connected neighborhood $V\Subset B$ of $\gamma$ disjoint from this closed set.
Proper covering theory successively constructs univalent inverse branches $h_n:V\to B$ of $f^n$ such that $h_0=\operatorname{id}$,
\[
 f\circ h_{n+1}=h_n,
 \qquad h_{n+1}(z_0)=h_n(z_1).
\]
Indeed, no $h_n(V)$ contains a critical value of $f$, since its image under $f^n$ would be a forward critical image in $V$.

The maps $h_n$ are bounded and form a normal family.
A nonconstant subsequential limit has image in $B$, by Hurwitz's theorem applied to each omitted point.
This is impossible: on a small compact neighborhood of a limit point in $B$, the iterates $f^n$ converge uniformly to $\zeta$, whereas $f^n(h_n(z_0))=z_0\ne\zeta$.
Thus every subsequential limit is constant, and
\[
 |h_n(z_1)-h_n(z_0)|\longrightarrow0.
\]
Put $z_n=h_n(z_0)$. Every cluster point of $z_n$ lies in $\partial B$,
by the same attraction argument, and $f(z_{n+1})=z_n$ with $|z_{n+1}-z_n|\to0$.
A cluster point $p$ therefore satisfies $f(p)=p$.
It cannot be attracting, since an attracting fixed point is an interior point of its immediate basin.
\end{proof}

\begin{theorem}\label{thm:thin-pole-gap}
Let $f:\C\to\sphere$ be transcendental meromorphic.
Assume the following.
\begin{enumerate}
\item The finite inverse singular values form a locally finite subset of
$\C$.
\item Outside a finite set $E_s$, every finite inverse singular value is
an attracting fixed point.
\item Every non-attracting finite fixed point belongs to a finite set $E_f$.
\item There are $C>1$ and $r_0>0$ such that each round annulus
$\{r<|z|<Cr\}$, $r>r_0$, contains a pole or a finite fixed point of $f$.
\end{enumerate}
Then $f$ has no thin singular-free wandering tail.
No hypothesis that the tail escapes is required.
\end{theorem}
\begin{proof}
The covering-group classification reduces such a tail to annuli $U_n$ with finite-degree covering transitions.
Write $D_n$ for the cumulative degree.
Thinness implies $D_n\to\infty$: otherwise the transitions would eventually have degree one, so the increasing deck groups would stabilize and their union would be discrete.

Choose a relatively compact concentric subannulus $A_0\Subset U_0$ in an annular uniformization, and choose a concentric essential analytic Jordan curve $\gamma_0\subset A_0$.
Set $A_n=f^n(A_0)$ and let $\gamma_n$ be the underlying Jordan curve of $f^n(\gamma_0)$.
Power-map uniformizations of annular coverings give
\[
 \operatorname{mod}A_n=D_n\operatorname{mod}A_0\longrightarrow\infty.
\]
Normality on $U_0$ and pairwise disjointness of the $U_n$ imply that every subsequential normal limit of $f^n|_{U_0}$ is constant.
Consequently $\operatorname{diam}_{\#}\gamma_n\to0$.
Let $B_n$ be the complementary Jordan disk of $\gamma_n$ with small spherical diameter.
Then $\operatorname{diam}_{\#}B_n\to0$.

We first show that $B_n$ is bounded for all sufficiently large $n$.
Otherwise choose indices $n_j$ with $\infty\in B_{n_j}$.
Then $\gamma_{n_j}\to\infty$ uniformly, and normality implies $f^{n_j}|_{U_0}\to\infty$ locally uniformly.
In particular $A_{n_j}$ avoids disks of radii tending to infinity.
Its core curve surrounds every fixed disk, so its bounded complementary continuum contains such disks.
Apply Lemma~\ref{q11e:round-ring}.
We obtain round annuli $\{r_j<|z|<R_j\}\subset A_{n_j}$ with $r_j\to\infty$ and $R_j/r_j\to\infty$.
A wandering Fatou component contains neither poles nor fixed points, so these annuli contradict the fourth hypothesis.
Thus $B_n$ is eventually the bounded Euclidean side of $\gamma_n$.

We next show that $B_n$ eventually contains at most one inverse singular value.
Suppose this fails and pass to a subsequence along which $f^n|_{U_0}\to a\in\sphere$.
Then $B_n$ shrinks spherically to $a$.
If $a$ is finite, local finiteness of the singular set gives the contradiction immediately.
If $a=\infty$, then $B_n$ avoids the finite set $E_s\cup E_f$ for large $n$.
Were it to contain a singular value $s$, this value would be an attracting fixed point.
Its immediate basin is connected, disjoint from the wandering curve $\gamma_n$, and contains $s$; hence it is contained in the bounded disk $B_n$.
Lemma~\ref{lem:bounded-basin-fixed} provides a non-attracting fixed point on its boundary.
That point lies in $B_n$: it cannot lie on $\gamma_n\subset U_n$, since a wandering Fatou component contains no fixed point.
This contradicts $B_n\cap E_f=\varnothing$.
Thus along an infinity subsequence $B_n$ actually contains no singular value.
We have established $|B_n\cap S(f)|\le1$ for all sufficiently large $n$.

Lemma~\ref{q11r:nonreset-step} now gives $f(B_n)\subset B_{n+1}$ for all large $n$.
As in the proof of Theorem~\ref{q11r:reset-propagation}, normality fills the essential curve in $U_n$, a contradiction.
\end{proof}

\subsection{Exponential perturbations of the identity}\label{sec:F}
We give a proof through the common area criterion.
\begin{theorem}\label{thm:F}
If $R\not\equiv0$ is rational and $Q$ is a nonconstant polynomial, then $f(z)=z+R(z)e^{Q(z)}$ has no wandering Fatou component.
\end{theorem}
The remaining cases in $\F$ are rational maps or the identity.

\subsubsection{The logarithmic critical coordinate}
Throughout this section let
\[
 d=\deg Q\ge1,\quad B=R'+RQ',\quad \Delta=Re^Q.
\]
The symbol $\Delta$ here is a holomorphic displacement, not a Laplacian.
Outside a disk, $R$, $B$, and $Q'$ have no zeros or poles.
The rational function $B$ is not identically zero: otherwise $Re^Q$ would be constant.
Define branch-independent functions
\begin{equation}\label{eq:globalcoords}
 u(z)=\Rea Q(z)+\log|B(z)|,
 \qquad W(z)=\frac{e^{-Q(z)}}{B(z)}.
\end{equation}
On any exterior simply connected chart choose
\[
 H(z)=Q(z)+\Log B(z).
\]
Its imaginary part is defined modulo $2\pi$, while $\Rea H=u$ and $e^{-H}=W$ are global.

\begin{lemma}\label{lem:coordinate}
For every fixed $L$, uniformly as $|z|\to\infty$ with $u(z)\le L$,
\begin{equation}\label{eq:normalform}
 H(f(z))-H(z)=e^{H(z)}(1+\varepsilon(z)),
 \qquad |\varepsilon(z)|\le C_L|z|^{-d}.
\end{equation}
The logarithms in this formula are continued over the short displacement from $z$ to $f(z)$.
Also,
\begin{equation}\label{eq:translation}
 W(f(z))=W(z)-1+O\bigl(|z|^{-d}+|W(z)|^{-1}\bigr)
\end{equation}
uniformly in the region $u(z)\le0$.
Exterior critical points satisfy $e^{H(c)}=-1$ exactly, and
\[
 H(f(c))=H(c)-1+O(|c|^{-d}).
\]
\end{lemma}
\begin{proof}
Rational asymptotics give
\[
 H'=Q'+O(|z|^{-1}),\quad
 \frac RB=\frac1{Q'+R'/R},\quad
 H'\frac RB=1+O(|z|^{-d}).
\]
The displacement is $\Delta=(R/B)e^H=O_L(|z|^{1-d})=o(|z|)$.
Taylor's formula, with $H''=O(|z|^{d-2})$ (or $O(|z|^{-2})$ if $d=1$), therefore gives
\[
 H(z+\Delta)-H(z)
 =e^H\bigl(1+O(|z|^{-d})+O(|z|^{-d}e^u)\bigr),
\]
which proves \eqref{eq:normalform}.
Exponentiating its negative and multiplying by $W$ proves \eqref{eq:translation}.
Finally, $f'=1+Be^Q=1+e^H$, proving the critical-point assertions.
\end{proof}

\begin{lemma}\label{lem:inverse}
Fix $A_0\in\R$.
For sufficiently large $T$ the domain
\[
 D_T=\{t:\Rea t<A_0,\ |t|>T\}
\]
has $d$ injective inverse charts $\psi_j$ for choices $H_j$ of $H$.
They satisfy $H_j\circ\psi_j=\mathrm{id}$ and, uniformly there,
\[
 |\psi_j(t)|\asymp |t|^{1/d},\quad
 \psi_j'(t)=\frac{1+O(|\psi_j(t)|^{-d})}{Q'(\psi_j(t))}.
\]
Every sufficiently large $z$ with $u(z)\le A_0-1$ belongs to one of these charts.
In addition, any fixed-size $t$-disk based at a sufficiently large point has a local inverse, after choosing a local logarithm.
\end{lemma}
\begin{proof}
Take the $d$ inverse branches $p_j$ of $Q$ on a slit exterior domain $\C\setminus(\overline{D(0,T_0)}\cup[T_0,\infty))$.
On the sectors relevant to $D_T$, choose logarithms of $B\circ p_j$.
Writing $\ell_j(w)=\Log B(p_j(w))$ gives
\[
 \ell_j(w)=O(\log|w|),\qquad \ell_j'(w)=O(|w|^{-1}).
\]
For $t\in D_T$, the disk $|w-t|<K\log|t|$ stays away from the slit when $T$ is large.
Rouch\'e's theorem, with fixed sufficiently large $K$, gives one solution of $w+\ell_j(w)=t$ in that disk.
Its uniqueness follows also from the derivative bound on the disk.
These solutions agree on overlaps and depend holomorphically on $t$.
Composing with $p_j$ gives $\psi_j$.
The identity $H_j\psi_j=t$ proves injectivity.
Differentiation gives the stated estimates.
Conversely, $u(z)\le A_0-1$ implies $\Rea Q(z)\le A_0-1+O(\log|z|)$.
Since $|Q(z)|\asymp |z|^d$ in the exterior, either $\Rea Q(z)$ is negative or $|\Ima Q(z)|\asymp |z|^d$.
Thus large $Q(z)$ lies away from the positive slit and $z$ is one of the roots just constructed.
The local assertion follows from the same argument, or from $H(z+v/Q'(z))=H(z)+v+O_C(|z|^{-d})$ for bounded $v$.
\end{proof}

\subsubsection{A closed invariant carrier}
Put
\[
 a=\log(5/4),\qquad \theta=\pi/3.
\]
Use Lemma~\ref{lem:inverse} with $A_0=a+2$.
Increase $T$ so that the error in \eqref{eq:normalform} throughout the inverse charts is less than $\epsilon$, where $0<\epsilon<\min(1/4,\theta/4)$.
For $\omega_k=(2k+1)\pi$, choose
\[
 a_k=\begin{cases}a,&|\omega_k|\ge T+2,\\-2T-4,&|\omega_k|<T+2,\end{cases}
 \qquad
 C_k=\{x+i(\omega_k+v):x\le a_k,\ |v|\le\theta\}.
\]
Every $C_k$ lies strictly inside the chart domain.
Set
\begin{equation}\label{eq:carrierdef}
 A=\bigcup_{j=1}^d\bigcup_{k\in\Z}\psi_j(C_k).
\end{equation}

\begin{proposition}\label{prop:traps}
The set $A$ is closed in $\C$, is contained in $\Fat(f)$, and satisfies $f(A)\subset A$.
Every Fatou component meeting $A$ is invariant.
Outside a sufficiently large disk, membership in $A$ is equivalent to
\begin{equation}\label{eq:comb}
 u(z)\le a,\qquad
 \dist\bigl(\Ima H(z),(2\Z+1)\pi\bigr)\le\theta.
\end{equation}
Furthermore, $S_{\C}(f)\setminus A$ is finite.
\end{proposition}
\begin{proof}
Consider $t=x+i(\omega_k+v)$, $|v|\le\theta$, and $\alpha=e^x\le5/4$.
For the model map $t\mapsto t+e^t$ the new coordinates are
\[
 x-\alpha\cos v,\qquad v-\alpha\sin v.
\]
The inequalities $\cos v\ge1/2$ and $|1-\alpha\cos v|\le1-\alpha/2$, valid for $0<\alpha\le5/4$, give
\[
 x_{\rm new}\le x-\alpha/2,\qquad
 |v_{\rm new}|\le(1-\alpha/2)|v|.
\]
The perturbation in \eqref{eq:normalform} has modulus at most $\epsilon\alpha$.
Thus the actual image satisfies
\begin{equation}\label{eq:stricttraps}
 x_{\rm new}\le x-(1/2-\epsilon)\alpha<x,\qquad
 |v_{\rm new}|\le\theta-(\theta/2-\epsilon)\alpha<\theta.
\end{equation}
The short displacement stays on the same inverse sheet; hence $f(\psi_j(C_k))\subset\psi_j(\Int C_k)$.

Each interior strip is mapped into itself, and the iterates there omit three fixed sphere points.
Montel's theorem gives normality.
By the strict inequalities every boundary point has a neighborhood mapped into the interior strip, so the whole closed strip lies in $\Fat(f)$.
It is connected and its image meets itself; the full Fatou component containing it is therefore invariant.
The charts are proper at infinity, and the collection of strips is locally finite.
This proves closedness of $A$.
If a strip index is among the finitely many cut off at $-2T-4$, all points satisfying $u\le a$ but omitted by that cutoff have bounded $H$ and hence bounded $z$.
This proves \eqref{eq:comb}.

Every sufficiently large critical point has $H(c)=\omega_k i$ for a large index $|k|$, so it belongs to $A$ and its critical value belongs to $A$.
Moreover,
\[
 f(c)=c-\frac{R(c)}{B(c)}=c+O(|c|^{1-d})\sim c.
\]
Thus the finite critical values have no finite accumulation other than values arising in a compact set, of which there are finitely many.
It remains to exclude finite asymptotic values.

Suppose $z=z(s)\to\infty$ continuously and $f(z(s))\to b\in\C$.
Writing $R(z)=r_0z^m(1+O(1/z))$ gives
\[
 R(z)e^{Q(z)}=-z(1+o(1)),\qquad
 \Rea Q(z)=(1-m)\log|z|+O(1).
\]
The last identity confines $z$ to shrinking sectors about the $2d$ rays where the leading term of $Q$ is purely imaginary.
Continuity selects one ray eventually, so $\arg z$ converges.
The first identity then implies that $e^{i\Ima Q(z)}$ converges to a point of the unit circle, whereas $|\Ima Q(z)|\asymp |z|^d\to\infty$.
A continuous real function whose exponential eventually stays in one proper arc must stay in a single bounded lift of that arc.
This is a contradiction.
There is no finite asymptotic value, completing the proof.
\end{proof}

\begin{lemma}\label{lem:topologyF}
A hypothetical wandering orbit for $f$ has an eventually simply connected tail, on which all transition maps are biholomorphic.
\end{lemma}
\begin{proof}
Proposition~\ref{prop:traps} makes the tail singular-free.
The restrictions are coverings by the argument in Proposition~\ref{prop:carrier}.
Stallard's theorem \cite[Theorem~3.4]{Stallard} excludes a wandering orbit all of whose components are multiply connected.
Its proof is completed before the quasiconformal section of that paper and uses normality, path lifting, and hyperbolic distance estimates.
If one component of the singular-free tail were multiply connected, injectivity of the induced map on fundamental groups would make every subsequent component multiply connected.
This contradicts that theorem.
Thus all components of the tail are simply connected, and their coverings are biholomorphisms.
\end{proof}

\subsubsection{A preimage estimate away from the deep negative region}
Let $X=\C\setminus A$ and $\Sigma=X\cap f^{-1}(A)$.
We estimate the distance to $\Sigma$ on each region $u\ge-M$.

\begin{proposition}\label{prop:net}
For each fixed $M>0$ there are $C_M,R_M<\infty$ such that
\begin{equation}\label{eq:net}
 d_X(z,\Sigma)\le C_M
 \quad\text{if }z\in X,\ |z|\ge R_M,\ u(z)\ge-M.
\end{equation}
The paths proving this estimate stay outside every prescribed compact set as their starting points tend to infinity.
\end{proposition}
\begin{proof}
We first treat sufficiently large positive $u$, with constants independent of $u$.
Then we treat a bounded band.

\proofstep{Targets at the displacement scale}
There are constants $c_0,C_0>0$ and $L_0$ such that if $u(z)=x\ge L_0$ and $|z|$ is large, one can choose $b_z\in\Int A$ satisfying
\begin{equation}\label{eq:targetscale}
 c_0\le\left|\frac{b_z-z}{\Delta(z)}\right|\le C_0.
\end{equation}
Fix a sufficiently small $\eta>0$.
If $|\Delta(z)|\le\eta|z|$, a logarithm $H$ is analytic on $D(z,3|\Delta(z)|)$.
Rational and polynomial estimates give, for $|\xi-z|\le2|\Delta(z)|$,
\[
 H(\xi)-H(z)=Q'(z)(\xi-z)+
 O\bigl((\eta+|z|^{-d})|Q'(z)|\,|\xi-z|\bigr).
\]
Also $|Q'(z)\Delta(z)|=e^x(1+O(|z|^{-d}))$.
Choose a real $v$ with $|v|\le\pi$ so that $\Ima(H(z)+iv)\in(2\Z+1)\pi$, and prescribe
\[
 H(b_z)=H(z)-e^x+iv.
\]
Rouch\'e's theorem on the circle $|\xi-z|=2|\Delta(z)|$ gives this solution; the same displayed estimate gives $|b_z-z|\asymp|\Delta(z)|$.
For $x\ge L_0$ large, its real coordinate $x-e^x$ is below $a$.
Since $|b_z|\asymp|z|\to\infty$, \eqref{eq:comb} puts it in $\Int A$.

If $|\Delta(z)|>\eta|z|$, let $r_* =\max(|z|,|\Delta(z)|)$.
A middle curve of a fixed strip, parametrized as $\psi_1(-s+i\pi)$ for large $s$, contains points with modulus comparable to every sufficiently large prescribed radius.
Choose $b_z$ on that curve with $K r_*\le |b_z|\le2K r_*$, for one fixed $K>3$.
The triangle inequality proves \eqref{eq:targetscale}, with constants independent of $z$.
This proves the target assertion.

\proofstep{Large positive band}
For each fixed $C$, uniformly on $|v|\le C$,
\begin{equation}\label{eq:scaledmap}
 \frac{f(z+v/Q'(z))-z}{\Delta(z)}
 =e^v\bigl(1+O_C(|z|^{-d})\bigr)
   +\frac{v}{Q'(z)\Delta(z)}.
\end{equation}
Every complex number $\beta$ in the compact annulus $c_0\le|\beta|\le C_0$ has a logarithm $v_\beta$ in one fixed disk.
Choose $C$ so that these logarithms, together with circles of radius $1/4$ about them, lie in $|v|<C$.
On these circles $|e^v-\beta|$ has a positive lower bound independent of $\beta$.
Since $|Q'\Delta|\asymp e^x$, choose $L\ge L_0$ sufficiently large, and then $|z|$ large, so that \eqref{eq:scaledmap} and Rouch\'e's theorem give a point
\[
 w=z+v_z/Q'(z),\qquad |v_z|<C,\qquad f(w)=b_z.
\]
Increase $L$ also so that $L>C+4+a$.
On $|v|\le C+2$ one has $H(z+v/Q')=H(z)+v+o(1)$, so this whole disk in the $z$-plane avoids $A$.
Its radius is $(C+2)/|Q'(z)|$.
The straight segment to $w$ has length at most $C/|Q'(z)|$ and stays at distance at least $2/|Q'(z)|$ from the disk boundary.
Domain monotonicity bounds its $X$-hyperbolic length by $C/2$.
Since $w\notin A$ and $f(w)\in A$, its endpoint is in $\Sigma$.

\proofstep{The bounded band $-M\le u\le L$}
In a bounded $H$-window centered at any sufficiently large $z$, the set $A$ is exactly the inverse image of the strip comb
\[
 A_* =\{x+iy:x\le a,\ \dist(y,(2\Z+1)\pi)\le\theta\}.
\]
The map in this window is $t\mapsto t+e^t+o(1)$ uniformly, by Lemma~\ref{lem:coordinate}; local inverse charts exist on any prescribed bounded enlargement of the window.
The boundary of $A_*$ in $-M-1\le x\le a$ maps strictly inside $A_*$, with a positive margin depending only on $M$, by the inequalities used in \eqref{eq:stricttraps}.
Therefore a fixed exterior collar of this boundary maps into $A$ for all sufficiently large $z$.
Points in that collar already belong to $\Sigma$.

For the other points, choose a small $\delta>0$ such that $a+\delta-e^{a+\delta}<a$.
In the $t$-plane join the starting point horizontally to $x=L+2$, vertically by at most $2\pi$ to an odd multiple of $\pi$, and horizontally back to
\[
 t_* =a+\delta+i(2k+1)\pi.
\]
This path has uniformly bounded length and stays at a fixed positive distance from $A_*$.
Indeed, on its first segment a starting point with $x\le a$ is in a complementary horizontal gap; if $x>a$ it is already to the right of the comb.
The vertical segment is to the right, and the last segment has $x\ge a+\delta$.
At its endpoint $t_*+e^{t_*}$ lies in $\Int A_*$, so the actual endpoint maps into $A$ when $|z|$ is sufficiently large.
The inverse image of a fixed-radius disk about every path point lies in $X$.
Comparing its metric with that of this disk bounds the hyperbolic length of the lifted path.
Its endpoint belongs to $\Sigma$.

Both constructions use $z$-displacements $O_{M,L,C}(1/|Q'(z)|)=o(|z|)$.
They consequently avoid every fixed compact set as $z\to\infty$.
\end{proof}

\subsubsection{Hyperbolic expansion outside the deep negative region}
We use two relative-metric inequalities.
If $V\subset X$ are hyperbolic surfaces and $z\in V$, then
\begin{align}
 d_X(z,X\setminus V)\le C
 &\ \Longrightarrow\ \frac{\rho_V(z)}{\rho_X(z)}\ge
 L(C):=\frac1{\sinh(2C)\log\coth C}>1,\label{eq:relative-lower}\\
 d_X(z,X\setminus V)\ge D
 &\ \Longrightarrow\ \frac{\rho_V(z)}{\rho_X(z)}\le\coth D.
 \label{eq:relative-upper}
\end{align}
For the upper bound lift a radius-$D$ ball to a universal disk; the lift contains the Euclidean disk of radius $\tanh D$.
For the lower bound lift a path to one omitted point and compare with the disk punctured there.
The formula $\rho_{\D\setminus\{0\}}(w)=(2|w|\log(1/|w|))^{-1}$ gives \eqref{eq:relative-lower}.
These estimates are also discussed in \cite[Proposition~3.4]{MBRG}.

If $T$ is a finite subset of $X$, properness of the complete hyperbolic metric implies $d_X(z,T)\to\infty$ as $|z|\to\infty$ in a component meeting $T$.
Thus
\begin{equation}\label{eq:compactpuncture}
 \frac{\rho_{X\setminus T}(z)}{\rho_X(z)}\longrightarrow1.
\end{equation}
The assertion is uniform across components: only finitely many meet $T$, and on the others the ratio is one.

\begin{lemma}\label{lem:expansion}
Let $P$ be a finite forward invariant repelling marking, put $X_P=\C\setminus(A\cup P)$, $Y_P=X_P\setminus E$, and $V_P=f^{-1}(Y_P)$, where $E=S_{\C}(f)\setminus A$.
For each $M>0$ there are $R<\infty$ and $\Lambda>1$ such that
\begin{equation}\label{eq:expansion}
 \frac{|f'(z)|\rho_{X_P}(f(z))}{\rho_{X_P}(z)}\ge\Lambda
\end{equation}
whenever $z\in V_P$, $u(z)\ge-M$, and $|z|,|f(z)|\ge R$.
\end{lemma}
\begin{proof}
The paths in Proposition~\ref{prop:net} avoid a growing compact set.
Equation~\eqref{eq:compactpuncture} for the finite set $P$ therefore bounds their $X_P$-length by, say, twice their $X$-length.
Also $\Sigma\cap P=\varnothing$: a repelling periodic point cannot map into the Fatou set.
Hence $d_{X_P}(z,X_P\setminus V_P)\le2C_M$.
Covering invariance gives the exact identity
\[
 \frac{|f'(z)|\rho_{X_P}(f(z))}{\rho_{X_P}(z)}
 =\frac{\rho_{V_P}(z)/\rho_{X_P}(z)}
 {\rho_{Y_P}(f(z))/\rho_{X_P}(f(z))}.
\]
Its numerator is at least $L(2C_M)>1$ by \eqref{eq:relative-lower}; its denominator tends uniformly to one as $|f(z)|\to\infty$, by \eqref{eq:compactpuncture} for $E\cap X_P$.
This proves \eqref{eq:expansion}.
\end{proof}

\subsubsection{Deep negative translation and completion for \texorpdfstring{$\F$}{class F}}
\begin{lemma}\label{lem:deep}
There is $M_0>0$ such that an orbit $z_n\to\infty$ cannot remain in $X=\C\setminus A$ while satisfying $u(z_n)<-M_0$ for all sufficiently large $n$.
\end{lemma}
\begin{proof}
Choose $M_0$ large.
Formula~\eqref{eq:translation} then implies, along such an orbit and after increasing the starting index,
\[
 W(z_{n+1})=W(z_n)-1+\epsilon_n,\qquad |\epsilon_n|\le1/10.
\]
Writing $W_n=W(z_n)$ gives
\[
 \Rea W_n\le\Rea W_N-\tfrac9{10}(n-N),\qquad
 |\Ima W_n|\le|\Ima W_N|+\tfrac1{10}(n-N).
\]
Consequently $|W_n|\to\infty$, and eventually $W_n$ lies in a sector strictly narrower than $\theta$ about the negative real axis.
Since $W_n=e^{-H(z_n)}$, this means that $u(z_n)=-\log|W_n|\le a$ and $\dist(\Ima H(z_n),(2\Z+1)\pi)<\theta$.
For large $n$, \eqref{eq:comb} gives $z_n\in A$, a contradiction.
\end{proof}

\begin{proof}[Proof of Theorem~\ref{thm:F}]
Suppose there is a wandering component.
Propositions~\ref{prop:traps} and \ref{prop:carrier}, together with Lemma~\ref{lem:topologyF}, give an escaping tail and one fixed marked background $X_P$ on which the one-step derivatives tend to one.
Use $M=M_0$ from Lemma~\ref{lem:deep}.
Lemma~\ref{lem:expansion} shows that $u(f^n(z))\ge-M_0$ can occur only finitely often, since both consecutive orbit points tend to infinity.
Therefore $u(f^n(z))<-M_0$ eventually.
This contradicts Lemma~\ref{lem:deep}.
No wandering component exists.
\end{proof}

\subsection{Finite order in the second Riccati class}
\begin{lemma}\label{lem:R2-finite-order}
Every meromorphic function satisfying
\[
 f'=r(z)(f-z)(f-\tau),\qquad r\in\C(z),\quad \tau\in\C,
\]
has finite order.
\end{lemma}
\begin{proof}
The claim is immediate for constant or rational functions.
Otherwise $f\not\equiv\tau$.
Set $k=1/(f-\tau)$.
Differentiation gives
\begin{equation}\label{rr:struct-7}
 k'=r(z)(z-\tau)k-r(z).
\end{equation}
At a point where $r$ is holomorphic, $k$ has no pole: the derivative of a pole would have order one larger than either term on the right.
Thus $k$ has only finitely many poles.
Choose $R$ enclosing those poles and all poles of $r$.
For some $C>0$ and integer $M\ge0$, $|r(z)(z-\tau)|+|r(z)|\le C(1+|z|)^M$ when $|z|\ge R$.
The integral form of \eqref{rr:struct-7} along every radial ray and Gronwall's inequality, with a uniform bound on the initial circle, give $|k(z)|\le C_1\exp(C_2|z|^{M+1})$ outside that circle.
After subtraction of its finitely many principal parts, $k$ is an entire function of finite order plus a rational function.
Hence $k$ and $f=\tau+1/k$ have finite Nevanlinna order.
\end{proof}

\subsection{Linear differential equations and the classes \texorpdfstring{$\mathcal N,\mathcal R_2$}{N and R2}}\label{sec:linear}
\subsubsection{A rational differential equation and its fixed roots}\label{nr:sec-ode}

We treat the two remaining first-order classes by a common local model.
The first step is an elementary consequence of finite order.

\begin{lemma}\label{nr:polynomial-exponent}
Suppose that $f$ is meromorphic of finite order and
\[
 f'=r e^p(f-z),
\]
where $r$ is rational and $p$ is polynomial.
If $r\not\equiv0$, then $p$ is constant.
\end{lemma}
\begin{proof}
Put $g=f-z$.
Then $g'=re^pg-1$.
A pole of $g$ at an ordinary point of $r$ would give different pole orders on the two sides.
Thus $g$ has finitely many poles.
Choose a polynomial $D$ clearing them, and put $G=Dg$.

We record a pointwise logarithmic-derivative estimate.
For almost every $\theta$, there are $C_\theta,R_\theta,M<\infty$ such that
\begin{equation}\label{nr:ray-derivative}
 |g'/g(te^{i\theta})|\le C_\theta t^M\qquad(t>R_\theta).
\end{equation}
Indeed, choose integers $q,K$ larger than the order of $G$ and use a Hadamard factorization
\[
 G(z)=z^\ell e^{B(z)}\prod_j E_q(z/a_j).
\]
The angular projections of the disks $D(a_j,|a_j|^{-K})$, for sufficiently large zeros, have summable lengths.
Almost every ray therefore meets only finitely many of them.
Outside these disks,
\[
 \frac{G'}G=\frac\ell z+B'(z)
       +\sum_j\frac{z^q}{a_j^q(z-a_j)}.
\]
For $|a_j|\le2|z|$, use $|z-a_j|\ge(2|z|)^{-K}$ and $\sum_j|a_j|^{-q}<\infty$; for $|a_j|>2|z|$, use $|z-a_j|\ge|a_j|/2$.
Both sums are bounded by a power of $|z|$.
The finitely many initial zeros and the rational function $D'/D$ do not affect the conclusion.
This proves \eqref{nr:ray-derivative}.

Suppose $\deg p>0$.
On a good ray where the leading term of $p$ has positive real part, the identity
\[
 g=\frac1{re^p-g'/g}
\]
shows that $g$ decays exponentially.
On a ray where that real part is negative, the coefficient $re^p$ decays exponentially.
Radial variation of constants gives
\[
 g(te^{i\theta})=-te^{i\theta}+C_\theta+o(1).
\]
For clarity, the latter estimate also follows by applying Gronwall to $|g(te^{i\theta})|/(1+t)$ and then integrating $(g(z)+z)'=re^pg$ along the ray.
Thus $G=O(t^{\deg D+1})$ on almost every ray; the finitely many Stokes rays may be discarded.

Choose $\kappa$ larger than the order of $G$ and finitely many good rays whose successive angular gaps are all less than $\pi/\kappa$.
Apply Phragm\'en--Lindel\"of on each exterior sector to $G(z)/z^{\deg D+1}$, using its two radial boundaries and its inner circular boundary.
It is bounded throughout each sector.
Hence $G$ is polynomial and $g$ is rational.
But $(g'+1)/g=re^p$ cannot then have a nonconstant polynomial exponent.
This contradiction proves the lemma.
\end{proof}

Consider now a finite-order transcendental meromorphic function $y$ satisfying
\begin{equation}\label{nr:linear-ode}
 y'=ay-1,\qquad a\text{ rational}.
\end{equation}
Fix $\tau\in\C$, put $Z=z-\tau$, and consider either of the maps
\begin{equation}\label{nr:two-maps}
 f_0(z)=z+y(z),\qquad
 f_1(z)=z+\frac{y(z)}{1-y(z)/Z}
       =\tau+\frac{Z^2}{Z-y(z)}.
\end{equation}
Removable values in these formulas are understood by meromorphic continuation.
The following estimates apply to both maps.

The function $y$ has finitely many poles, by the pole-order argument just used.
If $a=O(1/z)$, radial Gronwall gives a polynomial growth bound for $y$, uniformly in the argument: on an exterior ray,
\[
 |\partial_t y(te^{i\theta})|\le C|y(te^{i\theta})|/t+1,
\]
and the initial values on a fixed circle are uniformly bounded.
Clearing the poles would make $y$ a polynomial, contrary to the assumption.
Consequently
\begin{equation}\label{nr:a-asymptotics}
 a(z)=\alpha z^m(1+O(z^{-1})),\quad \alpha\ne0,\quad m\ge0.
\end{equation}
Write $d=m+1$ and let $P$ be a polynomial primitive of the polynomial part of $a$.
Then
\begin{equation}\label{nr:P-asymptotics}
 P'=a+O(z^{-1}),\qquad P(z)\sim\alpha z^d/d,
 \qquad za(z)\sim dP(z).
\end{equation}

\begin{lemma}\label{nr:local-model}
Put $a_0=a(z)$, $L_z=|za_0|$, $v=a_0y(z)$ and $V_z(\xi)=a_0y(z+\xi/a_0)$.
On each fixed disk $|\xi|\le K$,
\begin{align}
 V_z(\xi)&=1+(v-1)e^\xi+O_K(L_z^{-1})
       &&\text{if $v$ is bounded},\label{nr:bounded-model}\\
 \frac{y(z+\xi/a_0)}{y(z)}
   &=e^\xi\bigl(1+O_K(L_z^{-1})\bigr)+O_K(|v|^{-1})
       &&\text{if $|v|\longrightarrow\infty$}.
       \label{nr:large-model}
\end{align}
For bounded $v$, both maps in \eqref{nr:two-maps} satisfy
\begin{equation}\label{nr:map-model}
 a_0[f_j(z+\xi/a_0)-z]
   =\xi+1+(v-1)e^\xi+O_K(L_z^{-1}),\qquad j=0,1.
\end{equation}
\end{lemma}
\begin{proof}
The exact equation is
\[
 V_z'=b_zV_z-1,\qquad
 b_z(\xi)=\frac{a(z+\xi/a_0)}{a_0}=1+O_K(L_z^{-1}).
\]
With $E_z(\xi)=\exp\int_0^\xi b_z(s)\,ds$ its solution is
\[
 V_z(\xi)=E_z(\xi)
       \left(v-\int_0^\xi E_z(s)^{-1}\,ds\right).
\]
This proves the first two assertions.
When $v$ is bounded, $y(z+\xi/a_0)/(z+\xi/a_0-\tau)=O_K(L_z^{-1})$, which proves the last one.
\end{proof}

\begin{lemma}\label{nr:root-carrier}
For either map $f_j$ there is a closed forward invariant set $A\subset F(f_j)$, formed by small disks about all sufficiently large zeros of $y$, such that each component meeting $A$ is invariant and $S_\C(f_j)\setminus A$ is finite.
\end{lemma}
\begin{proof}
At a large zero $c$ of $y$, \eqref{nr:map-model} becomes
\[
 a(c)[f_j(c+\xi/a(c))-c]=\xi+1-e^\xi+O_K(|ca(c)|^{-1}).
\]
Both sides have value and first derivative zero at $\xi=0$.
Cauchy estimates, or the exact differential equation, therefore give a uniform quadratic bound near zero.
Choose $\epsilon>0$ sufficiently small.
For all large roots $c$, the disk of radius $2\epsilon/|a(c)|$ is mapped into the disk of radius $\epsilon/(2|a(c)|)$ about $c$.
Thus
\begin{equation}\label{nr:carrier-disks}
 A_c=\overline D(c,\epsilon/|a(c)|),\qquad A=\bigcup_c A_c
\end{equation}
is contained in the attracting basins of the fixed points $c$.
The disks are locally finite, so $A$ is closed.
They are uniformly separated in the local scale: otherwise \eqref{nr:bounded-model}, at a root, would contradict the separation of the zeros of $1-e^\xi$.
Reducing $\epsilon$ if necessary makes them disjoint.
The doubled disks provide a uniform exterior collar mapped into $A$.

For $f_0$ one has $f_0'=ay$; hence all but finitely many critical values are these fixed roots.
There are no finite asymptotic values.
Indeed, on a bounded value disk its inverse branches satisfy the exact equation
\begin{equation}\label{nr:inverse0}
 \frac{dz}{dw}=\frac1{a(z)(w-z)}.
\end{equation}
On a product of a fixed bounded value disk and $D(z_0,1)$, with $|z_0|\to\infty$, the right side and its $z$ derivative tend uniformly to zero.
Picard iteration gives a holomorphic inverse on that whole value disk, whose image remains in $D(z_0,1/2)$.
If an asymptotic curve tended to a finite value, take a sufficiently late starting point and a fixed disk containing its entire image tail.
Uniqueness of continuation makes the tail stay in this bounded inverse image, a contradiction.

For $f_1$, direct differentiation gives
\begin{equation}\label{nr:inverse1-identity}
 f_1'=r_1(f_1-z)(f_1-\tau),\qquad
 r_1=\frac{a-2/Z}{Z}.
\end{equation}
At an ordinary point, the nonconstant solution $f_1$ cannot assume $\tau$, by uniqueness for this differential equation.
Thus all but finitely many critical values again come from fixed roots.
An inverse branch away from the value $\tau$ satisfies
\[
 \frac{dz}{dw}=\frac1{r_1(z)(w-z)(w-\tau)}.
\]
On a fixed value disk whose closure avoids $\tau$, the right side is $O(|z|^{-m})$ and its $z$ derivative is $O(|z|^{-m-1})$.
Choose the value radius sufficiently small.
The same Picard argument gives inverse branches of uniformly bounded image displacement and excludes a finite asymptotic value other than $\tau$.
Put this one possible value and all finite coefficient exceptions into the residual set.
The critical values outside that set form the discrete root set, so there is no further finite accumulation of singular values.
\end{proof}

\begin{lemma}\label{nr:radial-roots}
The zeros of $y$ contain a sequence $c_n\to\infty$ such that $|c_{n+1}|/|c_n|\to1$.
In particular, for every $\lambda>1$ all sufficiently large open annuli $r<|z|<\lambda r$ contain a fixed point in $A$.
\end{lemma}
\begin{proof}
If $y$ had finitely many zeros, finite-order factorization and its finite pole set would give $y=R e^Q$, with $R$ rational and $Q$ polynomial.
Substitution into \eqref{nr:linear-ode} gives
\[
 [R'+(Q'-a)R]e^Q=-1.
\]
The bracket is a nonzero rational function, so $Q$ is constant and $y$ is rational, a contradiction.
Thus zeros tend to infinity.

At any sufficiently large root $c$, apply Rouch\'e to \eqref{nr:bounded-model} on fixed small circles about $\xi=\pm2\pi i$.
There are roots
\[
 c_\pm=c+\frac{\pm2\pi i+O(|ca(c)|^{-1})}{a(c)}.
\]
Using \eqref{nr:P-asymptotics} gives
\begin{equation}\label{nr:root-step}
 P(c_\pm)-P(c)=\pm2\pi i+O(|P(c)|^{-1}).
\end{equation}
Choose an arbitrarily large initial root and one fixed sign $\sigma$ with $\sigma\operatorname{Im}P(c_0)\ge0$.
Recursively take that successor.
Put $Z_n=P(c_n)$ and $T=2\pi$.
Above a fixed threshold, the step error is at most $\eta_0<T/\sqrt2$.
As long as earlier steps stay above it,
\[
 Z_n=Z_0+\sigma iTn+E_n,\quad |E_n|\le\eta_0n,
 \qquad
 |Z_n|\ge\frac{|Z_0|}{\sqrt2}
             +\left(\frac T{\sqrt2}-\eta_0\right)n.
\]
The last inequality follows from $|Z_0+\sigma iTn|\ge(|Z_0|+Tn)/\sqrt2$.
Starting beyond $\sqrt2$ times the threshold proves by induction that all steps exist in the exterior and escape.
Their $P$ increments are bounded, hence $|c_{n+1}|/|c_n|\to1$.
For any large $r$, take the first index with $|c_n|>r$.
Far enough along the chain the successive ratio is less than $\sqrt\lambda$, so $r<|c_n|<\sqrt\lambda r<\lambda r$.
This proves the open-annulus assertion without any monotonicity assumption on $|c_n|$.
\end{proof}

\subsubsection{The preimage mesh for the rational differential equation}
\label{nr:sec-mesh}

The next construction supplies targets even when the displacement is much larger than the local differential-equation scale.

\begin{lemma}\label{nr:target-scale}
There exist $c_0,C_0,V_0,R_0>0$ such that if $|z|\ge R_0$ and $|a(z)y(z)|\ge V_0$, a sufficiently large zero $c$ of $y$ satisfies
\begin{equation}\label{nr:target-annulus}
 c_0\le\left|\frac{c-z}{y(z)}\right|\le C_0.
\end{equation}
\end{lemma}
\begin{proof}
Fix a sufficiently small $\eta>0$.
If $|y(z)|\ge\eta|z|$, put $s=\max(|z|,|y(z)|)$ and choose a root with $4s\le|c|\le8s$ by Lemma~\ref{nr:radial-roots}.
The triangle inequality proves the claim.

It remains to treat
\begin{equation}\label{nr:small-target-range}
 V_0\le|v|=|a_0y(z)|\le\eta L_z,
 \qquad a_0=a(z),\quad L_z=|za_0|.
\end{equation}
Choose a small fixed $\delta>0$.
On $D(z,\delta|z|)$, $|a(w)/a_0-1|\le C\delta+o(1)<1/4$.
The primitive
\[
 T_z(w)=\int_z^w a(\zeta)\,d\zeta
\]
is univalent there: its difference quotient divided by $a_0$ stays within $1/4$ of 1.
Rouch\'e on the boundary shows that its image contains $|t|<cL_z$ for a fixed $c>0$.
On a smaller such disk its inverse $w_z(t)$ satisfies
\begin{equation}\label{nr:primitive-estimates}
 |w_z(t)-z|\asymp |t|/|a_0|,
 \qquad b(t)=\frac1{a(w_z(t))},
 \qquad |b'(t)|\le\frac C{|a_0|L_z}.
\end{equation}
Here $b'=-a'/a^3$.

In this coordinate $Y(t)=y(w_z(t))$ solves $Y'=Y-b(t)$.
Integration by parts gives the exact formula
\begin{equation}\label{nr:exact-primitive}
 a_0Y(t)=e^t(v-1)+a_0b(t)
       -a_0e^t\int_0^t e^{-s}b'(s)\,ds.
\end{equation}
Choose a logarithm of $v-1$ with imaginary part in $[-\pi,\pi]$ and an integer $k$ such that
\begin{equation}\label{nr:far-phase}
 t_0=-\operatorname{Log}(v-1)+(2k+1)\pi i,
 \qquad |v|\le\operatorname{Im}t_0\le2|v|.
\end{equation}
This is possible for large $V_0$.
Fix a small $r_*>0$ and put $t=t_0+\xi$, $|\xi|\le r_*$.
In \eqref{nr:exact-primitive} integrate first vertically from $0$ to $i\operatorname{Im}t$, then horizontally to $t$.
Both segments lie in the primitive chart if $\eta$ is sufficiently small.
On the vertical segment, the length $O(|v|)$ is cancelled by $|e^t|=O(|v|^{-1})$.
On the horizontal segment,
\[
 |e^t|\int_{\operatorname{Re}t}^0 e^{-x}\,dx=O(1).
\]
Consequently
\[
 \left|a_0e^t\int_0^t e^{-s}b'(s)\,ds\right|\le C/L_z.
\]
Also $a_0b(t)=1+O(|t|/L_z)=1+O(\eta)+O(L_z^{-1})$.
Since $e^{t_0}(v-1)=-1$, we obtain
\[
 a_0y(w_z(t_0+\xi))=1-e^\xi+O(\eta)+O(L_z^{-1})
 \qquad(|\xi|\le r_*).
\]
Choose $\eta$ small and then $R_0,V_0$ large.
Rouch\'e gives a root $c=w_z(t_0+\xi_*)$, $|\xi_*|<r_*$.
Equations \eqref{nr:primitive-estimates}--\eqref{nr:far-phase} imply $|c-z|\asymp |v|/|a_0|=|y(z)|$.
Moreover $|c-z|\le C\eta|z|$, so the root is large.
This completes the proof.
\end{proof}

Let $P_0$ be a fixed finite forward invariant repelling marking containing at least three points, and put
\[
 X=\C\setminus(A\cup P_0),\qquad
 \Sigma_j=X\cap f_j^{-1}(A).
\]
The marking is unrelated to the polynomial $P$ above.

\begin{proposition}\label{nr:mesh}
Suppose that $z_n,f_j(z_n)\in X$, $|z_n|,|f_j(z_n)|\to\infty$, and $z_n\notin f_j^{-1}(A)$.
If
\[
 d_X(z_n,\Sigma_j)\longrightarrow\infty,
\]
then $a(z_n)y(z_n)\to1$.
The paths used to contradict this distance divergence avoid every fixed compact subset of the plane.
\end{proposition}
\begin{proof}
Write $v=a(z)y(z)$.
We exclude a bounded subsequence with $|v-1|$ bounded away from zero and then an unbounded subsequence.

In the bounded case, \eqref{nr:bounded-model} gives, on any fixed enlarged normalized window, the root lattice
\[
 \xi_k=\operatorname{Log}[-1/(v-1)]+2\pi i k
\]
up to errors tending to zero.
Its spacing is uniform, and a representative root is uniformly bounded.
The carrier disks converge in that window to small disks about these lattice points.
Since the starting point lies outside $f_j^{-1}(A)$, it lies outside the doubled carrier disks.
For a nearby actual root $c$, the point
\[
 w=c+\frac{3\epsilon}{2a(c)}
\]
is outside $A$, with uniform local clearance, and lies in $f_j^{-1}(A)$ by the collar property in Lemma~\ref{nr:root-carrier}.
Join $z$ to $w$ in a fixed enlarged normalized window by a bounded-length path avoiding the carrier disks with fixed clearance.
One may take a polygonal path and detour around each enlarged disk; there are uniformly finitely many and they are uniformly separated.
Comparing with small Euclidean disks along the path bounds its $X$-hyperbolic length uniformly.

Suppose next that $|v|\to\infty$.
For $f_0$, choose the target root $c$ from Lemma~\ref{nr:target-scale}.
Then $\beta=(c-z)/y(z)$ lies in a fixed compact annulus, and \eqref{nr:large-model} gives
\[
 \frac{f_0(z+\xi/a(z))-z}{y(z)}=e^\xi+o(1)
\]
uniformly on fixed disks.
Rouch\'e on small circles about bounded logarithms of $\beta$ produces a point
\begin{equation}\label{nr:near-preimage}
 w=z+O(1/a(z)),\qquad f_0(w)=c\in A.
\end{equation}

For $f_1$, put $q=y(z)/Z$ and write $C=c-\tau$ for the target root.
The equation $f_1(w)=c$ is exactly
\begin{equation}\label{nr:reciprocal-target}
 y(w)=(w-\tau)\left(1-\frac{w-\tau}{C}\right).
\end{equation}
We choose $c$ so that
\begin{equation}\label{nr:beta-reciprocal}
 \beta=\frac{1-Z/C}{q}
\end{equation}
lies in one fixed compact annulus.
Fix $\eta>0$ below the small-displacement threshold used in the proof of Lemma~\ref{nr:target-scale}, and fix $K>4$.
If $|q|\le\eta$, that construction gives $|c-z|\asymp|y(z)|$ and $|c-z|\le C_1\eta|Z|$.
Thus $C\asymp Z$ and $\beta=(Z/C)(c-z)/y(z)$ has the required bounds.
If $\eta\le|q|\le2$, choose a root with $K|Z|\le|C|\le2K|Z|$, using Lemma~\ref{nr:radial-roots}; then
\[
 \frac{1-1/K}{2}\le|\beta|\le\frac{1+1/K}{\eta}.
\]
The same annulus property holds for the shifted root moduli $|c-\tau|$.
Finally, if $|q|>2$, put $T=|Z/q|$.
Since
\[
 |f_1(z)-\tau|=|Z/(1-q)|\asymp T,
\]
the escape assumption gives $T\to\infty$.
Choose a root with $T/K\le|C|\le2T/K$.
It satisfies
\[
 K/2-1/2\le|\beta|\le K+1/2.
\]
All these target roots tend to infinity.

It remains to check the variation of the right side of \eqref{nr:reciprocal-target} across the local window; this is essential when $q$ is large.
Put $a_0=a(z)$ and $w=z+\xi/a_0$.
Dividing that side by $y(z)$ gives a function $B(\xi)$ with the exact expansion
\begin{equation}\label{nr:target-variation}
 B(\xi)-\beta
   =\frac\xi v(-1+2\beta q)-\frac{\xi^2}{a_0Cv}.
\end{equation}
Use $q/v=1/(a_0Z)$ and $Z/C=1-\beta q$.
Uniformly on bounded windows,
\[
 B(\xi)=\beta+O\left(\frac1{|v|}+\frac1{|a_0Z|}
              +\frac1{|a_0Z||v|}+\frac1{|a_0Z|^2}\right)=\beta+o(1).
\]
The left side of \eqref{nr:reciprocal-target}, divided by $y(z)$, tends to $e^\xi$.
Rouch\'e again supplies \eqref{nr:near-preimage}, now for $f_1$.
It is an actual finite preimage: at a solution of \eqref{nr:reciprocal-target},
\[
 (w-\tau)-y(w)=(w-\tau)^2/C\ne0
\]
for all sufficiently large starting points.

For either map, every fixed normalized disk about $z$ avoids $A$ when $|v|\to\infty$.
Indeed, \eqref{nr:large-model} excludes zeros in each such disk.
If a carrier disk met it, its center would lie in a fixed larger normalized disk: distant noncomparable centers have radii $O(|c|^{-m})$ and cannot reach it, and comparable centers satisfy $|a(c)|\asymp|a(z)|$.
This is a contradiction.
The same disks avoid $P_0$.
If $|a(z)(w-z)|\le K_0$, choose a fixed $K_1>K_0$ such that the disk of normalized radius $K_1$ lies in $X$.
Its intrinsic distance gives
\[
 d_X(z,w)\le\operatorname{arctanh}(K_0/K_1).
\]
Taking the starting points sufficiently far out and then $K_1$ arbitrarily large shows that this distance even tends to zero.
In particular it is bounded, contradicting the hypothesis.
All constructed paths have Euclidean displacement $O(1/|a(z)|)=o(|z|)$ and avoid compact sets.
\end{proof}

\begin{lemma}\label{nr:algebraic-channel}
For either map $f_j$, an escaping orbit $z_{n+1}=f_j(z_n)$ satisfying $a(z_n)y(z_n)\to1$ eventually lies in an invariant Fatou component.
\end{lemma}
\begin{proof}
One has $y(z_n)=(1+o(1))/a(z_n)$ and, for $j=1$, $y(z_n)/(z_n-\tau)\to0$.
Thus
\[
 z_{n+1}-z_n=(1+o(1))/a(z_n),\qquad
 P(z_{n+1})-P(z_n)=1+o(1).
\]
It follows that $P(z_n)/n\to1$.
Eventually the orbit lies in one branch of a positive-real sector for $P$.
The branch cannot switch: successive $z_n$ differ by $o(|z_n|)$, while the finitely many branches over a narrow positive-real cone have disjoint angular ranges.

On a slightly larger such sector choose a primitive $A_0$ of $a$ and put $E=e^{A_0}$.
Since $A_0=P+O(\log z)$, the integral
\begin{equation}\label{nr:canonical-particular}
 y_*(z)=E(z)\int_z^\infty E(\zeta)^{-1}\,d\zeta
\end{equation}
converges along the positive-real direction of the $P$ coordinate and defines a holomorphic solution of $y_*'=ay_*-1$.
Uniformly on closed subsectors,
\begin{equation}\label{nr:particular-estimate}
 y_*(z)=\frac1{a(z)}\left(1+O(|P(z)|^{-1})\right).
\end{equation}
Here is a direct estimate.
Parametrize the contour by $P(\zeta)=P(z)+t$, $t\ge0$.
The exponential quotient is bounded by $C e^{-t}(1+t/|P(z)|)^C$, and $d\zeta/dt=1/P'(\zeta)$.
For $t\le|P(z)|/2$, its integrand equals $e^{-t}/a(z)$ with relative error $O((1+t)^C/|P(z)|)$; the remaining exponentially decaying tail satisfies the same integrated bound.
This proves \eqref{nr:particular-estimate}.

The difference $y-y_*$ equals $CE$ on the connected sector.
Along $z_n$, $E(z_n)$ grows exponentially because $P(z_n)/n\to1$, whereas both $y(z_n)$ and $y_*(z_n)$ are asymptotic to $1/a(z_n)$.
Therefore $C=0$.
Substitution into either formula in \eqref{nr:two-maps} gives throughout this sector
\begin{equation}\label{nr:cone-translation}
 P(f_j(z))=P(z)+1+O(|P(z)|^{-1}).
\end{equation}
For fixed small $\gamma>0$ and sufficiently large $R$, the region corresponding to
\[
 \{w:\operatorname{Re}w>R,\quad
          |\operatorname{Im}w|<\gamma\operatorname{Re}w\}
\]
is mapped into itself: the real increment is at least $1/2$ and the absolute imaginary increment is less than $\gamma/2$.
Iterates escape locally uniformly there.
The inverse branch is preserved because $f_j(z)-z=O(1/a(z))=o(|z|)$, whereas distinct branches over this cone are separated by a fixed positive angle.
The region is connected, so its full Fatou component is invariant.
Since $P(z_n)/n\to1$, the orbit eventually enters this region.
\end{proof}

\subsubsection{Completion for the classes \texorpdfstring{$\mathcal N$ and $\mathcal R_2$}{N and R2}}
\label{nr:sec-completion}

\begin{theorem}\label{nr:no-wandering}
Let $f$ be a meromorphic function of finite order satisfying either
\[
 f'=r e^p(f-z)
 \qquad\text{or}\qquad
 f'=r(f-z)(f-\tau),
\]
where $r$ is rational, $p$ is polynomial, and $\tau\in\C$.
Then $f$ has no wandering Fatou components.
\end{theorem}
\begin{proof}
Constant and rational maps are already covered by the rational case.
For the first equation, Lemma~\ref{nr:polynomial-exponent} makes $p$ constant.
Absorb it into $r$, set $y=f-z$, and obtain \eqref{nr:linear-ode} and $f=f_0$.

For the second equation put
\[
 Z=z-\tau,\qquad k=\frac1{f-\tau},\qquad
 y=Z-Z^2k.
\]
Direct differentiation gives
\[
 k'=rZk-r,\qquad y'=ay-1,\qquad a=rZ+2/Z,
\]
and $f=f_1$ in \eqref{nr:two-maps}.
The function $y$ has finite order; its poles are confined to the poles of the rational coefficient $a$.
If $y$ is rational, so is $f$.
Otherwise all the preceding lemmas apply.

Choose the root carrier $A$ from Lemma~\ref{nr:root-carrier}.
Its construction also shows that finite inverse singular values form a locally finite set and, except for finitely many, are attracting fixed points.
Every finite fixed point outside the finite coefficient exceptions is a zero of $y$ and is superattracting.
Thus the first three hypotheses of Theorem~\ref{thm:thin-pole-gap} hold.
The radial fixed-point obstruction in Lemma~\ref{nr:radial-roots}, together with Theorem~\ref{thm:thin-pole-gap}, excludes a thin wandering tail.
Apply Lemma~\ref{lem:carrier-discrete} to the remaining discrete tail.
Its conclusion gives local uniform escape on a nonempty open disk and, for a fixed finite repelling marking, convergence of the one-step $X$-hyperbolic derivatives to one along every orbit starting in that disk.

Put $E=(S_\C(f)\setminus A)\cap X$, $Y=X\setminus E$, and $V=f^{-1}(Y)\subset X$.
On the tail, covering invariance gives
\[
 \frac{|f'(z)|\rho_X(f(z))}{\rho_X(z)}
 =\frac{\rho_V(z)}{\rho_X(z)}
       \frac{\rho_X(f(z))}{\rho_Y(f(z))}.
\]
The second factor tends to one because $E$ is finite and $f(z)$ escapes.
The relative hyperbolic metric inequality implies that $d_X(z,X\setminus V)\to\infty$, and hence $d_X(z,\Sigma_j)\to\infty$.
Proposition~\ref{nr:mesh} forces $a(z_n)y(z_n)\to1$ along any orbit in the tail.
By Lemma~\ref{nr:algebraic-channel}, this orbit enters an invariant Fatou component, contradicting wandering.
\end{proof}

\subsection{The Riccati classes}\label{sec:riccati}
\subsubsection{The first Riccati equation and its pole gaps}\label{rr:structure}

Suppose a nonconstant meromorphic function on the plane satisfies

\[
 f'=r(z)(f-z)^2,\qquad r\in\C(z),\quad r\not\equiv0.
\]

Put $h=1/(z-f)$.
Then, as an identity of meromorphic functions,

\begin{equation}\label{rr:struct-1}
h'+h^2=r.
\end{equation}

At a finite point where $r$ is holomorphic, any pole of $h$ is simple and has residue $1$.
Indeed, a pole of order at least two makes the highest pole of $h^2$ impossible to cancel in \eqref{rr:struct-1}; for a simple pole of residue $a$, its order-two coefficient is $a^2-a$,
so $a=1$.

Let $q$ be the rational function obtained by summing the full principal parts of $h$ at the finitely many poles of $r$.
Then $h-q$ has only simple poles, all with residue one.
Consequently there is an entire function $v$, unique up to a nonzero constant, for which

\begin{equation}\label{rr:struct-2}
h-q=\frac{v'}v,\qquad
 f=z-\frac{v}{v'+qv}.
\end{equation}

For completeness, on the complement of the poles define $v=\exp\int(h-q)\,dz$.
Integrals over closed curves lie in $2\pi i\Z$, so the exponential is single-valued.
It extends with a simple zero at each ordinary pole of $h$, and extends nonvanishingly across the exceptional points because their full principal parts have been removed.

\paragraph{The regular-singular case is rational}

If $r(z)=O(z^{-2})$ at infinity, every single-valued meromorphic solution $h$ of \eqref{rr:struct-1} is rational.
Here is the ODE reason, including the monodromy point.
On an exterior universal cover let $u=\exp\int h$, so $u''=ru$.
Infinity is a regular singular point.
Analytic continuation once around infinity multiplies $u$ by a nonzero scalar because $u'/u=h$ is single-valued.
Thus $u$ spans an eigenline of local monodromy.
The Frobenius forms of solutions on a monodromy eigenline have a power factor times a convergent Laurent series at infinity, without a nontrivial logarithmic summand; in a resonant case the eigenline without logarithmic terms is selected,
or all solutions have no logarithm when the monodromy is scalar.
Therefore $u'/u$ is meromorphic at infinity.
A meromorphic function on the sphere is rational.

It follows that a transcendental R1 function necessarily has

\begin{equation}\label{rr:struct-3}
r(z)=az^d(1+O(z^{-1})),\qquad a\ne0,\quad d\ge-1.
\end{equation}

\paragraph{Uniform fixed-point normal form}

Let $c_j\to\infty$ be fixed points of a transcendental R1 function.
For all large $j$, choose $s_j$ with $s_j^2=r(c_j)$, and set

\[
 F_j(t)=s_j\bigl(f(c_j+t/s_j)-c_j\bigr).
\]

Then

\begin{equation}\label{rr:struct-4}
F_j'(t)=\frac{r(c_j+t/s_j)}{r(c_j)}(F_j(t)-t)^2,
 \qquad F_j(0)=0.
\end{equation}

Since $|c_js_j|\to\infty$, the coefficient in \eqref{rr:struct-4} tends locally uniformly to one.
Holomorphic dependence for an analytic ODE gives, on every closed disk $|t|\le b<\pi/2$,

\begin{equation}\label{rr:struct-5}
F_j(t)\longrightarrow T_1(t):=t-\tanh t.
\end{equation}

The use of ODE dependence here can be made on one disk at a time: the limiting solution is holomorphic on a neighborhood of the disk and bounded there, so the usual integral-equation perturbation argument continues the nearby solutions across that disk and gives uniform convergence.
No approximation across a pole of the limiting solution is claimed.

All these fixed points are cubic superattracting points.
More precisely,

\[
 F_j'(0)=F_j''(0)=0,\qquad F_j'''(0)=2,
\]

and $T_1(t)=t^3/3+O(t^5)$.
Consequently there exists one $b>0$, independent of all sufficiently large fixed points, such that

\begin{equation}\label{rr:struct-6}
f\left(\overline{D(c_j,b/|s_j|)}\right)
 \subset D(c_j,b/(2|s_j|)).
\end{equation}

Thus uniform invariant trapping disks are available at every exterior fixed point, at the intrinsic ODE scale $|r(c_j)|^{-1/2}$.

\subsubsection{The second Riccati equation}

Suppose

\[
 f'=r(z)(f-z)(f-\tau),\qquad r\in\C(z),\quad
 \tau\in\C,
\]

and $f\not\equiv\tau$.
Put $k=1/(f-\tau)$.
By Lemma~\ref{lem:R2-finite-order}, its exact linear equation is \eqref{rr:struct-7}.

At a point where $r$ is holomorphic, a solution of \eqref{rr:struct-7} cannot have a pole: the derivative would have a pole one order higher than either term on the right.
Hence $k$ has only finitely many poles, all among the poles of $r$.
Equivalently, a nonconstant $f$ takes the value $\tau$ only at this finite exceptional set.
This also follows from uniqueness for the original ODE, since the constant function $\tau$ is a solution.

\paragraph{The regular-singular case is rational}

If $r(z)=O(z^{-2})$, then outside a fixed disk

\[
 |r(z)(z-\tau)|\le C/|z|,\qquad |r(z)|\le C/|z|^2.
\]

On every radial ray, the integral form of \eqref{rr:struct-7} and Gronwall's inequality give $|k(z)|\le C'(1+|z|)^M$, with constants uniform in the angle (use the maximum of $|k|$ on the initial circle).
There are no exterior poles.
Subtracting the finitely many finite principal parts produces an entire function of polynomial growth, hence a polynomial.
Thus $k$, and therefore $f$, is rational.

A transcendental R2 function therefore also satisfies \eqref{rr:struct-3}.

\paragraph{Uniform fixed-point normal form}

Let $c_j\to\infty$ be fixed points and set $a_j=r(c_j)(c_j-\tau)$.
For all large $j$, $a_j\ne0$, $|a_jc_j|\to\infty$, and

\[
 F_j(t)=a_j\bigl(f(c_j+t/a_j)-c_j\bigr)
\]

satisfies

\[
 F_j'=A_j(t)(F_j-t)+B_j(t)(F_j-t)^2,\quad F_j(0)=0,
\]

where

\[
 A_j(t)=\frac{r(c_j+t/a_j)(c_j+t/a_j-\tau)}{a_j}\to1,
 \qquad B_j(t)=\frac{r(c_j+t/a_j)}{a_j^2}\to0
\]

locally uniformly.
Therefore

\begin{equation}\label{rr:struct-9}
F_j(t)\longrightarrow T_2(t):=t+1-e^t
\end{equation}

locally uniformly on every fixed disk, by the same bounded-domain ODE argument as above.
In particular $F_j'(0)=0$, $F_j''(0)=-1$; these are quadratic superattracting fixed points.
There exists a uniform small $b>0$ such that

\begin{equation}\label{rr:struct-10}
f\left(\overline{D(c_j,b/|a_j|)}\right)
 \subset D(c_j,b/(2|a_j|)).
\end{equation}

The natural exterior trapping scale is thus $|r(c_j)(c_j-\tau)|^{-1}$.

\subsubsection{Finite order and finite asymptotic values}

The second class has finite order by Lemma~\ref{lem:R2-finite-order}.
For the first class this follows directly, without a nonlinear ODE growth theorem, as follows.

In R1, the entire function $v$ in \eqref{rr:struct-2} satisfies

\begin{equation}\label{rr:struct-11}
v''+2qv'+(q'+q^2-r)v=0.
\end{equation}

Outside a fixed disk all coefficients are rational without poles and bounded in modulus by $C(1+|z|)^M$, for a finite $M\ge0$.
Apply Gronwall to the first-order system for $(v,v')$ on each radial ray.
The initial data on the initial circle have a uniform bound, so

\[
 |v(z)|+|v'(z)|\le C_1\exp(C_2|z|^{M+1}).
\]

Thus $v$ is an entire function of finite order, and \eqref{rr:struct-2} expresses $f$ as a quotient of finite-order entire functions after clearing the rational denominator of $q$.
Standard elementary inequalities for the Nevanlinna characteristic of a quotient give finite order of $f$.

For R2, Lemma~\ref{lem:R2-finite-order} gives the general order bound.
In the transcendental case \eqref{rr:struct-3}, one may take the sharper bound

\[
 |k(z)|\le C_1\exp(C_2|z|^{d+2}),\qquad \rho(k)\le d+2.
\]

To obtain this estimate uniformly in the angle, use $|k|+1$ in the scalar Gronwall inequality, and bound the homogeneous coefficient by $C|z|^{d+1}$ and the inhomogeneous coefficient by $C|z|^d$.
Subtracting the finite principal parts reduces the growth assertion to an entire function.
The more precise bound is useful for this structural description; finite order of $f$ was established before applying the common linear-ODE no-wandering argument.

\paragraph{A local inverse-ODE lemma}

Let $A(w,y)$ be holomorphic near $(a,0)$, with $A(w,0)=0$.
There cannot be a nonconstant solution curve of

\[
 \frac{dy}{dw}=A(w,y)
\]

which remains in that neighborhood, has $y\ne0$, and tends to $(a,0)$, even when its projection to the $w$-plane is a nonrectifiable or infinitely winding path.

Here is a proof that avoids an assumption on the length of that path.
The holomorphic initial-value solutions $y=\phi(w,\eta)$, $\phi(a,\eta)=\eta$, exist on a sufficiently small product
neighborhood. Because $\partial_\eta\phi(a,0)=1$, the holomorphic
inverse-function theorem gives a first integral $\Psi(w,y)$, characterized by $\Psi(w,\phi(w,\eta))=\eta$.
Uniqueness and $A(w,0)=0$ give

\[
 \Psi(w,0)=0,\qquad \{\Psi=0\}=\{y=0\}
\]

on a smaller product neighborhood.
Along a solution curve $\Psi$ is constant.
A solution curve tending to $(a,0)$ therefore has $\Psi=0$ and is contained in $y=0$, a contradiction.

\paragraph{R1 has no finite asymptotic value}

Suppose, towards a contradiction, that $z\to\infty$ along a continuous curve and $f(z)\to a\in\C$.
Write $w=f(z)$, $y=1/z$.
The equation for the inverse graph is

\begin{equation}\label{rr:struct-12}
\frac{dy}{dw}=-\frac{y^4}{r(1/y)(wy-1)^2}.
\end{equation}

By \eqref{rr:struct-3}, the right-hand side extends holomorphically near every $(a,0)$, vanishes on $y=0$, and is $O(y^{d+4})$.
Along a sufficiently late part of the proposed asymptotic curve, $f'\ne0$, so \eqref{rr:struct-12} applies.
The inverse-ODE lemma gives a contradiction.

\paragraph{R2 has no finite asymptotic value other than possibly \texorpdfstring{$\tau$}{tau}}

For a hypothetical finite asymptotic value $a\ne\tau$, the inverse graph instead satisfies

\begin{equation}\label{rr:struct-13}
\frac{dy}{dw}=\frac{y^3}{r(1/y)(1-wy)(w-\tau)}.
\end{equation}

Again the right-hand side extends holomorphically near $(a,0)$, vanishes on $y=0$, and is $O(y^{d+3})$.
The same lemma rules out this curve.
No assertion excluding $\tau$ is made.

\paragraph{A finite-residual singular carrier follows}

For either class, choose a disk containing all zeros and poles of $r$, the point $\tau$ when relevant, and all finite exceptional points.
Outside this disk, every critical point is a fixed point.
For R2 this uses the previously proved fact that a nonconstant $f$ cannot take $\tau$ at an ordinary point.
Inside a compact set there are only finitely many critical points of a nonconstant meromorphic function (including critical poles when singular values are considered on the sphere).
Thus the critical values outside a finite set are precisely the exterior fixed points.
Their set has no finite accumulation point.

Equations \eqref{rr:struct-12} and \eqref{rr:struct-13} identify all possible finite asymptotic values.
The usual covering characterization of inverse singular values as the closure of critical and asymptotic values now implies that all finite inverse singular values, except finitely many,
are exterior superattracting fixed points.

Take the closed trapping disks \eqref{rr:struct-6} or \eqref{rr:struct-10} at all sufficiently large fixed points and let $A$ be their union.
Their radii are $o(|c_j|)$, so this union is locally finite and closed in the plane.
Every disk is contained in the basin of its center, and is mapped strictly into itself; disks belonging to distinct fixed points cannot overlap because their iterates have different limits.
Hence

\begin{equation}\label{rr:struct-14}
A\subset \Fat(f),\qquad f(A)\subset A,\qquad
 S_{\C}(f)\setminus A\text{ is finite}.
\end{equation}

Every Fatou component meeting $A$ is an invariant immediate attracting basin.
Thus both Riccati equations admit the closed singular carrier required by Proposition~\ref{lem:carrier-discrete}.

\subsubsection{Pole gaps from local equations}

\begin{proposition}\label{rr:pole-gap} For a transcendental meromorphic solution of either R equation,
and for every $\epsilon>0$, there is $R_\epsilon$ such that every round annulus

\[
 \{R<|z|<(1+\epsilon)R\},\qquad R>R_\epsilon,
\]

contains a pole of $f$.

\end{proposition}
\begin{proof}
We first show that the function has arbitrarily large poles.

\paragraph{Step 1: sufficiently large poles exist}

For R1, if the entire finite-order function $v$ in \eqref{rr:struct-2} had only finitely many zeros, its Hadamard factorization would have the form $v=P e^Q$, with polynomials $P,Q$.
Then $h=v'/v+q$ and $f=z-1/h$ would be rational, a contradiction.
Thus there are arbitrarily large ordinary fixed points $c$.

At such a point choose $s^2=r(c)$.
On the local $t$-disk, solve

\[
 U_{tt}=\frac{r(c+t/s)}{r(c)}U,
 \qquad U(0)=0,\quad U_t(0)=1.
\]

The fixed-point identity gives $s(f(c+t/s)-c)=t-U/U_t$.
On every fixed compact $t$-set, the coefficient tends uniformly to one and hence $U\to\sinh t$, $U_t\to\cosh t$.
These are convergence statements about the linear ODE solutions, so they remain valid at the eventual poles of the quotient.
Rouch\'e's theorem near $t=\pi i/2$ gives a simple zero of $U_t$, at which $U\ne0$.
This is a pole $p=c+(\pi i/2+o(1))/s$ of $f$, and $|p|\to\infty$.

For R2, if $k$ had only finitely many zeros, its finite order and finitely many poles would give $k=R e^Q$ with rational $R\ne0$ and polynomial $Q$.
Equation \eqref{rr:struct-7} would imply

\[
 \left(\frac{R'}R+Q'-r(z)(z-\tau)\right)R e^Q=-r.
\]

The rational factor in parentheses is not identically zero because $r\ne0$.
Therefore $e^Q$ would be rational and $Q$ constant, implying that $f$ is rational.
This contradiction proves that $k$ has arbitrarily large zeros, i.e. $f$ has arbitrarily large poles.

\paragraph{Step 2: each large pole has two controlled neighbors}

For R1, let $p$ be any sufficiently large pole of $f$, and choose $s^2=r(p)$.
The local function $u$ with $u'/u=h$ is nonzero at $p$, and $u'(p)=0$.
Normalize as above with $U(0)=1$, $U_t(0)=0$.
Then $U\to\cosh t$, $U_t\to\sinh t$ on fixed compact sets.
Its two simple zeros near $\pm\pi i$ give poles

\begin{equation}\label{rr:struct-15}
p_\pm=p+\frac{\pm\pi i+o(1)}{\sqrt{r(p)}}.
\end{equation}

The error is uniform as $p\to\infty$.
The new poles are distinct from $p$, and are genuine because a nonzero solution of a regular second-order linear ODE cannot have its value and derivative vanish at the same point.

For R2, at a sufficiently large pole $p$ one has $k(p)=0$ and $k'(p)=-r(p)\ne0$.
Set $a_p=r(p)(p-\tau)$ and

\[
 K_p(t)=\frac{a_p}{r(p)}k(p+t/a_p).
\]

Equation \eqref{rr:struct-7} gives $K_p'\to K_p-1$, with $K_p(0)=0$, uniformly in its coefficients on fixed disks.
Thus $K_p\to1-e^t$.
The simple zeros near $\pm2\pi i$ give poles

\begin{equation}\label{rr:struct-16}
p_\pm=p+\frac{\pm2\pi i+o(1)}{r(p)(p-\tau)}.
\end{equation}

\paragraph{Step 3: a pole chain escapes and has vanishing relative gaps}

For R2, write $r(z)\sim az^d$, $d\ge-1$, and put

\[
 Z=P_0(p):=\frac{a}{d+2}p^{d+2}.
\]

Taylor's theorem and \eqref{rr:struct-16} give

\begin{equation}\label{rr:struct-17}
P_0(p_\pm)-P_0(p)=\pm2\pi i+o(1).
\end{equation}

For R1, use the exterior double cover $p=w^2$.
Choose the holomorphic square root of $r(w^2)$ on a sufficiently large exterior annulus with asymptotic value $\sqrt a\,w^d$, and set

\[
 Z=P_0(w):=\frac{2\sqrt a}{d+2}w^{d+2}.
\]

Lift the two neighboring poles in \eqref{rr:struct-15} to the points $w_\pm$ near $w$.
Then

\begin{equation}\label{rr:struct-18}
w_\pm-w=\frac{\pm\pi i+o(1)}{2w\sqrt{r(w^2)}},
 \qquad P_0(w_\pm)-P_0(w)=\pm\pi i+o(1).
\end{equation}

The square root exists because the exterior winding number of $r(w^2)$ is even.
The polynomial $P_0$, rather than an exact primitive, is sufficient: its derivative divided by $2w\sqrt{r(w^2)}$ tends to one.
This avoids a logarithmic monodromy issue for a primitive.

Both cases now have the same form.
A sufficiently large pole coordinate $Z$ has pole successors $Z_\pm=Z\pm iT+e_\pm$, where $T=\pi$ or $2\pi$ and $e_\pm\to0$ as $|Z|\to\infty$.
Choose a threshold beyond which $|e_\pm|\le\eta<T/\sqrt2$, a starting pole with $|Z_0|$ larger than $\sqrt2$ times that threshold, and a fixed sign $\sigma\in\{1,-1\}$ with $\sigma\operatorname{Im}Z_0\ge0$.
Inductively take the successor of sign $\sigma$.
As long as the chain is outside the threshold,

\[
 Z_n=Z_0+\sigma iTn+E_n,\qquad |E_n|\le\eta n.
\]

Since $\sigma\operatorname{Im}Z_0\ge0$,

\[
 |Z_0+\sigma iTn|\ge\sqrt{|Z_0|^2+T^2n^2}
 \ge\frac{|Z_0|+Tn}{\sqrt2}.
\]

Therefore

\begin{equation}\label{rr:struct-19}
|Z_n|\ge\frac{|Z_0|}{\sqrt2}
       +\left(\frac T{\sqrt2}-\eta\right)n.
\end{equation}

This estimate closes the induction, keeps the full chain outside the threshold, and proves $|Z_n|\to\infty$.
The increments are bounded, so $|Z_{n+1}|/|Z_n|\to1$.
Since $P_0$ is a monomial, in either case the projected poles satisfy

\begin{equation}\label{rr:struct-20}
|p_n|\to\infty,\qquad |p_{n+1}|/|p_n|\to1.
\end{equation}

Given a large radius $R$, choose the first index with $|p_n|>R$.
Then $|p_{n-1}|\le R$, and \eqref{rr:struct-20} implies $|p_n|<(1+\epsilon)R$ for all sufficiently large $R$.
\end{proof}

\subsection{Exterior expansion and capture for the first Riccati equation}\label{rr:exterior}

Write \[ r(z)=az^d(1+O(z^{-1})),\quad a\ne0,\quad d\ge-1, \qquad s(z)^2=r(z).
\] All square roots below are taken on the indicated local or sectorial chart.
The unordered set of exceptional normalized displacements is $\{1,-1\}$, so the first assertion is independent of the choice of square root.

\subsubsection{Choose the trapping disks with a collar}

At an exterior fixed point $c$, normalized local dynamics converges uniformly to \[ T(t)=t-\tanh t.
\] Choose a sufficiently small fixed $b>0$.
Because $T(t)=O(t^3)$, after discarding finitely many fixed points we have the uniform strict inclusions \begin{equation}\label{rr:geo-1} f\bigl(\overline{D(c,2b/|s(c)|)}\bigr) \subset D(c,b/(2|s(c)|)).
\end{equation} Let \[
A=\bigcup_c\overline{D(c,b/|s(c)|)}, \qquad A^+=\bigcup_cD(c,2b/|s(c)|), \] where the union runs over all sufficiently large fixed points.
The R1 singular-value calculation makes $S_{\C}(f)\setminus A$ finite.
Fix any finite repelling marking $P$ containing at least three points, and put \[ X=\C\setminus(A\cup P),\quad E=(S_{\C}(f)\setminus A)\cap X, \quad Y=X\setminus E, \quad V=f^{-1}(Y)\subset X.
\] If $f(z)\in X$ and $|z|$ is sufficiently large, \eqref{rr:geo-1} implies $z\notin A^+$.
This uniform collar exclusion is needed in the metric argument; merely knowing $z\notin A$ would not suffice.

\subsubsection{A bounded hyperbolic distance to a pole}

\begin{lemma}\label{rr:pole-net} Fix $\epsilon>0$. There are constants $L,C,R$ such that, whenever
\[
 |z|>R,\quad z,f(z)\in X,\quad
 \operatorname{dist}_{\#}\bigl(s(z)(f(z)-z),\{1,-1\}\bigr)\ge\epsilon,
\]
one can join $z$ to a pole $p\in X\setminus V$ of $f$ by a path in $X$ of hyperbolic length at most $C$.
The path is contained in $D(z,L/|s(z)|)$.
In particular, its Euclidean distance from the origin tends to infinity with $|z|$.

\end{lemma}
\begin{proof} Put $D=s(z)(f(z)-z)$. The normalized local solutions are most conveniently described by
the linear equation.
In the coordinate $t=s(z)(\zeta-z)$, choose the projective initial data
\[
 (U(0),U_t(0))=(D,-1)
\]
for finite $D$, with $(1,0)$ at $D=\infty$, normalizing their Euclidean norm when necessary.
Then
\begin{equation}\label{rr:geo-2}
U_{tt}=\frac{r(z+t/s(z))}{r(z)}U,
 \qquad s(z)(f(z+t/s(z))-z)=t-U/U_t.
\end{equation}
The coefficient converges uniformly to one on every fixed disk.
For normalized initial data in a compact subset of the projective line, the linear solutions and their derivatives converge uniformly on every fixed disk to the solutions with coefficient one.
These limiting solutions have normalized displacement
\begin{equation}\label{rr:geo-3}
D(t)=\frac{D-\tanh t}{1-D\tanh t}.
\end{equation}
When $D$ stays in a compact subset of $\sphere\setminus\{1,-1\}$, a pole of \eqref{rr:geo-3} can be selected at uniformly bounded distance from zero: it satisfies
\[
 e^{2t}=\frac{D+1}{D-1}.
\]
The quotient on the right is bounded above and away from zero on this compact set.
A logarithm with imaginary part in $[-\pi,\pi]$ selects a bounded $t$.
These poles are simple.
The zeros of \eqref{rr:geo-3}, which correspond to fixed points, form the translate by $\pi i/2$ of this pole lattice, and each of the two lattices has spacing $\pi i$.
Thus poles and zeros have a fixed positive separation.

Uniform convergence of the linear solutions, Rouch\'e's theorem, and compactness of the initial-data set now give the following facts on one fixed sufficiently large coordinate disk:
there is a pole of the original $f$; the relevant fixed-point centers converge, with their multiplicities, to the finitely many simple zeros of the limiting $U$; their separations are bounded below;
and the selected pole remains a definite distance from all of these centers.
Use a slightly larger coordinate disk for these convergence statements so that zeros on a proposed path's boundary cause no issue.

The disks defining $A$ have coordinate radii $b(1+o(1))$ in this fixed local window, because $s(c)/s(z)\to1$ for $|c-z|=O(1/|s(z)|)$.
No remotely centered trapping disk can meet the window: its radius is $o(|c|)$, which first forces $|c|/|z|\to1$, after which the same local scale comparison applies.
The collar exclusion $z\notin A^+$ keeps the coordinate origin at a definite distance from every trapping disk.

Choose $b$ smaller than one tenth of the limiting root separation.
In the coordinate plane connect zero to the selected pole by a path of bounded length, detouring along circles of radius, say, $3b/2$ around the finitely many root disks it encounters.
Uniform separation and a uniform bound on their number give a bound independent of $z,D$, and the path has a fixed positive clearance from the disks of radius $b(1+o(1))$.
The finite marking $P$ is absent from this window for large $|z|$.
After returning to the original plane, the path has length $O(1/|s(z)|)$ and clearance $\ge c/|s(z)|$ from $A\cup P$.
The elementary disk comparison
\[
 \rho_X(\zeta)\le\frac1{\operatorname{dist}(\zeta,\C\setminus X)}
\]
therefore bounds its hyperbolic length uniformly.
\end{proof}
\paragraph{Expansion consequence.}

The selected pole lies in $X\setminus V$.
The relative hyperbolic-metric estimate at bounded distance from a removed point gives \begin{equation}\label{rr:geo-4} \frac{\rho_V(z)}{\rho_X(z)}\ge1+\delta_\epsilon>1.
\end{equation}
For $f(z)\in Y=X\setminus E$, and hence $z\in V$, covering invariance gives the exact identity
\begin{equation}\label{rr:geo-5}
\frac{|f'(z)|\rho_X(f(z))}{\rho_X(z)}
 =\frac{\rho_V(z)/\rho_X(z)}{\rho_Y(f(z))/\rho_X(f(z))}.
\end{equation}
Along an escaping orbit this condition holds eventually because $E$ is finite.
For the finite set $E$, the denominator tends to one as $|f(z)|\to\infty$, by the finite-puncture comparison.
This comparison can also be seen directly: the distance in $X$ from an escaping point to a fixed finite set tends to infinity, by comparison with the complete metric on $\C\setminus P$.
If that distance exceeds $R$, restrict a universal covering of $X$ to its centered hyperbolic disk of radius $R$.
Its image avoids $E$, so Schwarz--Pick bounds $\rho_{X\setminus E}/\rho_X$ at the center by the ratio of the centered disk metrics, which tends to one as $R\to\infty$.
The lower bound is one by monotonicity.
Thus an escaping orbit $z_n$ with one-step derivatives in \eqref{rr:geo-5} tending to one satisfies
\begin{equation}\label{rr:geo-6}
\operatorname{dist}_{\#}\bigl(s(z_n)(f(z_n)-z_n),\{1,-1\}\bigr)
 \longrightarrow0.
\end{equation}

The relative lower bound \eqref{rr:geo-4} is equally elementary: lift $X$ to the unit disk centered above $z$, lift the path to the missing pole, and compare the lifted component of $V$ with the disk minus that lifted point.
At bounded hyperbolic distance, the density of a once-punctured disk exceeds the disk density by a uniform factor greater than one.
Thus neither metric estimate is a new dynamical assumption about $f$.

\subsubsection{Sign locking and entry into one sector}

Continue $s$ along the short orbit segments.
Put $D_n=s(z_n)(z_{n+1}-z_n)$.
From \eqref{rr:geo-6}, $D_n$ is close to one of $1,-1$, and \[ |z_{n+1}-z_n|=O(1/|s(z_n)|)=o(|z_n|).
\] In the normalized local equation \eqref{rr:geo-3}, the constant solutions $D=1$ and $D=-1$ remain constant.
Applying \eqref{rr:geo-2} on a fixed disk containing $t=D_n$ therefore shows that $D_{n+1}$ is close to the same sign as $D_n$; also $s(z_{n+1})/s(z_n)\to1$ under this continuation.
Consequently one sign $\sigma\in\{1,-1\}$ is fixed for all sufficiently large $n$, and
\begin{equation}\label{rr:geo-7}
D_n\to\sigma.
\end{equation}

Lift the short orbit segments to $z=w^2$, and choose the sign of $\sqrt a$ so that the exterior branch $s(w^2)\sim\sqrt a\,w^d$ agrees with the one continued along the orbit.
In the leading monomial coordinate \[ P_0(w)=\frac{2\sqrt a}{d+2}w^{d+2} \] one has \begin{equation}\label{rr:geo-8} P_0(w_{n+1})-P_0(w_n)=D_n+o(1)\to\sigma, \qquad P_0(w_n)/n\to\sigma.
\end{equation} Because $w_{n+1}/w_n\to1$, the orbit eventually remains
near one of the finitely many inverse rays of this monomial; it cannot switch between these separated rays by such small relative steps.
On the corresponding sector in the original $z$-plane choose a primitive \[ P(z)=\int s(z)\,dz, \qquad P(z)\sim P_0(w).
\] After narrowing the sector, $P$ has a holomorphic inverse on a sufficiently narrow cone about the ray $\sigma\R_+$.
This follows directly by using $z^{(d+2)/2}$ as a preliminary coordinate: the derivative of $P$ with respect to that coordinate tends uniformly to a nonzero constant on a slightly wider sector,
giving univalence on a convex narrow cone; Rouch\'e's theorem supplies the inverse on an inner cone.
Equation \eqref{rr:geo-8} puts the orbit eventually in every fixed inner cone of this type.

\subsubsection{An elementary one-cone WKB lemma}

On this simply connected sector choose a solution $u$ of $u''=ru$ whose initial quotient agrees with $u'/u=1/(z-f)$.
The map $z-u/u'$ solves the same Riccati equation as the actual $f$.
Local uniqueness and meromorphic continuation identify them throughout the sector, including across the isolated zeros of $u$.
Thus this choice does not assume that $1/(z-f)$ is pole-free.
Write
\[
 u(z)=s(z)^{-1/2}v(P(z)).
\] A direct calculation gives \begin{equation}\label{rr:geo-9}
v_{PP}=(1+\eta(P))v,
 \quad
 \eta=\frac{s''}{2s^3}-\frac{3(s')^2}{4s^4}
 =O(P^{-2}).
\end{equation} Set $Q=\sigma P$; the same form of equation holds, with
$\eta(\sigma Q)=O(Q^{-2})$, on a right cone $\operatorname{Re}Q>R$,
$|\operatorname{Im}Q|<\beta\operatorname{Re}Q$.

There are independent solutions, uniformly on each narrower cone,
\begin{equation}\label{rr:geo-10}
v_-(Q)=e^{-Q}(1+O(Q^{-1})),\qquad
 v_+(Q)=e^Q(1+O(Q^{-1})),
\end{equation} whose logarithmic derivatives tend respectively to
$-1,+1$.

Here is a direct construction. Seek $v_-=e^{-Q}w(Q)$. The Volterra
equation is \begin{equation}\label{rr:geo-11}
w(Q)=1+\frac12\int_0^\infty(1-e^{-2t})
       \eta(\sigma(Q+t))w(Q+t)\,dt.
\end{equation} Horizontal rays stay in the cone and
$\int_0^\infty |\eta(\sigma(Q+t))|dt\le C/\operatorname{Re}Q$. For
large $R$, the integral operator is a contraction on bounded
holomorphic functions. Its fixed point has
$w=1+O(1/\operatorname{Re}Q)$, and Cauchy's estimate on narrower cones
gives $w'=O(Q^{-2})$. Differentiating the convergent integral
equation, or the corresponding integral with kernel $\sinh(t-Q)$,
verifies \eqref{rr:geo-9}.

On a still smaller right cone $v_-$ is nonzero. Reduction of order
gives \[
 v_+(Q)=2v_-(Q)\int_{Q_*}^Q v_-(t)^{-2}\,dt.
\] The integral is path-independent there. For its asymptotic estimate
take the straight segment from a fixed interior $Q_*$ to $Q$, in a
narrower convex cone. Its slope is uniformly bounded.  Writing
$v_-^{-2}=e^{2t}w(t)^{-2}$ and using $(w^{-2})'=O(t^{-2})$,
integration by parts gives
\[
 \int_{Q_*}^Qv_-(t)^{-2}\,dt
 =\tfrac12e^{2Q}w(Q)^{-2}
   +O\bigl(e^{2\operatorname{Re}Q}/|Q|^2\bigr)+O(1).
\]
The $O(1)$ term comes from the fixed lower endpoint.
This proves the second asymptotic in \eqref{rr:geo-10} and its differentiated version.
The additive constant in the primitive only adds a multiple of $v_-$ and does not change the stated growing asymptotic.
The Wronskian is nonzero by the reduction-of-order formula.
This establishes \eqref{rr:geo-10}.

\subsubsection{The orbit selects the subdominant solution and enters a Baker component}

Write $v=A v_++B v_-$ in this one cone.
If $A\ne0$, then $v_Q/v\to1$ uniformly on narrower right cones.
Since \[ \frac{u'}u=-\frac{s'}{2s}+\sigma s\frac{v_Q}v, \qquad s(f-z)=-\frac{s}{u'/u}, \] this would imply $D_n\to-\sigma$, contradicting \eqref{rr:geo-7}.
Therefore $A=0$: the actual solution is the subdominant one throughout this cone.

Consequently, uniformly on a narrower cone, \[ \frac{u'}u=-\sigma s\bigl(1+O(Q^{-1})\bigr),\qquad f(z)-z=\frac{\sigma}{s(z)}\bigl(1+O(Q^{-1})\bigr).
\] Taylor expansion of $Q(z)=\sigma P(z)$, with $Q'=\sigma s$, gives
\begin{equation}\label{rr:geo-12}
Q(f(z))=Q(z)+1+O(Q(z)^{-1}).
\end{equation} Choose a sufficiently large $R_1$ and a fixed narrow
cone \[ \mathcal C=\{\operatorname{Re}Q>R_1, |\operatorname{Im}Q|<\beta_1(\operatorname{Re}Q-R_1)\} \] contained in the region where \eqref{rr:geo-12} holds.
The translation estimate is first established in an inner cone whose directions lie strictly inside a wider cone on which the inverse coordinate exists.
Since $f(z)-z=O(1/s(z))=o(|z|)$, its image stays in that wider chart; the angular margin in the $Q$ coordinate is comparable to $|Q|$, whereas the displacement is bounded.
If the error bound is $C/|Q|$, choose $R_1$ so large that $\beta_1(1-C/R_1)>C/R_1$.
Then \eqref{rr:geo-12} maps $\mathcal C$ into itself and increases the real part by at least $1/2$.
The inverse image of $\mathcal C$ under the chosen branch of $Q$ is therefore an open forward invariant set on which iterates tend locally uniformly to infinity.
It belongs to an invariant Baker Fatou component.

The orbit from \eqref{rr:geo-8} eventually lies in this cone.
Hence an escaping orbit in a wandering Fatou component cannot have the derivative-one property \eqref{rr:geo-5}.

\Needspace{6\baselineskip}
\begin{theorem}\label{rr:R1-no-wandering}
Every meromorphic function satisfying
\[
 f'(z)=r(z)(f(z)-z)^2,\qquad r\in\C(z),
\]
has no wandering Fatou component.
\end{theorem}
\begin{proof}
The rational case follows from Theorem~\ref{core:sullivan}\textup{(i)} and the elementary degree-at-most-one cases discussed in Section~\ref{sec:q9}, so assume that $f$ is transcendental.
Section~\ref{rr:structure} gives a closed forward invariant union $A$ of attracting trapping disks, with finitely many residual finite singular values.
The finite singular set is locally finite; outside finitely many exceptions its elements are superattracting fixed points.
All non-attracting finite fixed points are confined to the poles of $r$, hence are finite in number.
Proposition~\ref{rr:pole-gap} verifies the annular obstruction required by Theorem~\ref{thm:thin-pole-gap}.
That theorem excludes the thin covering-group case for a hypothetical wandering tail.

In the discrete covering-group case, Lemma~\ref{lem:carrier-discrete} gives an escaping orbit in the wandering component and a finite repelling marking for which the one-step hyperbolic derivatives tend to one.
The marking selected in that lemma already contains a fixed triple.
The expansion and capture argument of Section~\ref{rr:exterior} puts this orbit in an invariant Baker component, contradicting wandering.
This excludes both cases.
\end{proof}

\subsection{Completion of the meromorphic no-wandering theorem}
\begin{proof}[Proof of Theorem~\ref{thm:main}\textup{(ii)}]
Theorem~\ref{thm:finite-type} treats $\SSS$, Theorem~\ref{thm:F} treats $\F$, Theorem~\ref{nr:no-wandering} treats $\N$ and, using Lemma~\ref{lem:R2-finite-order} to supply its finite-order hypothesis,
all of $\RR_2$, and Theorem~\ref{rr:R1-no-wandering} treats $\RR_1$.
\end{proof}

\subsection{An entire-function supplement: compact singular sets}\label{sec:transcendental}

\begin{proof}[Proof of Theorem~\ref{core:singular-accumulation}]
Since $f\in\mathcal B$, every Fatou component is simply connected \cite[Proposition~3, p.~993]{EL}.
The compact set $S(f)'$ meets only finitely many Fatou components.
Using Riemann maps in these components, choose finite unions of pairwise disjoint closed smooth Jordan disks
\[
 S(f)'\subset\operatorname{Int}K,\qquad
 K\Subset\operatorname{Int}N\Subset\Fat(f).
\]
If $S(f)'=\varnothing$, take $K=N=\varnothing$.
The set $E=S(f)\setminus\operatorname{Int}K$ is finite; otherwise compactness would force $S(f)'$ to meet $\C\setminus\operatorname{Int}K$.

Suppose a wandering component exists.
Since $N\cup E$ meets only finitely many Fatou components, we may discard finitely many iterates and obtain pairwise distinct components $U_n$ satisfying
\[
 f(U_n)\subset U_{n+1},\qquad
 U_n\cap(N\cup E)=\varnothing\quad(n\ge0).
\]
Each $U_n$ is a component of $f^{-1}(U_{n+1})$, by complete invariance of the Fatou set.
Since $U_{n+1}$ avoids $S(f)$, the restriction $f:U_n\to U_{n+1}$ is a covering \cite[Section~3.2]{Sixsmith}, hence a conformal isomorphism.
Thus, for $W=\bigsqcup_{n\ge0}U_n$, $f:W\to W\setminus U_0$ is a conformal isomorphism.

Choose finite unions $P_j\subset\Jul(f)$ of repelling periodic orbits such that $P_j\subseteq P_{j+1}$, $|P_1|\ge3$, and $\bigcup_jP_j$ is dense in $\Jul(f)$ \cite[Theorem~4]{Bergweiler}.
Set $X_j=\C\setminus P_j$ and $Y_j=X_j\setminus S(f)$, and define measures by
\[
 \dd\mu_j=\dd\alpha_{X_j}\quad\text{on }X_j,\qquad
 \dd\nu_j=\dd\alpha_{Y_j}-\dd\alpha_{X_j}\ge0
 \quad\text{on }Y_j.
\]
Since $W$ avoids $N\cup E$ and $S(f)\subset K\cup E$, domain monotonicity gives
\begin{align}\label{eq:trans-cost}
 \nu_j(W)
 &\le\int_W(\dd\alpha_{X_j\setminus K}-\dd\alpha_{X_j})\notag\\
 &\quad+\int_W(\dd\alpha_{X_j\setminus(K\cup E)}
                         -\dd\alpha_{X_j\setminus K})
 \le C+|E|.
\end{align}
Here Lemma~\ref{core:holes}, with punctures $P_j\cup\{\infty\}$, bounds the first integral uniformly in $j$, because a fixed neighborhood of $N\Subset\Fat(f)$ avoids all these punctures.
Lemma~\ref{core:puncture}, applied to $E\cap(X_j\setminus K)$, bounds the second.

The periodic-orbit construction gives $f(P_j)=P_j$, hence $f^{-1}(Y_j)\subset X_j$.
Since $Y_j$ avoids $S(f)$, $f:f^{-1}(Y_j)\to Y_j$ is a covering.
Covering invariance and domain monotonicity from Section~\ref{sec:area} give, respectively,
\[
 f^*(\dd\mu_j+\dd\nu_j)
 =\dd\alpha_{f^{-1}(Y_j)}\ge\dd\mu_j
 \quad\text{on }f^{-1}(Y_j).
\]
Since $W\subset\Fat(f)\setminus S(f)$ and $P_j\subset\Jul(f)$, we have $f(W)=W\setminus U_0\subset W\subset Y_j$, so $W\subset f^{-1}(Y_j)$.
Integrating over $W$ and changing variables under $f:W\to W\setminus U_0$ yields
\begin{align*}
 \mu_j(U_0)+\mu_j(W\setminus U_0)
 &=\mu_j(W)\\
 &\le\mu_j(W\setminus U_0)+\nu_j(W\setminus U_0).
\end{align*}
As $\mu_j(X_j)=|P_j|-1<\infty$, we may subtract the finite term $\mu_j(W\setminus U_0)$.
Nonnegativity of $\nu_j$ and \eqref{eq:trans-cost} then give
\[
 \mu_j(U_0)\le\nu_j(W\setminus U_0)\le\nu_j(W)\le C+|E|.
\]
But Lemma~\ref{core:kernel}, with punctures $P_j\cup\{\infty\}$, and monotone convergence give
\[
 \mu_j(U_0)\uparrow\alpha_{U_0}(U_0)=\infty,
\]
a contradiction.
\end{proof}

\begin{corollary}[Eremenko--Lyubich; Goldberg--Keen]\label{core:el}
A transcendental entire function with finitely many singular values has no wandering Fatou components \cite[Theorem~3]{EL} and \cite[Theorem~4.2]{GoldbergKeen}.
\end{corollary}
\begin{proof}
A finite singular set has no accumulation points, so Theorem~\ref{core:singular-accumulation} applies.
\end{proof}

The hypothesis also permits infinitely many singular values in a rotation domain, as the following example shows.

\begin{samepage}
\begin{example}\label{example:sine}
Let $\lambda=e^{2\pi i\theta}$, where $\theta$ is Diophantine irrational, and define the entire function $h_\lambda:\C\to\C$ by
\[
 h_\lambda(z)=\lambda\frac{\sin^2z}{z},\qquad h_\lambda(0)=0.
\]
Then $h_\lambda\in\mathcal B\setminus\mathcal S$, its only finite asymptotic value is zero, and
\begin{equation}\label{eq:sine-singular}
 S(h_\lambda)=\{0\}\cup
 \left\{\frac{4\lambda x}{1+4x^2}:
 x\in\R\setminus\{0\},\ \tan x=2x\right\},
 \qquad S(h_\lambda)'=\{0\}.
\end{equation}
The origin is a Siegel fixed point \cite{Siegel}.
Hence $h_\lambda$ has no wandering Fatou components by Theorem~\ref{core:singular-accumulation}.
\end{example}
\end{samepage}

The derivative vanishes at the nonzero zeros of $\sin z$ and at the nonzero solutions of $2z\cos z=\sin z$.
Every solution of the latter equation is real.
Indeed, for $u(t)=\sin(zt)$ the equation is the Robin condition $u'(1)=u(1)/2$, and
\[
 z^2\int_0^1|u|^2\,dt
 =\int_0^1|u'|^2\,dt-\tfrac12|u(1)|^2>0
\]
by $u(0)=0$ and Cauchy--Schwarz.
Its nonzero critical values are $4\lambda x/(1+4x^2)$ and tend to zero as $|x|\to\infty$.
There is no nonzero finite asymptotic value: along a path on which $h_\lambda$ tends to such a value, the imaginary part must tend to $+\infty$ or $-\infty$, since $\sin z$ is bounded on horizontal strips.
In the respective half-planes, $h_\lambda(z)=-\lambda e^{\mp2iz}(1+o(1))/(4z)$.
The modulus of the limit forces $2|\operatorname{Im}z|-\log|z|=O(1)$, so $|\operatorname{Re}z|\to\infty$ and its sign is eventually fixed.
A continuous logarithm then forces $\mp2\operatorname{Re}z-\arg z$ to remain bounded, a contradiction.
Along the positive real axis $h_\lambda(z)\to0$, proving the asymptotic-value assertion.

The function-theoretic precedent is the classical example $\sin z/z\in\mathcal B\setminus\mathcal S$ \cite[Section~4.2]{Sixsmith}.
The maps in Example~\ref{example:sine} are neither geometrically finite \cite[Proposition~5.3]{ARS}, postcritically separated \cite[Lemma~2.6]{PS}, nor topologically hyperbolic \cite[Definition~1.2]{BFJK}.
They do not satisfy the uniformly escaping singular-set hypothesis of \cite[Theorem~1.2]{MBRG}, and no affine conjugate $g$ satisfies the real entire hypothesis $g(\R)\cup S(g)\subset\R$ of \cite[Theorem~1.3]{MBRG}.

\section{Wandering orbits: proof of Theorem~\ref{q8r2:headline}}\label{sec:q8}
We now prove the meromorphic wandering-orbit assertion stated in Theorem~\ref{q8r2:headline}.
Prochorov, Rempe and Waterman \cite[Theorem~1.2]{PRW} recently proved the corresponding conclusion for transcendental entire functions, using a hyperbolic-area argument inspired by the first version of the present paper.
The wandering-component assertion in their local meromorphic theorem \cite[Theorem~1.3(1)]{PRW} assumes simple connectivity of the relevant trapped components.
Here we prove the meromorphic conclusion without a restriction on the connectivity of the Fatou components or on the inverse singular values.
The proof uses local proper charts and a truncation estimate, together with an injectivity-radius argument for multiply connected wandering tails.

\subsection{A local area estimate}\label{q8loc:section}
For a related uniform pullback-area estimate using local proper charts and a subharmonic cutoff, see \cite[Section~3]{PRW}.
The following localization estimate is used only on bounded windows.
Its bound is uniform over arbitrary inner boundaries and finite Julia markings.

\begin{lemma}\label{lem:local-truncation}
Fix a finite set $P_0\subset\C$ of at least three distinct points and put $\Omega_0=\C\setminus P_0$.
Fix bounded open sets $V_0\Subset V_1\Subset V_2$, which need not be connected, and a smooth cutoff $\chi\in C_c^\infty(V_1)$, with $0\le\chi\le1$ and $\chi=1$ on $V_0$.
Let

\[
 \delta_0=\inf\{d_{\Omega_0}(z,w):z\in\Omega_0\cap\overline{V_1},\
 w\in\Omega_0\cap\partial V_2,\ z,w\text{ in the same component}\}.
\]

Empty infima are interpreted as $+\infty$.
We have $\delta_0>0$: every such path has a subpath in the bounded region $\overline{V_2}$ of Euclidean length at least $\operatorname{dist}(\overline{V_1},\C\setminus V_2)>0$; the background metric has a positive lower bound on this bounded region away from its finitely many punctures,
and tends to infinity at the punctures.

For a hyperbolic open set $\Omega\subset\Omega_0$, put $Y=\Omega\cap V_2$, taking metrics componentwise.
Then

\begin{equation}\label{q8loc:eq-1}
\int_{\Omega\cap V_0}(d\alpha_Y-d\alpha_\Omega)
 \le C(V_0,V_1,V_2;\Omega_0,\chi)
 :=\frac{\log\coth\delta_0}{2\pi}\,\|\Delta\chi\|_{L^1(\C)}.
\end{equation}

The bound is independent of the inner boundaries of $\Omega$.
\end{lemma}
\begin{proof}
On $Y$, write $w=\log(\rho_Y/\rho_\Omega)\ge0$.
The curvature equation gives

\begin{equation}\label{q8loc:eq-2}
\Delta w=4(\rho_Y^2-\rho_\Omega^2)\ge0.
\end{equation}

For $z\in\Omega\cap V_1$, a centered hyperbolic disk of radius $\delta_0$ in a universal covering of its component of $\Omega$ maps
into $V_2$. Indeed, any path reaching $\partial V_2$ would have
length at least $\delta_0$, by metric monotonicity and the definition of $\delta_0$.
Schwarz--Pick on this restricted covering disk therefore gives

\begin{equation}\label{q8loc:eq-3}
1\le\frac{\rho_Y(z)}{\rho_\Omega(z)}\le\coth\delta_0,
 \qquad 0\le w(z)\le M:=\log\coth\delta_0.
\end{equation}

Here a hyperbolic disk of radius $\delta_0$ corresponds to the Euclidean disk of radius $\tanh\delta_0$, explaining the factor $\coth\delta_0$.

We also need to justify extension across the possibly irregular inner
boundary. If $a\in\partial\Omega\cap V_2$, choose two distinct points
$b,c\in P_0\setminus\{a\}$.
The metric of every component of $\Omega$ dominates the metric on $\sphere\setminus\{a,b,c\}$.
As $z\to a$, the distance
in the latter metric from $z$ to $\partial V_2$ tends uniformly to
infinity, by its cusp at $a$.
Restricting the covering disk as above, now with a radius tending to infinity, shows $w(z)\to0$.
The same comparison works even if a sequence approaches $a$ through different components of $\Omega$.

Extend $w$ by zero on $V_2\setminus\Omega$.
The resulting $\widetilde w$ is subharmonic.
One direct justification is to extend $\max\{w-\varepsilon,0\}$ by zero locally.
At every inner boundary point this truncation vanishes in a neighborhood, by the preceding boundary limit; the subharmonic pasting lemma applies.
Now let $\varepsilon\downarrow0$.
The distributional Laplacian of $\widetilde w$ is a positive measure; on $Y$ it is the density in \eqref{q8loc:eq-2}, while any additional boundary contribution is nonnegative.
Hence

\[
 \begin{aligned}
 \int_{\Omega\cap V_0}(d\alpha_Y-d\alpha_\Omega)
 &\le\frac1{2\pi}\int\chi\,d(\Delta\widetilde w)\\
 &=\frac1{2\pi}\int\widetilde w\,\Delta\chi\,dA\\
 &\le\frac{M}{2\pi}\|\Delta\chi\|_{L^1}.
 \end{aligned}
\]

This proves \eqref{q8loc:eq-1}.
\end{proof}
\subsection{Bounded wandering orbits force diverging injectivity radii}
\label{q8r2:inj-section}

All hyperbolic metrics in this section have curvature $-4$.
Let $f$ be transcendental meromorphic, let $U_n$ be the distinct full forward Fatou components of a wandering domain $U_0$, and fix $z_0\in U_0$ with bounded orbit $z_n=f^n(z_0)$.
Put $K=\omega(z_0)\Subset\C$.
No assumption on singular values, or on the full restrictions being coverings, is made.

\begin{lemma}
\label{q8r2:cover-disks}
Choose any universal covers $\pi_n:\D\to U_n$ with $\pi_n(0)=z_n$.
For every fixed $0<r<1$,
\[
 \operatorname{diam}\pi_n(\overline{D_r})\longrightarrow0,
 \qquad
 \sup_{w\in\pi_n(\overline{D_r})}\operatorname{dist}(w,K)
 \longrightarrow0.
\]
The diameter and distance are Euclidean.
In particular, these images are uniformly bounded for all sufficiently large $n$.
\end{lemma}
\begin{proof}
All covers omit the same three finite Julia points, so they form a normal family of meromorphic maps to the sphere.
A nonconstant subsequential limit is impossible: choose a small source disk where that limit is finite and univalent, and apply Rouch\'e's theorem on its boundary.
One fixed target disk then belongs to $\pi_n(\D)\subset U_n$ along an infinite subsequence, contradicting pairwise disjointness of the $U_n$.
Every subsequential limit is consequently constant.
Evaluation at zero shows that its value lies in $K$, hence is finite.
Spherical convergence to a finite constant is Euclidean uniform convergence on compact source sets.
If either asserted conclusion failed, normality would provide a contradictory subsequence.
\end{proof}

\begin{lemma}
\label{q8r2:hull-fill}
There is a bounded open neighborhood $W$ of $K$ whose closure contains no pole of $f$.
Suppose $\gamma\subset U_n$ is a Jordan curve and, for every $j\ge0$, the polynomial hull
\[
 H_j:=\widehat{f^j(\gamma)}
\]
is contained in $W$.
Then the bounded side of $\gamma$ lies in $U_n$; in particular, $\gamma$ is null-homotopic in $U_n$.
\end{lemma}
\begin{proof}
The compact set $K$ contains no pole: otherwise a subsequence $z_{n_k}$ converging to a pole would give $z_{n_k+1}\to\infty$, contrary to boundedness.
Since the poles are discrete, choose the bounded $W$ with pole-free closure.

For a compact plane set $E$, its polynomial hull is $E$ together with the bounded components of $\C\setminus E$.
If $g$ is holomorphic on a neighborhood of $\widehat E$, then
\[
 g(\widehat E)\subset\widehat{g(E)}.
\]
Indeed, for each polynomial $p$, the maximum principle on every bounded complementary component of $E$ gives $|p(g(w))|\le\max_E|p\circ g|$ for $w\in\widehat E$; the defining polynomial inequalities characterize the target hull.

Apply this with $g=f$ and $E=f^j(\gamma)$.
The hypothesis places $H_j$ in a pole-free neighborhood and gives $f(H_j)\subset H_{j+1}$.
Thus all iterates are holomorphic on the bounded side $B$ of $\gamma$ and take values in the fixed bounded set $W$.
They form a normal family there, so $B\subset \Fat(f)$.
A collar on the inner side of $\gamma$ lies in $U_n$, and $B$ is connected; maximality of the full Fatou component therefore gives $B\subset U_n$.
Since $\gamma$ is Jordan, this proves the final assertion.
\end{proof}

\begin{theorem}
\label{q8r2:inj-diverges}
For a bounded-orbit wandering component as above,
\[
 \operatorname{inj}_{U_n}(z_n)\longrightarrow\infty.
\]
Here simply connected components have infinite injectivity radius.
\end{theorem}
\begin{proof}
Suppose otherwise.
For infinitely many $n$ there is a non-null-homotopic based loop $\ell_n$ at $z_n$ of hyperbolic length at most a fixed $L$.
Increase $L$ once if necessary, and approximate these loops by polygonal loops without changing their homotopy classes.
Decomposing a polygonal loop at its finitely many self-intersections expresses its homotopy class as a product of conjugates of simple polygonal cycles.
Hence $\ell_n$ contains an essential Jordan cycle $\gamma_n\subset U_n$.
Both the cycle and the paths connecting it to the basepoint are contained in $\ell_n$; only the uniform length bound for $\ell_n$ will be used below.

For every $j\ge0$, Schwarz--Pick for $f^j:U_n\to U_{n+j}$ gives
\[
 \operatorname{length}_{U_{n+j}}(f^j(\ell_n))\le L.
\]
Lift this parametrized based loop to $\D$ starting at zero under $\pi_{n+j}$.
Every point of the lifted path has hyperbolic distance at most $L$ from zero.
Choose $r<1$ with $\operatorname{arctanh}r>L$.
Then
\[
 f^j(\gamma_n)\subset f^j(\ell_n)
 \subset\pi_{n+j}(\overline{D_r})
 \qquad(j\ge0).
\]
By Lemma~\ref{q8r2:cover-disks}, the sets on the right have diameters tending to zero and distances to $K$ tending to zero, uniformly for all indices $n+j\ge n$ as $n\to\infty$.
Choose $\varepsilon>0$ so that the closed $2\varepsilon$-neighborhood of $K$ lies in $W$.
For large $n$, each cover image above has diameter below $\varepsilon$ and lies within $\varepsilon$ of $K$.
It lies in a single $2\varepsilon$-disk about a point of $K$, as does its Euclidean convex hull.
Since a polynomial hull is contained in the convex hull, every hull $\widehat{f^j(\gamma_n)}$, $j\ge0$, is contained in the fixed pole-free neighborhood $W$ from Lemma~\ref{q8r2:hull-fill}.
That lemma makes $\gamma_n$ null-homotopic in $U_n$, contradicting its choice.
\end{proof}

For every fixed $r<1$, sufficiently late covers $\pi_n$ are injective on a neighborhood of $\overline{\D_r}$.
Indeed, two points of a fixed covering disk in the same deck orbit would give a nontrivial deck displacement of the origin bounded by twice that disk's hyperbolic radius, contradicting Theorem~\ref{q8r2:inj-diverges}.
This conclusion does not assert eventual simple connectivity of the full components.

\subsection{Proof of the wandering-orbit theorem}
\label{q8new:local-chart-section}

\begin{proof}[Proof of Theorem~\ref{q8r2:headline}]
Keep the notation above.
The pole-free compact limit set $K$ satisfies $f(K)\subset K$.
For every $r<1$, Lemma~\ref{q8r2:cover-disks} gives
\begin{equation}\label{q8new:cover-shrink-input}
 \sup_{|\zeta|\le r}|\pi_n(\zeta)-z_n|\longrightarrow0,
\end{equation}
and Theorem~\ref{q8r2:inj-diverges} makes $\pi_n$ injective on a neighborhood of $\overline{\D}_r$ for all sufficiently large $n$.

For every $a\in K$, the local normal form of a nonconstant holomorphic function gives a relatively compact Jordan domain $D_a$ containing $a$ and a disk $B_a=D(f(a),R_a)$ such that
\[
 f:D_a\longrightarrow B_a
\]
is proper of finite degree.
The domain can be chosen within the fixed pole-free neighborhood.
More explicitly, in a local conformal coordinate $h_a$, one has $f=f(a)+h_a^{d_a}$, and concentric coordinate disks give the required restriction.
Its finite set $E_a$ of critical values contains at most the center $f(a)$.

Choose concentric target disks
\[
 B_a^0=D(f(a),R_a/4),\quad
 B_a^1=D(f(a),R_a/2),\quad
 B_a^2=D(f(a),3R_a/4),\quad B_a.
\]
Let $A_a$ be the component of $f^{-1}(B_a^0)\cap D_a$ containing $a$.
Then $a\in A_a\Subset D_a$ and $f(A_a)\subset B_a^0$.
Select finitely many of these source cores covering $K$, and index them by $i=1,\ldots,s$.
We henceforth write $D_i,A_i,B_i^0,B_i^1,B_i^2,B_i,E_i$.
All choices are fixed independently of the time index and of the Julia marking introduced later.
Put $E=\bigcup_i E_i$.

Since $E$ is finite and the $U_n$ are pairwise disjoint, discard finitely many indices so that $U_n\cap E=\varnothing$.
Also $z_n\in\bigcup_i A_i$ for every sufficiently large $n$.
Choose an index $i_n$ with $z_n\in A_{i_n}$.
The finite number of charts and the positive uniform source and target
margins $\operatorname{dist}(\overline{A_i},\partial D_i)>0$ and
$\operatorname{dist}(\overline{B_i^0},\partial B_i^1)>0$,
together with \eqref{q8new:cover-shrink-input}, imply that for each fixed $r<1$ and all sufficiently large $n$,
\begin{equation}\label{q8new:chart-margins}
 \pi_n(\overline{\D}_r)\subset D_{i_n},
 \qquad
 \pi_{n+1}(\overline{\D}_r)\subset B_{i_n}^1.
\end{equation}

We next compare these local coverings with the complete metrics of the full Fatou components.
Let $W_n$ be the component of $U_{n+1}\cap B_{i_n}$ containing $z_{n+1}$, and let $C_n$ be the component of $D_{i_n}\cap f^{-1}(W_n)$ containing $z_n$.
Properness of $f:D_{i_n}\to B_{i_n}$ and $W_n\cap E_{i_n}=\varnothing$ make $f:C_n\to W_n$ an unbranched finite-degree covering onto $W_n$.
Backward invariance of the Fatou set places $C_n$ inside $U_n$.
Write
\[
 \lambda_n=
 \frac{|f'(z_n)|\rho_{U_{n+1}}(z_{n+1})}{\rho_{U_n}(z_n)},
\]
so Schwarz--Pick gives $0\le\lambda_n\le1$.
Choose the arguments of the covers $\pi_n$ recursively so that the lifts $F_n:\D\to\D$ of $f\pi_n$ through $\pi_{n+1}$ satisfy
\[
 F_n(0)=0,\qquad F_n'(0)=\lambda_n>0,
 \qquad f\pi_n=\pi_{n+1}F_n.
\]
The positivity follows already from the local covering $C_n\to W_n$.
Rotating coordinates does not affect either geometric input.

Fix any $0<t<1$.
By \eqref{q8new:chart-margins}, for all large $n$ the map $\pi_{n+1}|_{\D_t}$ takes values in $W_n$.
Lift it through the covering $f:C_n\to W_n$ with base point $z_n$, and then lift that map through $\pi_n$.
This gives a holomorphic map $\beta_n:\D_t\to\D$ fixing zero and satisfying
\[
 F_n\beta_n=\mathrm{id}_{\D_t}.
\]
The identity follows from uniqueness of lifts through $\pi_{n+1}$.
Schwarz's lemma on $\D_t$ gives $|\beta_n'(0)|\le1/t$, whence $\lambda_n\ge t$.
Since $t<1$ was arbitrary,
\begin{equation}\label{q8new:lambda-to-one}
 \lambda_n\longrightarrow1.
\end{equation}
This lift uses only the proper local chart over $W_n$.

We now construct a disk with injective forward iterates and arbitrarily large initial intrinsic area.
Fix $0<r<r_1<r_2<1$.
Equations \eqref{q8new:near-id} and \eqref{q8new:lambda-to-one}, followed by Cauchy's estimate, show that for all sufficiently large $N$,
\[
 \sup_{n\ge N}\|F_n'-1\|_{\overline{\D}_{r_1}}<\tfrac12.
\]
Integration along line segments in the convex disk $\D_{r_1}$ shows that every such $F_n$ is injective there.
Schwarz's lemma gives $F_n(\D_{r_1})\subset\D_{r_1}$, so all compositions $F_{N,m}=F_{N+m-1}\cdots F_N$ are injective on $\D_{r_1}$ as well.
Increase $N$ so that $\pi_n$ is injective on a neighborhood of $\overline{\D}_{r_1}$ for all $n\ge N$, and so \eqref{q8new:chart-margins} holds with $r_1$.
Set
\[
 Q=\pi_N(\overline{\D}_r),\qquad Q_m=f^m(Q).
\]
Then every $f^m$ is injective on a neighborhood of $Q$, because
\[
 f^m\pi_N=\pi_{N+m}F_{N,m}
\]
and $F_{N,m}(\D_{r_1})\subset\D_{r_1}$.
The $Q_m$ are pairwise disjoint, being in distinct full Fatou components.
Moreover
\begin{equation}\label{q8new:packet-routing}
 Q_m\subset D_{i_{N+m}},\qquad
 Q_{m+1}\subset B_{i_{N+m}}^1.
\end{equation}
Covering invariance for $\pi_N$ gives the exact initial area
\begin{equation}\label{q8new:large-set-area}
 \alpha_{U_N}(Q)=\alpha_\D(\overline{\D}_r)
 =\frac{2r^2}{1-r^2}.
\end{equation}
Since the latter tends to infinity as $r\uparrow1$, this constructs disks with injective forward iterates of arbitrarily large initial intrinsic area.
The time $N$ is chosen after $r$, so no preservation of area during an uncontrolled initial segment is required.

The local coverings give a uniform upper bound for that area.
Fix a finite forward-invariant union $P_0$ of repelling cycles containing at least three points.
Choose increasing finite unions $P_j\supset P_0$ of complete repelling cycles with union dense in $\Jul(f)$, and put $X_j=\C\setminus P_j$.
For each chart define
\[
 Y_{j,i}=(X_j\setminus E_i)\cap B_i,
 \qquad
 \nu_{j,i}=\alpha_{Y_{j,i}}-\alpha_{X_j}
 \quad\hbox{on }Y_{j,i}.
\]
Put $Z_{j,i}=X_j\setminus E_i$.
On $Y_{j,i}$ split the defect as
\[
 \nu_{j,i}=(\alpha_{Y_{j,i}}-\alpha_{Z_{j,i}})
                +(\alpha_{Z_{j,i}}-\alpha_{X_j}).
\]
The first term has the uniform local relative-area bound of Lemma~\ref{lem:local-truncation} with $\Omega=Z_{j,i}$ and background $\C\setminus P_0$.
The finite-puncture cost for the second term is at most $|E_i\cap X_j|\le|E_i|$.
Hence
\begin{equation}\label{q8new:ambient-chart-budget}
 \nu_{j,i}(Y_{j,i}\cap B_i^1)\le |E_i|+C_i
 \quad\hbox{for every }j.
\end{equation}
Here $C_i$ is the uniform local relative-area constant for inner set $B_i^1$, a cutoff supported in $B_i^2$, outer set $B_i$, and fixed background $\C\setminus P_0$.
The constant is uniform in all Julia markings.
Set
\[
 B_*:=\sum_{i=1}^{s}(|E_i|+C_i)<\infty.
\]

The local restriction
\[
 f:D_i\cap f^{-1}(Y_{j,i})\longrightarrow Y_{j,i}
\]
is a covering, componentwise onto each component, since $f:D_i\to B_i$ is proper and its critical values have been deleted.
Its source is contained in $X_j$: if $x\in P_j$, then $f(x)\in P_j$, so $x$ cannot map into $Y_{j,i}$.
Therefore covering invariance and metric monotonicity give
\begin{equation}\label{q8new:local-ambient-pullback}
 f^*\alpha_{Y_{j,i}}
 =\alpha_{D_i\cap f^{-1}(Y_{j,i})}\ge\alpha_{X_j}
 \quad\hbox{on }D_i\cap f^{-1}(Y_{j,i}).
\end{equation}
This comparison uses only the indicated local source.
It does not assert that the full map over $X_j$ or $Y_{j,i}$ is a covering.

Apply the set construction with $r$ chosen so that $\alpha_\D(\overline{\D}_r)>B_*$.
For fixed $j$, put $b_m=\alpha_{X_j}(Q_m)$.
By \eqref{q8new:packet-routing}, the corresponding local inequality \eqref{q8new:local-ambient-pullback} holds on $Q_m$, and $Q_{m+1}\subset Y_{j,i_{N+m}}\cap B_{i_{N+m}}^1$.
Injectivity gives
\[
 (b_m-b_{m+1})_+
 \le\nu_{j,i_{N+m}}(Q_{m+1}).
\]
For each chart index the image sets are disjoint.
Consequently
\begin{equation}\label{q8new:final-finite-loss}
 \sum_{m\ge0}(b_m-b_{m+1})_+
 \le\sum_i\nu_{j,i}(Y_{j,i}\cap B_i^1)\le B_*.
\end{equation}
The punctured plane $X_j$ has finite complete normalized area $|P_j|-1$.
Disjointness of the $Q_m$ therefore implies $b_m\to0$.
Telescoping \eqref{q8new:final-finite-loss} yields $\alpha_{X_j}(Q)\le B_*$ for every $j$.
Finally, recovery of the complete Fatou metric by dense Julia markings gives
\[
 \alpha_{U_N}(Q)=\lim_{j\to\infty}\alpha_{X_j}(Q)\le B_*,
\]
contrary to \eqref{q8new:large-set-area} and the choice of $r$.

The contradiction excludes bounded point orbits in a wandering component.
Every normal subsequential limit of the iterates on a wandering component is constant.
Otherwise choose a small source disk on which the nonconstant limit is finite and univalent.
Uniform convergence and Rouch\'e's theorem place a fixed nonempty target disk in the image of that source disk for every large index of the subsequence.
That target disk would belong to infinitely many pairwise disjoint full forward Fatou components, a contradiction.
If the point orbit is unbounded, choose a subsequence tending to infinity in the sphere and then a normal subsubsequence on the component.
Its constant normal limit must be infinity by evaluation at the chosen point.
Conversely, an infinity normal subsequence makes every point orbit unbounded.
Thus exclusion of bounded point orbits is exactly the required existence of a locally uniform infinity subsequence.

This proves the theorem.
\end{proof}

\section{Baker cycles: proof of Theorem~\ref{q4:main}}\label{sec:q4}\label{sec:boundary}

Question~4 of Bergweiler \cite[Section~4.3]{Bergweiler} asks for a relation between inverse singular values and Baker-cycle boundaries.
We prove the finite-deletion and separation assertions of Theorem~\ref{q4:main}.

Zheng \cite[Theorem~2.2 and p.~30]{Zheng2005} established phasewise accumulation of singular values of an iterate at Baker-cycle limits.
To see the finite-deletion consequence, let $B_p$ be the set of defined images $f^j(s)$ with $s\in S$ and $0\le j<p$.
An inverse branch of $f^p$ may fail to continue only at a failure of one of its $p$ successive inverse steps.
Thus the inverse singular values of the return iterate lie in the spherical closure of $B_p$.
For finite $T$, $B_p\setminus\mathcal O_T$ is finite: each deleted generator contributes at most $p$ images.
Taking derived sets gives $a_i\in A_T'$.
Infinitely many distinct points of $B_p$ approaching $a_i$ necessarily originate from distinct generators.
For an infinite phase limit, Bergweiler's Theorem~16 \cite{Bergweiler} gives the finite-step accumulation directly.
We retain an area proof because the same argument supplies the separation assertion.
Earlier geometry for the full postsingular set includes Bergweiler's theorem for invariant entire Baker domains \cite[Theorem~3]{Bergweiler1995}, the meromorphic extension \cite[Proposition~7.5]{MBRG},
and the result under a complement condition \cite[Theorem~A]{BFJK}.
Rippon and Stallard \cite{RipponStallard2006} also established a finite-pole annular-density theorem for singular values of inverse iterates.
The hyperbolic separation in Theorem~\ref{q4:main} applies after deleting any finite set of original singular generators.
Rempe's example \cite[Theorem~1.2]{RempeSingularOrbits2022} has a Baker domain with every finite inverse singular value periodic; thus the accumulation statement does not force any fixed finite original generator to escape.
Distances are taken in the component containing the orbit point; distance to another component or the empty set is infinity.
\subsection{Topology of a periodic quotient}

\begin{lemma}\label{q4:planar-cyclic-genus}
Let $W$ be a connected orientable surface with a connected infinite cyclic cover $V\to W$.
If $V$ is planar and $W$ has finite type, then the genus of $W$ is at most one.
\end{lemma}
\begin{proof}
Let $\chi:\pi_1(W)\to\Z$ be the epimorphism defining the cover.
It factors through integral homology.
If two integral homology classes lie in $\ker\chi$, represent them by based closed curves in general position and lift them to closed curves $\widetilde\alpha,\widetilde\beta$ in $V$.
Writing $\Theta$ for a generator of the deck group, their algebraic intersection downstairs is
\[
 [\alpha]\cdot[\beta]
 =\sum_{k\in\Z}
 [\widetilde\alpha]\cdot[\Theta^k\widetilde\beta].
\]
Only finitely many terms can be nonzero, by proper discontinuity and compactness of the curves.
Every term is zero because the intersection form of a planar surface is zero.
Hence the intersection form on $H_1(W;\R)$ vanishes on the codimension-one subspace $\ker\chi$.
If the genus is $g$, the full intersection form has rank $2g$, and its restriction to a hyperplane has rank at least $2g-2$.
Therefore $g\le1$.
\end{proof}

\begin{lemma}
\label{q4:quotient-infinite-area}
Let $V\to\Sigma$ be a connected infinite cyclic holomorphic covering of complete hyperbolic surfaces.
Suppose that $V$ is planar and has no puncture ends.
Then $\Sigma$ has infinite complete hyperbolic area.
\end{lemma}
\begin{proof}
Suppose otherwise.
A complete finite-area hyperbolic surface is a compact surface of some genus $g$ with finitely many punctures, say $r$; in particular $\chi(\Sigma)=2-2g-r<0$.
Let $\omega:\pi_1(\Sigma)\to\Z$ define the infinite cyclic cover.
For a peripheral loop about the $j$th puncture, write $a_j$ for its image under $\omega$.
Each $a_j$ is nonzero: if $a_j=0$, a sufficiently small punctured-disk neighborhood lifts homeomorphically and conformally to a puncture end in $V$.

Let $\Sigma_N$ be the connected $N$-sheet cyclic cover corresponding to $\omega^{-1}(N\Z)$.
Its number of punctures is
\[
 r_N=\sum_{j=1}^{r}\gcd(N,|a_j|)
 \le\sum_{j=1}^{r}|a_j|.
\]
The formula counts the orbits of addition by $a_j$ on $\Z/N\Z$.
The same $V$ is an infinite cyclic cover of $\Sigma_N$, so Lemma~\ref{q4:planar-cyclic-genus} gives $g_N\le1$.
Multiplicativity of Euler characteristic now yields
\[
 2-2g_N-r_N=N\chi(\Sigma).
\]
The left-hand side is bounded below independently of $N$, while the right-hand side tends to $-\infty$, a contradiction.
\end{proof}

\begin{lemma}
\label{q4:direct-limit-no-punctures}
Let $\Gamma_0\subset\Gamma_1\subset\cdots$ be discrete torsion-free subgroups of $\operatorname{Aut}(\D)$ whose union $\Gamma_\infty$ is discrete.
Suppose each $U_n=\D/\Gamma_n$ is conformally a plane domain with no isolated spherical boundary point.
Then $V=\D/\Gamma_\infty$ has no puncture end.
\end{lemma}
\begin{proof}
Suppose $B\subset V$ is a punctured-disk neighborhood of a puncture end.
Its cyclic fundamental group, with a connecting path to a fixed base point, is generated by one element of $\Gamma_\infty$, hence belongs to $\Gamma_N$ for some $N$.
The covering $U_N\to V$ therefore has a single-valued lift $h:B\to U_N$ of the inclusion $B\hookrightarrow V$.
This lift is injective, since its composition with the covering is the identity on $B$.

Identify $B$ with a punctured disk and $U_N$ with its plane realization.
The univalent function $h$ has a removable singularity or a simple pole at the puncture: an essential singularity is incompatible with univalence, and an extension of multiplicity greater than one is also incompatible with univalence.
Its limiting value $b\in\sphere$ cannot lie in $U_N$, since continuity of $U_N\to V$ would then give an interior limiting value of $B$ at its missing puncture.
The image under $h$ contains a punctured spherical neighborhood of $b$.
Thus $b$ is an isolated spherical boundary point of $U_N$, a contradiction.
\end{proof}

\subsection{The periodic covering model and disjoint forward images}
\label{q4:covering-model-section}

Let $U$ be a Baker component of period $p$ and assume that $g=f^p:U\to U$ is a holomorphic covering.
In this section we use only this covering property, local uniform convergence of $g^n$ to a point
of $\partial U$, and the fact that $U$ is a plane Fatou component of a
transcendental meromorphic map.
In particular, no iterate of $g$ has a fixed point in $U$.

\begin{lemma}\label{q4:no-puncture-end}
Every Fatou component of a transcendental meromorphic function has no puncture end.
In particular this holds for $U$.
\end{lemma}
\begin{proof}
A puncture end has a univalent parametrization by a punctured disk whose outer boundary is compactly contained in $U$.
Its inclusion in the sphere extends at the puncture: an injective meromorphic function cannot have an essential singularity there.
The extension is locally univalent, so its image contains a punctured spherical neighborhood of the limiting point.
This would make that point an isolated point of the complement of $U$, and hence an isolated Julia point.
This is impossible.
The same argument at infinity uses that the Julia set of a transcendental meromorphic map is unbounded, so infinity is not an isolated Julia point either.
\end{proof}

We use the following consequence of J\o rgensen's inequality \cite{Jorgensen1976}.
With lifts to $\operatorname{SL}_2(\C)$, the inequality states that a discrete non-elementary group generated by $A,B$ satisfies
\begin{equation}\label{q4:jorgensen}
 |\operatorname{tr}(A)^2-4|
 +|\operatorname{tr}(ABA^{-1}B^{-1})-2|\ge1.
\end{equation}

\begin{lemma}\label{q4:nested-groups}
An increasing union of torsion-free discrete Fuchsian groups containing a non-elementary subgroup is discrete.
\end{lemma}
\begin{proof}
Fix in one stage two hyperbolic elements $\gamma_1,\gamma_2$ with disjoint pairs of fixed points.
If the union were not discrete, it would contain nonidentity elements $h_j\to1$.
For each $j$, the elements $h_j,\gamma_1,\gamma_2$ lie in one common discrete stage.
Applying \eqref{q4:jorgensen} with $A=h_j$, $B=\gamma_i$, and choosing lifts of $h_j$ tending to the identity, shows that $\langle h_j,\gamma_i\rangle$ is elementary for both $i=1,2$ when $j$ is sufficiently large.
Therefore $h_j$ preserves each pair of fixed points.
It cannot interchange a pair: an orientation-preserving hyperbolic-plane isometry interchanging the endpoints of a geodesic is a half-turn, whereas the stage is torsion-free.
Thus $h_j$ fixes all four boundary points, which forces $h_j=1$.
\end{proof}

Choose a universal covering $\pi:\D\to U$, with deck group $\Gamma_0$.
Because $g$ is a covering, it has an automorphic lift $T\in \operatorname{Aut}(\D)$ satisfying $\pi T=g\pi$.
Hence
\begin{equation}\label{q4:ascending-groups}
 T\Gamma_0T^{-1}\subset\Gamma_0,
 \qquad \Gamma_n=T^{-n}\Gamma_0T^n,
 \qquad \Gamma_\infty=\bigcup_{n\ge0}\Gamma_n.
\end{equation}
The sequence $\Gamma_n$ is increasing, and $T$ normalizes its union.

\begin{lemma}\label{q4:stable-group-discrete}
The group $\Gamma_\infty$ is discrete and torsion-free.
The surface $W=\D/\Gamma_\infty$ has genus zero, and the natural map $h:U\to W$ is a holomorphic covering.
The map $T$ induces an automorphism $\tau:W\to W$ with $hg=\tau h$.
\end{lemma}
\begin{proof}
If $\Gamma_0$ is non-elementary, discreteness follows from Lemma~\ref{q4:nested-groups}.
If $\Gamma_0$ is trivial, so is every $\Gamma_n$.
The remaining torsion-free elementary possibilities are infinite cyclic groups.
A parabolic cyclic group would make $U$ a punctured disk, contrary to Lemma~\ref{q4:no-puncture-end}.
If $\Gamma_0=\langle\gamma\rangle$ is hyperbolic cyclic, the inclusion in \eqref{q4:ascending-groups} gives $T\gamma T^{-1}=\gamma^d$ for a nonzero integer $d$.
Translation length is invariant under conjugacy, so $|d|=1$.
Thus all the groups coincide and are discrete.
Every element of the union lies in a torsion-free stage, proving torsion-freeness.
The covering and intertwining assertions now follow from subgroup inclusion and normalization.

To verify genus zero without a conformal uniformization assertion about $W$, take any compact subsurface $K\subset W$.
Its fundamental group is finitely generated.
Its image in $\Gamma_\infty$ is contained in some $\Gamma_n$, so the inclusion $K\hookrightarrow W$ lifts to $\D/\Gamma_n$.
The lift is an embedding because its composition with the covering projection is the original inclusion.
But $\D/\Gamma_n$ is conformally equivalent to the plane domain $U$.
Consequently $K$ cannot have positive genus.
This proves the assertion.
\end{proof}

\begin{lemma}\label{q4:full-group}
Let $H=\langle\Gamma_\infty,T\rangle$.
Then $H$ is a torsion-free discrete Fuchsian group, and
\[
 H/\Gamma_\infty\cong\Z.
\]
The automorphism $\tau$ generates a free properly discontinuous action on $W$, and $\Sigma=W/\langle\tau\rangle=\D/H$ is a hyperbolic Riemann surface.
There is a holomorphic covering $\Pi:U\to\Sigma$ satisfying $\Pi g=\Pi$.
\end{lemma}
\begin{proof}
Normalization makes $H/\Gamma_\infty$ cyclic.
If $T^k\in\Gamma_\infty$ for $k\ne0$, then $T^k\in\Gamma_n$ for some $n$.
Conjugating by $T^n$ shows $T^k\in\Gamma_0$.
Taking $k>0$ by inversion, this gives $g^k=\mathrm{id}$ on $U$, contrary to escape.
The quotient is therefore infinite cyclic.

More generally, $\tau^k$ has no fixed point for any $k\ne0$.
Otherwise $T^k\zeta=\gamma\zeta$ for some $\gamma\in\Gamma_\infty$ and $\zeta\in\D$.
Choose $n$ with $\gamma\in\Gamma_n$ and conjugate by $T^n$.
Since $T^n\gamma T^{-n}\in\Gamma_0$, projection by $\pi$ gives a fixed point of $g^k$ in $U$ (replace $k$ by $-k$ if necessary).
This is impossible.

If $\Gamma_\infty$ is non-elementary, its normalizer is discrete.
Here is the needed argument.
Were $s_j\to1$ nonidentity elements of the normalizer, then for two fixed hyperbolic elements of $\Gamma_\infty$ with disjoint endpoint pairs, $s_j\gamma_i s_j^{-1}\to\gamma_i$.
Discreteness of $\Gamma_\infty$ forces equality for all large $j$ and both $i$.
The common centralizer of those two hyperbolic elements is trivial, a contradiction.
Since $H$ lies in this normalizer, it is discrete.

If $\Gamma_\infty$ is elementary, the preceding proof shows $\Gamma_\infty=\Gamma_0$; thus $W=U$ and $g$ is an automorphism.
Local uniform escape implies that its cyclic action is properly discontinuous.
Indeed, for compact $K,L\subset U$, only finitely many $n\ge0$ satisfy $g^n(K)\cap L\ne\varnothing$, and negative $n$ reduce to the same statement with $K,L$ interchanged.
Lifting to $\D$ and using discreteness of $\Gamma_0$ proves that $H$ is discrete in this case as well.

There are no nontrivial elliptic elements of $H$.
An element has the form $\gamma T^k$.
For $k=0$ this follows from torsion-freeness of $\Gamma_\infty$.
If $k\ne0$ and it fixes a point of $\D$, the preceding finite-stage conjugation gives a periodic point of $g$ in $U$.
Therefore $H$ is torsion-free.
Its action on $\D$ is free and properly discontinuous, and the quotient map induces the asserted free cyclic covering $W\to\Sigma$.
Composition with $h$ gives $\Pi$.
\end{proof}

\begin{proposition}\label{q4:periodic-packet}
For every $B<\infty$ there is a relatively compact simply connected domain $D\subset U$ and a compact Jordan disk $K\subset D$ such that
\[
 \alpha_U(K)>B,
\]
every $g^n$ is injective on $D$, and the domains $g^n(D)$, $n\ge0$, are pairwise disjoint.
\end{proposition}
\begin{proof}
Every $\D/\Gamma_n$ is conformally equivalent to $U$ and has no puncture end by Lemma~\ref{q4:no-puncture-end}.
Hence Lemma~\ref{q4:direct-limit-no-punctures} shows that $W$ has no puncture end.
Its genus is zero by Lemma~\ref{q4:stable-group-discrete}.
Lemma~\ref{q4:quotient-infinite-area} now gives $\alpha_\Sigma(\Sigma)=\infty$.
There is a smoothly bounded relatively compact disk $D_0\subset\Sigma$, whose closure has a slightly larger disk neighborhood, and a compact Jordan disk $K_0\subset D_0$ with $\alpha_\Sigma(K_0)>B$.
One explicit construction is to take finitely many disjoint compact coordinate disks whose total area exceeds $B$, join them by an embedded finite tree of arcs, and take a sufficiently small regular neighborhood.
The tree creates no handle or cycle, so this neighborhood is a disk.
Take $K_0$ as the closure of a smaller smooth disk that contains the chosen coordinate disks, and enlarge its neighborhood slightly to obtain $D_0$.

Lift a simply connected neighborhood of $\overline{D_0}$ through the covering $\Pi:U\to\Sigma$ of Lemma~\ref{q4:full-group}.
One component is mapped biholomorphically onto that neighborhood.
Let $D$ and $K$ be the corresponding lifts of $D_0$ and $K_0$.
They are relatively compact, and covering invariance of the complete hyperbolic metric gives $\alpha_U(K)=\alpha_\Sigma(K_0)>B$.

The restriction of $h$ to $D$ is injective, because $\Pi$ is injective there.
Its image $h(D)$ is a component over $D_0$ for the cyclic covering $W\to\Sigma$.
The components $\tau^n h(D)$ are pairwise disjoint: if a nonzero power preserved one component, injectivity of its projection to $D_0$ would make that power the identity on an open set and hence on $W$,
contradicting Lemma~\ref{q4:full-group}.
Since $h g^n=\tau^n h$, this proves that $g^n(D)$ are pairwise disjoint.
If $g^n(x)=g^n(y)$ for $x,y\in D$, applying $h$ gives $\tau^n h(x)=\tau^n h(y)$, hence $x=y$.
Thus each iterate is injective on $D$.
\end{proof}

\begin{corollary}\label{q4:full-cycle-packet}
With the same $D$, all maps $f^{np+j}|_D$, $n\ge0$, $0\le j<p$, are injective, and their images are pairwise disjoint.
\end{corollary}
\begin{proof}
Injectivity of $f^{np+j}$ follows from injectivity of $g^{n+1}$, since $f^{p-j}\circ f^{np+j}=g^{n+1}$.
Images in different cycle components are disjoint.
For the same $j$, an intersection of the images for indices $n\ne m$ would remain an intersection after applying $f^{p-j}$, contrary to the disjointness of $g^{n+1}(D)$ and $g^{m+1}(D)$.
\end{proof}

\subsection{The Baker-cycle conclusions}
\begin{proof}[Proof of the accumulation assertion in Theorem~\ref{q4:main}]
Fix $T$ and put $A=A_T$ and $E=S\setminus A\subset T$.
If $A$ meets one component, choose $x\in A\cap U_j$.
Its return orbit consists of distinct points of $A$ tending to $a_j$; repetition would give a periodic point in the Baker component.
Following the orbit through the other phases proves $a_i\in A'$ for every $i$.
We may therefore assume $A$ is disjoint from the whole cycle.

Lemma~\ref{q4:transitions} makes every transition a covering, including the return $g=f^p:U_0\to U_0$.
Apply Proposition~\ref{q4:periodic-packet} and Corollary~\ref{q4:full-cycle-packet} to obtain a compact set $K\Subset U_0$ with $\alpha_{U_0}(K)>|E|+2$, such that every iterate is injective on a neighborhood of $K$ and the sets $K_n=f^n(K)$ are pairwise disjoint.

Set $X_j=\Chat\setminus(A\cup P_j)$, $E_j=E\cap X_j$, $Y_j=X_j\setminus E_j$, $\mu_j=\alpha_{X_j}$, and $\nu_j=\alpha_{Y_j}-\alpha_{X_j}$.
Forward invariance of $A$ at finite points and of $P_j$ gives $f^{-1}(Y_j)\subset X_j$.
All singular values have been removed from $Y_j$, so $f:f^{-1}(Y_j)\to Y_j$ is a covering.
Hence
\[
 f^*(\mu_j+\nu_j)\ge\mu_j,
 \qquad \nu_j(Y_j)\le |E_j|\le |E|.
\]
Choose $j$ by Lemma~\ref{core:kernel} so that $\mu_j(K)>|E|+1$, and keep this $j$ fixed.
Write $\mu=\mu_j$ and $\nu=\nu_j$.
The covering inequality gives
\[
 (\mu-f^*\mu)_+\le f^*\nu.
\]
Injectivity and disjointness of all the $K_n$ now imply
\begin{equation}\label{q4:loss-estimate}
 \sum_{n\ge0}\int_{K_n}(\mu-f^*\mu)_+
 \le\sum_{n\ge0}\nu(K_{n+1})\le |E|.
\end{equation}
Thus, writing $b_n=\mu(K_n)$ and telescoping only the negative variation,
\begin{equation}\label{q4:mass-floor}
 b_n\ge b_0-|E|>1 \qquad(n\ge0).
\end{equation}

Fix $i$.
Surjectivity of $f^i:U_0\to U_i$ and the defining Baker limits give $f^{np+i}|_K\to a_i$ uniformly in the spherical metric.
If $a_i\notin A'$, choose a spherical disk $B$ about $a_i$ meeting $A\cup P_j$ in at most $a_i$.
On a smaller concentric disk $B'$ the measure $\mu$ has finite mass: compare with the complete metric of $B$ if $a_i$ is retained, and with that of $B\setminus\{a_i\}$ otherwise.
A cusp has finite area near its puncture.
The disjoint sets $K_{np+i}$ eventually lie in $B'$, contradicting \eqref{q4:mass-floor}.
Therefore $a_i\in A'$.

For distinct generators, after choosing $s_1,\ldots,s_{k-1}$, apply the result with $T\cup\{s_1,\ldots,s_{k-1}\}$.
Since the limit $a_i$ lies in the derived set of the corresponding carrier, choose a new generator and an image within spherical distance $1/k$ of $a_i$, avoiding all previous image points.
This proves both distinctness requirements.
\end{proof}

\begin{proof}[Proof of the separation assertion in Theorem~\ref{q4:main}]

Use the set and sufficiently large finite marking $P_j$ in the proof of Theorem~\ref{q4:main}.
They give $X=\Chat\setminus(A_T\cup P_j)$, $Y=X\setminus E$, $E=(S\setminus A_T)\cap X$, and $\alpha_X(K)>|E|$.
The third assertion of Lemma~\ref{lem:finite-loss} and Lemma~\ref{q4:summable-defects} yield a point $z\in U_0$ with
\[
 d_X(f^n(z),X\setminus U_{n\bmod p})\longrightarrow\infty.
\]
For any $w\in U_0$, Schwarz--Pick gives
\[
 d_X(f^n(w),f^n(z))
 \le d_{U_{n\bmod p}}(f^n(w),f^n(z))
 \le d_{U_0}(w,z).
\]
Thus the same separation holds locally uniformly for starting points in $U_0$.

It remains to remove the finite marking.
The accumulation assertion gives $a_i\in A_T'$ for every phase, so $A_T$ is infinite and $Z_T$ is hyperbolic.
For $w_n=f^n(w)$, completeness and properness of its Poincar\'e metric give
\[
 d_{Z_T}(w_n,P_j\setminus A_T)\longrightarrow\infty:
\]
on each phase $w_{np+i}\to a_i\in A_T$, whereas the relevant finite marks lie in the interior of the fixed component of $Z_T$.

Fix $R<\infty$.
On $B_{Z_T}(w_n,R)$ the distance from the finite marks tends uniformly to infinity.
The elementary covering-disk comparison therefore gives
\[
 1\le\rho_X/\rho_{Z_T}\le1+o(1)
\]
uniformly on that ball.
Indeed, at a point whose $Z_T$-distance from the marks exceeds $L$, restrict a universal cover to its centered hyperbolic disk of radius $L$; its image avoids the marks, and Schwarz--Pick bounds the density ratio by $\coth L$.
Every $Z_T$-path of length at most $R$ from $w_n$ consequently lies in $X$ and has $X$-length at most $(1+o(1))R$.
Separation in $X$ therefore puts $B_{Z_T}(w_n,R)$ inside $U_{n\bmod p}$ for large $n$.
Since $R$ was arbitrary, this is \eqref{eq:baker-head-separation}.
\end{proof}

\begin{corollary}
\label{q4:deletion-spatial-mesh}
Let $W$ be a phase of the singular-free Baker cycle for which $f^{np}|_W\to\infty$.
For every finite $T\subset S$, there is a sequence of finite points $p_n\in\mathcal O_T$ such that
\[
 |p_n|\to\infty,
 \qquad
 \sup_n\frac{|p_{n+1}|}{|p_n|}<\infty,
 \qquad
 \frac{\dist(p_n,W)}{|p_n|}\to0.
\]
No boundedness of the forward times defining $p_n$ is asserted.
\end{corollary}
\begin{proof}
Put $g=f^p|_W$.
We use the two Baker-orbit estimates appearing in the proof of \cite[Proposition~7.5]{MBRG}: for a fixed orbit $w_n=g^n(w)$ tending to infinity, there are constants $C_0,C_1>1$ such that,
for all large $n$,
\[
 |w_{n+1}|\le C_1|w_n|,
 \qquad
 \partial W\cap\{ |w_n|/C_0<|z|<C_0|w_n|\}\ne\varnothing.
\]
These are intrinsic Baker-domain estimates and do not depend on a choice of postsingular or finite-deletion carrier.

If $A_T$ meets the cycle, forward invariance gives a point in $A_T\cap W$.
Use its return orbit for $w_n$.
Approximate each $w_n$ by a point of the dense set $\mathcal O_T$ with error $o(|w_n|)$.
The displayed growth bound proves the conclusion in this case.

Otherwise use Theorem~\ref{q4:main} and fix any orbit in $W$.
Choose
$\zeta_n\in\partial W$ in the indicated annulus. Follow the shorter
arc of $|z|=|w_n|$ from $w_n$ to the ray through $\zeta_n$, then the radial segment to $\zeta_n$; stop at the first intersection with
$\partial W$. The resulting $\gamma_n$ lies in
$\{|w_n|/C_0\le |z|\le C_0|w_n|\}$, has length at most $(\pi+C_0)|w_n|$, and has interior in $W$.
We claim
\[
 \dist(\gamma_n,A_T\cap\C)=o(|w_n|).
\]
If not, along a subsequence the distance is at least $\epsilon|w_n|$ for some $\epsilon>0$.
The whole curve then lies in $Z_T$, and disk comparison gives $\rho_{Z_T}\le1/(\epsilon|w_n|)$ on it.
Its hyperbolic length is at most $C_2/\epsilon$, contrary to \eqref{eq:baker-head-separation} at its endpoint in $Z_T\setminus W$.

Choose $q_n\in A_T\cap\C$ at distance $o(|w_n|)$ from $\gamma_n$.
The annular location of $\gamma_n$ gives a constant $C_3>1$ with
\[
 |w_n|/C_3\le |q_n|\le C_3|w_n|,
 \qquad
 \dist(q_n,W)=o(|w_n|).
\]
Approximate $q_n$ by finite $p_n\in\mathcal O_T$ with error $o(|w_n|)$.
The same inequalities, with a slightly larger $C_3$, hold for $p_n$.
They imply escape and relative proximity, while $|p_{n+1}|/|p_n|\le C_3^2C_1$ for all sufficiently large $n$.
Discarding finitely many terms completes the proof.
\end{proof}

\section{Wandering boundaries: proof of Theorem~\ref{q11all:main}}\label{sec:q11}

For the full postsingular set, wandering-domain limit relations were established for entire functions by Bergweiler, Haruta, Kriete, Meier and Terglane \cite{BHKMT1993} and for meromorphic functions by Zheng \cite{Zheng2003} and Baker \cite[Theorem~1]{Baker2002}.
Bara\'nski, Fagella, Jarque and Karpi\'nska \cite[Theorem~B]{BFJK} proved proximity of the full postsingular set to forward Fatou components.
In fact, their \cite[Lemma~2.1 and the proof of Theorem~B]{BFJK} implies \eqref{q11all:distance-ratio} and \eqref{q11all:actual-witness} for $T=\varnothing$ when the full postsingular set avoids the forward components:
apply the lemma to a
shortest segment from $z_n$ to $\partial U_n$ and approximate the
resulting postsingular point by an actual forward singular image.
The full-set separation under the same avoidance condition follows also from \cite[Proposition~7.4]{MBRG}.
The assertions here hold for $A_T$ after any finite set $T$ of original singular generators is omitted, even when their orbits enter the forward components.
The witnesses have no prescribed forward times, and the orbits of the omitted generators are unrestricted.
The thin case gives \eqref{q11all:distance-ratio}-- \eqref{q11all:actual-witness} without asserting \eqref{q11all:thick-hyperbolic}.

\subsection{Thick wandering tails}
\begin{proposition}\label{q11geo:main}
The conclusions of Theorem~\ref{q11all:main} hold for a thick tail.
In the case $A_T\cap U_n=\varnothing$ for all $n$, every constant normal limit belongs to $A_T'$.
\end{proposition}
\begin{proof}
Write $A=A_T$.
If $A\cap U_N\ne\varnothing$, density and openness give $p_N\in\mathcal O_T\cap U_N$.
Its subsequent orbit proves the first assertion.
The covering $f^N:U_0\to U_N$ is onto.
Lifting $p_N$ to $U_0$ shows that every constant normal limit is a limit of distinct points of $A$, since these points lie in different forward components.
It therefore belongs to $A'$.

Assume that $A$ avoids all $U_n$, and put $E_0=S\setminus A\subset T$.
Lemma~\ref{lem:large-disks} gives disks with $\alpha_{U_0}(Q)>|E_0|+2$ and jointly injective, disjoint forward images.
Use \eqref{eq:recovery} to fix a finite repelling marking $P$ with
\[
 X=\sphere\setminus(A\cup P),\quad
 E=E_0\cap X,\quad Y=X\setminus E,
 \qquad\alpha_X(Q)>|E_0|+1.
\]
Forward invariance gives $f^{-1}(Y)\subset X$, and removal of the singular values makes $f:f^{-1}(Y)\to Y$ a covering.
Lemma~\ref{lem:finite-loss} supplies $\alpha_X(f^n(Q))>1$ for all $n$.

Let $a$ be a constant normal limit.
If $a\notin A'$, take a small spherical disk about $a$ meeting $A\cup P$ in at most $a$.
A smaller disk has finite $\alpha_X$-area, by comparison with the disk or punctured-disk metric.
A subsequence of the disjoint sets $f^n(Q)$ lies there, contradicting their positive area lower bound.
Thus $a\in A'$.
In particular $A$ is infinite and $Z=\sphere\setminus A$ is hyperbolic.

The third assertion of Lemma~\ref{lem:finite-loss} and Lemma~\ref{q4:summable-defects} give one $w\in U_0$ for which
\[
 d_X(f^n(w),X\setminus U_n)\longrightarrow\infty.
\]
For $z$ in any fixed compact subset of $U_0$, Schwarz--Pick gives
\[
 d_X(f^n(z),f^n(w))\le d_{U_n}(f^n(z),f^n(w))
 \le d_{U_0}(z,w).
\]
Hence the separation holds locally uniformly for all starting points.
All spherical cluster points of each orbit belong to $A$.
Lemma~\ref{q11geo:remove-marks} removes the finite marking and yields \eqref{q11all:thick-hyperbolic}.
The same conclusion is locally uniform: normality makes the cluster-point assertion uniform on a compact subset of $U_0$, and the preceding distance bound is uniform there.

Set $a_n=\dist(z_n,A\cap\C)\ge d_n$.
If $a_n\ge(1+\eta)d_n$ along a subsequence, a nearest boundary point of $U_n$ lies strictly inside the disk $D(z_n,a_n)\subset Z$ and is at $Z$-distance at most $\operatorname{arctanh}(1/(1+\eta))$ from $z_n$.
This contradicts the separation just proved.
Thus $a_n/d_n\to1$.
Choose a nearest point of the closed set $A\cap\C$ and approximate it by a finite point $p_n\in\mathcal O_T$ with error $o(d_n)$.
Then $|p_n-z_n|/d_n\to1$, and
\[
 \dist(p_n,U_n)\le |p_n-z_n|-d_n=o(d_n),
\]
because $D(z_n,d_n)\subset U_n$.
These are the required witnesses.
\end{proof}

\subsection{Thin wandering tails}
\begin{theorem}
\label{q11e:thin-spatial}
Let $(U_n)$ be a thin singular-free wandering tail of full Fatou components.
For a finite $T\subset S$, suppose $A_T\cap U_n=\varnothing$ for all $n$.
Then for every $z\in U_0$, writing
$z_n=f^n(z)$ and $d_n=\operatorname{dist}(z_n,\partial U_n)$, there are
finite points $p_n\in\mathcal O_T$ such that
\begin{equation}\label{q11e:thin-witness}
 \frac{|p_n-z_n|}{d_n}\longrightarrow1,
 \qquad
 \frac{\operatorname{dist}(p_n,U_n)}{d_n}\longrightarrow0.
\end{equation}
In particular,
\begin{equation}\label{q11e:thin-ratio}
 \frac{\operatorname{dist}(z_n,A_T\cap\C)}{d_n}
 \longrightarrow1.
\end{equation}
No restriction on the finite deleted set $T$ is required.
\end{theorem}
\begin{proof}
Discard the initial finite segment before the thin tail consists of finite-modulus annuli with finite covering degrees.
Reindex this tail temporarily from zero, replace the base point by its iterate at that initial time, and let $D_n$ be its cumulative degree.
Then $D_n\to\infty$: otherwise all sufficiently late degrees would equal one, so the increasing deck groups would stabilize and their union would be discrete, contrary to thinness.

In an annular uniformization of $U_0$, choose three concentric coordinate circles $\gamma_0^-,\gamma_0^0,\gamma_0^+$, with $z$ on the middle one and with the closed band between the two outer circles compactly contained in $U_0$.
A finite annular covering is conformally a power map, possibly followed by inversion.
Thus the underlying images $\gamma_n^-,\gamma_n^0, \gamma_n^+$ are disjoint nested Jordan curves, with $z_n\in\gamma_n^0$.
The two intervening bands have moduli equal to $D_n$ times their fixed positive initial moduli.
The labels $-$ and $+$ may exchange their Euclidean nesting order; the middle circle remains the middle one.

Normality and disjointness of the full wandering components imply that all normal subsequential limits are constant.
The spherical diameter of the image of the whole initial compact band therefore tends to zero.
In particular, each of the three circles has a small spherical Jordan side $B_n^\sigma$, of diameter tending to zero.
Apply Theorem~\ref{q11r:reset-propagation} separately to the two initial outer coordinate circles.
For all sufficiently large $n$ this supplies
\[
 q_n^-\in\mathcal O_T\cap B_n^-,
 \qquad q_n^+\in\mathcal O_T\cap B_n^+.
\]
That theorem also gives infinitely many shrinking sides containing two distinct points of $S$, so $S$ is infinite.
Fix once and for all a finite point $a\in S\setminus T\subset\mathcal O_T$.

We first justify a dichotomy of the small sides used in the witness selection.
Apart from finitely many indices, either all three small sides are bounded Euclidean disks, or all three contain infinity.
For otherwise take an infinite mixed subsequence and then a normal subsequence on the initial compact band.
If its constant limit is finite, all small sides are eventually contained in one small finite disk, which rules out mixing.
If the limit is infinity, the image of the compact band eventually avoids every fixed finite disk, in particular $a$.
The three curves are homotopic in this image band and hence have the same separation relation with $a$ and infinity.
If one small side contains infinity, that side shrinks to infinity, so its opposite bounded side contains $a$ for large indices.
The same is true for all three bounded sides.
None of them can then itself be the small side, since its diameter would be at least the spherical distance from $a$ to its curve, which tends to $d_\#(a,\infty)>0$.
This also rules out mixing.
The same argument shows that whenever all three small sides contain infinity, all their bounded sides contain the fixed point $a$, eventually.

Order the three curves now by Euclidean nesting, and denote their bounded closed sides by
\[
 D_n^-\subset D_n^0\subset D_n^+.
\]
The middle boundary is still $\gamma_n^0$.
The two annuli between these disks are contained in $U_n$ and have moduli tending to infinity.
Let $K_n$ and $L_n$ be the bounded and unbounded complementary continua of the full component $U_n$.
Then
\[
 K_n\subset D_n^-,
 \qquad L_n\subset\sphere\setminus\operatorname{int}D_n^+.
\]

If all the small sides are bounded, let $p_n$ be the reset witness for whichever of the original $-$ or $+$ circles bounds $D_n^-$.
This point is finite and belongs to $\mathcal O_T\subset A_T$.
Since $A_T$ avoids $U_n$, the point must lie in $K_n$.
If all the small sides contain infinity, put $p_n=a$.
The preceding paragraph gives $a\in D_n^-$, and again $a\notin U_n$ forces $a\in K_n$.
These choices give finite actual orbit points $p_n\in K_n\cap\mathcal O_T$ at every sufficiently large time.
They do not require applying $f$ to an infinity-valued reset witness.

Apply Lemma~\ref{q11e:round-ring} with the varying center $p_n$ to each of the two bands.
Define
\[
 R_n^-=\max_{w\in D_n^-}|w-p_n|,
 \quad r_n^0=\min_{w\in\partial D_n^0}|w-p_n|,
\]
\[
 R_n^0=\max_{w\in D_n^0}|w-p_n|,
 \quad r_n^+=\min_{w\in\partial D_n^+}|w-p_n|.
\]
Then
\[
 r_n^0/R_n^-\longrightarrow\infty,
 \qquad r_n^+/R_n^0\longrightarrow\infty.
\]
Since $z_n\in\partial D_n^0$, these inequalities imply, with
$a_n=|z_n-p_n|$,
\[
 R_n^-=o(a_n),\qquad r_n^+/a_n\longrightarrow\infty.
\]
The distance from $z_n$ to $K_n$ lies between $a_n-R_n^-$ and $a_n$, because $p_n\in K_n\subset D_n^-$.
Its distance to $L_n$ is at least $r_n^+-a_n$, which is much larger.
Consequently
\[
 d_n/a_n\longrightarrow1.
\]
This proves the first assertion of \eqref{q11e:thin-witness}.
Since $D(z_n,d_n)\subset U_n$,
\[
 \operatorname{dist}(p_n,U_n)\le a_n-d_n=o(d_n),
\]
giving the second.
Finally $A_T\cap U_n=\varnothing$ and $p_n\in A_T$ imply
\[
 d_n\le\operatorname{dist}(z_n,A_T\cap\C)
 \le |z_n-p_n|,
\]
which yields \eqref{q11e:thin-ratio}.
Restore the discarded finite initial segment to complete the assertion for the original tail.
\end{proof}

\begin{proof}[Proof of Theorem~\ref{q11all:main}]
Fix a finite $T\subset S$.
If $A_T\cap U_N\ne\varnothing$, openness and density give an actual finite point $p_N\in\mathcal O_T\cap U_N$.
Its subsequent orbit stays in the forward full components, proving the first alternative.
Assume otherwise that $A_T$ avoids all of them.
If the covering-group union is discrete, Proposition~\ref{q11geo:main} gives \eqref{q11all:distance-ratio}, \eqref{q11all:actual-witness} and the stronger assertion \eqref{q11all:thick-hyperbolic}.
If it is nondiscrete, the annular reduction and Theorem~\ref{q11e:thin-spatial} give the two Euclidean conclusions without restricting the finite deleted set.
These cases exhaust the covering tails.
The limit-point assertions follow from Proposition~\ref{q11geo:main} in the thick case and Theorem~\ref{q11r:reset-propagation} in the thin case.
\end{proof}

\section{Application: a relative area theorem for polynomial skew products}
\label{sec:elliptic-skew}

The area argument does not require continuity of the fiber Julia sets.
The missing link between fiberwise and joint normality can instead be recovered on almost every fiber by a plurisubharmonic upper envelope.

Let $V\subset\C^m$ be a bounded open set and $T:V\to V$ a biholomorphism.
Let $\sigma$ be a Borel probability measure such that
\begin{equation}\label{eq:rel-measure}
 T_*\sigma=\sigma,\qquad \operatorname{supp}_V\sigma=V,
 \qquad \sigma\ll\operatorname{Leb}_{2m}.
\end{equation}
Suppose that
\[
 F(z,w)=(Tz,q_z(w)),\qquad
 q_z(w)=\sum_{k=0}^d a_k(z)w^k,\qquad d\ge2,
\]
where the $a_k$ are holomorphic on $V$ and
\begin{equation}\label{eq:rel-coeff}
 \sup_{z\in V}|a_k(z)|<\infty\quad(0\le k\le d),
 \qquad\inf_{z\in V}|a_d(z)|>0.
\end{equation}
Write $Q_{0,z}(w)=w$ and $Q_{n,z}=q_{T^{n-1}z}\circ\cdots\circ q_z$.
Uniformity in \eqref{eq:rel-coeff} supplies an $R>1$ such that
\begin{equation}\label{eq:rel-escape}
 |q_z(w)|>2|w|\qquad(z\in V,\ |w|\ge R).
\end{equation}
Set
\[
 K_z=\{w\in\C:|Q_{n,z}(w)|\le R\text{ for all }n\ge0\},
 \qquad K=\{(z,w):w\in K_z\},\qquad\mathcal B=\operatorname{int}K.
\]
The components of $\mathcal B$ are precisely the bounded-orbit components of the relative Fatou set in $V\times\C$.

\begin{theorem}\label{thm:elliptic-skew}
Under \eqref{eq:rel-measure} and \eqref{eq:rel-coeff}, every connected component of $\mathcal B$ is eventually periodic.
\end{theorem}

We prove Theorem~\ref{thm:elliptic-skew} through an almost-everywhere recovery of complete fiber components and a relative area bound.

For fixed $z$, let $\mathcal F_z$ be the normality set of $(Q_{n,z})_n$ and put $C_z=\operatorname{int}K_z$.
The connected components of $C_z$ are exactly the bounded-orbit components of $\mathcal F_z$: a fiber Fatou component cannot contain both bounded and escaping points, by the identity theorem for a normal limit.
Each component $C\subset C_z$ is simply connected.
If a Jordan curve lies in $C$, all $Q_{n,z}$ are bounded by $R$ on it, and the maximum principle gives the same bound on its interior.
The interior therefore lies in $C$.

The map $F:V\times\C\to V\times\C$ is finite and proper.
Indeed, $T$ is proper, and roots of $q_z(w)=u$ remain uniformly bounded when $(z,u)$ ranges over a compact set, by \eqref{eq:rel-coeff}.
Moreover
\begin{equation}\label{eq:rel-complete}
 F^{-1}(\mathcal B)=\mathcal B,
 \qquad q_z^{-1}(C_{Tz})=C_z.
\end{equation}
For the first equality, note that $F^{-1}(K)=K$ by \eqref{eq:rel-escape}; taking interiors commutes with inverse image because $F$ is open.
The second follows in the same way from the openness of $q_z$.
Properness now shows that $F$ maps each component of $\mathcal B$ properly and onto its successor, and that $q_z$ maps each complete component of $C_z$ properly and onto a complete component of $C_{Tz}$.

The following extension lemma makes the latter statement usable for slices of a joint Fatou component.

\begin{lemma}\label{lem:rel-psh}
Let $B\subset\C^m$ be a domain, $f_n\in\mathcal O(B\times\C)$, and suppose that for some $a\in\C$, $r_0>0$, and $M<\infty$,
\[
 |f_n(z,w)|\le M\quad(z\in B,\ |w-a|<r_0,\ n\ge0).
\]
There is a Borel Lebesgue-null set $P\subset B$ with this property: if $z_0\in B\setminus P$ and $(f_n(z_0,\cdot))_n$ is locally uniformly bounded on $D(a,r)$, then for every $s<r$ there is a neighborhood $B'\ni z_0$ on which $(f_n)_n$ is uniformly bounded on $B'\times D(a,s)$.
\end{lemma}

\begin{proof}
Write $f_n(z,a+\zeta)=\sum_{k\ge0}c_{n,k}(z)\zeta^k$.
Cauchy's estimate on any fixed disk of radius $t<r_0$ gives $|c_{n,k}(z)|\le Mt^{-k}$ for all $z,n,k$.
For $N\ge1$ let
\[
 s_N(z)=\sup_{\substack{n\ge0\\ k\ge N}}
       k^{-1}\log|c_{n,k}(z)|,
 \qquad v_N=s_N^*.
\]
Here a logarithm of zero has value $-\infty$.
The functions inside the supremum are plurisubharmonic and locally uniformly bounded above.
The upper-envelope theorem \cite[Chapter~I, Theorem~5.7]{DemaillyCADG} says that $v_N$ is plurisubharmonic and agrees with $s_N$ off a Lebesgue-null set.
Delete the countable union $P$ of these discrepancy sets.

For $z_0\notin P$, fiberwise Cauchy estimates give
\[
 \inf_N v_N(z_0)
 =\limsup_{k\to\infty}k^{-1}\log\sup_n|c_{n,k}(z_0)|
 \le-\log r.
\]
Given $s<s_1<r$, choose $N$ such that $v_N(z_0)<-\log s_1$.
Upper semicontinuity gives the same strict inequality throughout a smaller $B'$.
Hence $|c_{n,k}(z)|\le s_1^{-k}$ for $k\ge N$, $n\ge0$, and $z\in B'$.
The finitely many lower coefficients have the previous Cauchy bound.
Summing the geometric tail proves a common bound on $B'\times D(a,s)$.
\end{proof}

\begin{lemma}\label{lem:rel-full-slices}
There is a Borel Lebesgue-null set $P\subset V$ such that, whenever $z\notin P$ and $U$ is a component of $\mathcal B$, each component of the slice $U_z$ is a complete component of $C_z$.
\end{lemma}

\begin{proof}
Fix $U$.
For rational $a\in\C$, choose from a countable basis all product tubes $B\times D(a,r_0)\subset U$.
On each tube all the $Q_{n,z}$ have modulus at most $R$, so Lemma~\ref{lem:rel-psh} applies.
Delete the union $P_U$ of its exceptional sets, extended by the empty set outside the corresponding $B$.

Let $z\notin P_U$, let $D$ be a component of $U_z$, and let $C$ be the component of $C_z$ containing $D$.
For each rational $a\in D$
and $s<r<\operatorname{dist}(a,\partial C)$, the sequence
$(Q_{n,z})_n$ is uniformly bounded on $D(a,r)$.
A tube through $(z,a)$ and the preceding lemma imply that $D(a,s)$ belongs to $\mathcal B_z$.
It meets $D$, so it belongs to $D$.
Letting $s$ increase gives
\[
 \operatorname{dist}(a,\partial D)
       =\operatorname{dist}(a,\partial C).
\]
If $D$ were a proper subset of $C$, a point of $\partial D\cap C$
and a nearby rational $a\in D$ would contradict this equality.
There are only countably many components $U$ of the open set $\mathcal B$; taking the union of their $P_U$ proves the assertion.
\end{proof}

\begin{proof}[Proof of Theorem~\ref{thm:elliptic-skew}]
Suppose that $U_0,U_1,\ldots$ is a pairwise disjoint forward orbit of components of $\mathcal B$, with $F(U_n)=U_{n+1}$.
Put
\[
 B_n=\{z\in V:(U_n)_z\cap\operatorname{Crit}(q_z)\ne\varnothing\}.
\]
These sets are open, since polynomial critical roots persist under parameter perturbation.
At a fixed base point $z$, each of the at most $d-1$ distinct finite critical points of $q_z$ lies in at most one of the disjoint $U_n$.
Thus
\begin{equation}\label{eq:rel-critical}
 \sum_{n\ge0}\sigma(T^{-n}B_n)
 =\sum_{n\ge0}\sigma(B_n)\le d-1.
\end{equation}
By summability, almost every base orbit meets $B_n$ only finitely often at time $n$.
By absolute continuity, we may also require that every forward and backward iterate avoid the null set in Lemma~\ref{lem:rel-full-slices}.

Choose a product tube $B\times D_*\Subset U_0$ and $w_*\in D_*$.
For $z\in B$, let $H_{0,z}$ be the component of $(U_0)_z$ containing $w_*$, and let $H_0$ be their union.
A fiber path joining $w_*$ to another point of $H_{0,z}$ persists for nearby base points, so $H_0$ is open.
Set $H_n=F^n(H_0)$; these sets are open.
Since $\sigma$ has full support, there are $N\ge0$ and a Borel set $S\subset B$ of positive $\sigma$-measure such that each $z\in S$ and all its $T$-iterates avoid the exceptional set,
and $T^nz\notin B_n$ for every $n\ge N$.
Define
\[
 W_k=H_{N+k}\cap\pi^{-1}(T^{N+k}S),\qquad
 W=\bigsqcup_{k\ge0}W_k.
\]
These are Borel sets.
For $z\in S$, Lemma~\ref{lem:rel-full-slices} and \eqref{eq:rel-complete} say that $H_{n,T^nz}$ is a complete component of $C_{T^nz}$.
For $n\ge N$, its map to the next slice is proper and has no critical point.
Both slices are simply connected, so this map is a conformal isomorphism.
As the $U_n$ are disjoint and $T$ is injective, $F$ maps $W$ fiberwise bijectively onto $W\setminus W_0$.

Choose distinct constants $a,b$ with $|a|,|b|>R$, and set $A=\{\infty,a,b\}$.
Introduce increasing finite escape markers
\[
 E_0(z)=A,\qquad
 E_j(z)=A\cup q_z^{-1}\bigl(E_{j-1}(Tz)\bigr)
       =\bigcup_{n=0}^jQ_{n,z}^{-1}(A).
\]
They avoid every bounded fiber component.
Set $X_{j,z}=\sphere\setminus E_j(z)$ and let $\alpha_{j,z}=\alpha_{X_{j,z}}$ be the normalized Poincar\'e area measure from Definition~\ref{def:area}.
For $j\ge1$ the Schwarz--Pick inequality gives a nonnegative defect measure
\[
 \nu_{j,z}=\alpha_{j,z}-q_z^*\alpha_{j-1,Tz}
 \quad\text{on }X_{j,z}.
\]
Because $\infty\in q_z^{-1}(E_{j-1}(Tz))$, the fixed marker set $A$ adds at most two new punctures.
Counting distinct finite preimages,
\[
 |E_j(z)|\le d\bigl(|E_{j-1}(Tz)|-1\bigr)+3.
\]
Gauss--Bonnet and change of variables, ignoring only finitely many zero-area points, therefore give
\begin{equation}\label{eq:rel-defect}
 \nu_{j,z}(X_{j,z})
 =|E_j(z)|-2-d\bigl(|E_{j-1}(Tz)|-2\bigr)
 \le d+1.
\end{equation}

For any bounded fiber component $C\subset C_z$, the hyperbolic densities satisfy
\begin{equation}\label{eq:rel-kernel}
 \rho_{X_{j,z}}\uparrow\rho_C\quad\text{on }C.
\end{equation}
Indeed, a disk outside the closure of all the markers has every $Q_{n,z}$ omit the fixed triple $A$, and so lies in $\mathcal F_z$.
The marker closure misses $C$, because $C\subset\operatorname{int}K_z$.
As $\partial C$ lies in the fiber Julia set, $C$ is exactly its
component of the complement of the marker closure.
Now apply Lemma~\ref{core:kernel}, which uses Hejhal's kernel theorem \cite{Hejhal}.
Each such $C$ is a bounded simply connected plane domain, and therefore $\alpha_C(C)=\infty$.

The finite root sets $E_j(z)$ vary measurably with $z$, and so do their Poincar\'e densities, including at collisions of punctures; this follows from kernel convergence.
Extend their area measures by zero at punctures and set
\[
 M_j(H)=\int_V\alpha_{j,z}(H_z)\,\dd\sigma(z).
\]
Each $M_j$ is finite, since $|E_j(z)|\le3(1+d+\cdots+d^j)$.
The measures increase with $j$.
Fiberwise change of variables under the bijection $F:W\to W\setminus W_0$, and invariance of $\sigma$, imply
\[
 \int_V(q_z^*\alpha_{j-1,Tz})(W_z)\,\dd\sigma(z)
   =M_{j-1}(W\setminus W_0).
\]
Integrating \eqref{eq:rel-defect}, we obtain
\[
 d+1\ge M_j(W)-M_{j-1}(W\setminus W_0)
 =\bigl[M_j(W)-M_{j-1}(W)\bigr]+M_{j-1}(W_0)
 \ge M_{j-1}(W_0).
\]
On every selected fiber, $W_0$ is a complete bounded component and \eqref{eq:rel-kernel} gives $\alpha_{j,z}((W_0)_z)\uparrow\alpha_{(W_0)_z}((W_0)_z)=\infty$.
The base projection of $W_0$ has measure $\sigma(S)>0$.
Monotone convergence therefore yields $M_j(W_0)\to\infty$, a contradiction.
\end{proof}

For a unitary $T$ on a ball $B_r$, suppose that the coefficients are holomorphic on a neighborhood of $\overline{B_r}$ and the leading coefficient is nonzero throughout $\overline{B_r}$.
Then normalized Lebesgue measure satisfies \eqref{eq:rel-measure}, while compactness gives \eqref{eq:rel-coeff}.
Theorem \ref{thm:elliptic-skew} then excludes all bounded wandering components without a condition on $z\mapsto\Jul_z$.
The common exterior region $B_r\times(\{|w|>R\}\cup\{\infty\})$ lies in an invariant escape Fatou component; an escaping component eventually enters it.
Thus the compactified relative system on $B_r\times\sphere$ has no wandering Fatou component.

\section{Application: regular polynomial skew products}
\label{sec:elliptic-global}

A polynomial skew product
\[
 F(z,w)=(p(z),q(z,w)),\qquad (z,w)\in\C^2,
\]
is \emph{regular of degree $d$} if its homogeneous degree-$d$ part has no common nonzero zero, so that $F$ extends to a holomorphic endomorphism of $\mathbb P^2$ of algebraic degree $d$.
Theorem~\ref{thm:elliptic-global} combines the relative-area theorem on periodic Siegel base components with the known theorem of Lilov near the superattracting fiber at infinity \cite[Theorem~1.1]{PetersVivas2016}.

\begin{proof}[Proof of Theorem~\ref{thm:elliptic-global}]
Let $V$ be the union of all periodic Siegel components of $p$.
The filled Julia set of $p$ is compact, so $V$ is bounded; it may be empty or disconnected.
On each periodic Siegel cycle, a linearizing coordinate for the return map transports normalized area measure on the unit disk to an absolutely continuous invariant probability.
Average over the phases of that cycle and then take a positive summable mixture over all cycles.
When $V\ne\varnothing$, the resulting probability measure $\sigma$ has full support in $V$, is absolutely continuous with respect to planar area, and satisfies $p_*\sigma=\sigma$.
Also $p:V\to V$ is biholomorphic.

Regularity forces $\deg p=d$: otherwise the first leading homogeneous coordinate would vanish identically and the second would have a zero on $\mathbb P^1$.
Write $q(z,w)=\sum_{j=0}^d a_j(z)w^j$.
Regularity also implies that $a_d$ is a nonzero constant: the total degree is $d$, and a zero of the homogeneous leading pair on $[0:1]$ would otherwise occur.
All $a_j$ are bounded on $V$, since they are polynomials and $V$ is bounded.
Hence, when $V\ne\varnothing$, the invariant-measure form of Theorem~\ref{thm:elliptic-skew} applies to $F:V\times\C\to V\times\C$ and excludes wandering components of its bounded normality set.

We verify that this relative conclusion covers every globally bounded Fatou component.
Since the homogeneous part of $F$ has no nonzero common zero, there is an escape radius $R$ such that $\|F(x)\|>2\|x\|$ whenever $\|x\|>R$.
Thus each affine orbit either stays in the closed ball of radius $R$ or tends to the line at infinity.
A Fatou component of $F$ cannot contain points of both kinds.
Indeed, if an open subset escapes, any normal subsequential limit of the iterates sends that open subset into the line at infinity and hence sends the whole connected component into that line by the identity theorem.
At a bounded-orbit point the same subsequence has a finite limit, which is impossible.

Let $U$ be a Fatou component with bounded orbits.
All its iterates remain in the ball of radius $R$.
Consequently $p^n$ is uniformly bounded on the open connected projection $\pi(U)$, so $\pi(U)$ lies in one Fatou component of $p$.
That component cannot contain an escaping point: normality and the identity theorem give the same bounded-versus-escaping dichotomy as above.
Hence it is contained in the compact filled Julia set of $p$ and is bounded.
By hypothesis some iterate of that base component is a periodic Siegel component.
The full forward component $U_N$ of $U$ then lies in $V\times\C$: its projection is a connected Fatou subset meeting a periodic Siegel component and therefore lies in it.
The bounded normality set on $V\times\C$ is exactly the restriction of the global bounded normality set.
Thus $U_N$ is a full component to which Theorem~\ref{thm:elliptic-skew} applies.
It cannot be wandering.
If $V=\varnothing$, the same projection argument shows that no bounded Fatou component exists.

It remains to treat an escaping Fatou component.
The line at infinity has a superattracting point $P_\infty=[0:1:0]$.
In the chart $w\ne0$, the coordinates $(u,v)=(z/w,1/w)$ show directly that both components of $F(u,v)$ vanish to order at least $d$ at $(0,0)$; regularity makes their common denominator nonzero there.
Hence a component meeting a sufficiently small attracting neighborhood of $P_\infty$ belongs to its basin: normality and the identity theorem propagate convergence to $P_\infty$ from that open intersection to the whole component.
Every component of this basin eventually enters the invariant immediate basin, so it is preperiodic.

For the remaining escaping points the base coordinate tends to infinity.
Use the chart $z\ne0$ near the rest of the line at infinity, with
\[
  s=\frac1z,\qquad t=\frac wz.
\]
If $p(z)=a z^d+\text{lower-degree terms}$, then in these coordinates
\[
  (s,t)\longmapsto
  \left(\frac{s^d}{P(s)},\frac{Q(s,t)}{P(s)}\right),
  \qquad
  P(0)=a\ne0,\quad
  Q(0,t)=q_d(1,t),
\]
where $q_d$ is the homogeneous degree-$d$ part of $q$.
The base map is superattracting at $s=0$; the fiber map is a degree-$d$ polynomial in $t$ with nonzero leading coefficient, holomorphic in $s$ near $0$.
Its central fiber is $t\mapsto q_d(1,t)/a$.

Let $U$ be an escaping Fatou component outside the basin of $P_\infty$ and choose an affine point $(z_0,w_0)\in U$.
Its base orbit $z_n=p^n(z_0)$ tends to infinity.
Otherwise the base orbit remains bounded, and the escaping total orbit eventually enters an attracting neighborhood of $P_\infty$.
For large $n$ put $s_n=1/z_n$ and $t_n=w_n/z_n$.
The sequence $(t_n)$ is eventually bounded: if $|t_{n_k}|\to\infty$ along a subsequence, then in the $w\ne0$ chart $(u_{n_k},v_{n_k})=(1/t_{n_k},s_{n_k}/t_{n_k})\to(0,0)$, so that orbit also enters the attracting basin of $P_\infty$.
Choose $n$ such that $s_n$ belongs to the immediate basin of $s=0$.
Openness of the forward component $U_n$ supplies a vertical disk in $U_n$ over $s=s_n$.
Lilov's theorem, in the form stated by Peters and Vivas \cite[Theorem~1.1]{PetersVivas2016}, says that the forward orbit of this disk meets a bulging Fatou component of the central polynomial.
Such a component is preperiodic because the central one-dimensional Fatou component is preperiodic by Sullivan's theorem \cite{Sullivan}.
Hence the original global Fatou component is preperiodic.
This excludes wandering in the escaping case and proves the theorem.
\end{proof}

\begin{corollary}\label{cor:brjuno-base-global}
Let $d\ge2$, let $\lambda=e^{2\pi i\theta}$ with $\theta$ an irrational Brjuno number, and let $q(z,w)$ be a polynomial of total degree at most $d$ whose $w^d$ coefficient is nonzero.
Then
\[
  F(z,w)=\bigl(\lambda z+z^d,q(z,w)\bigr)
\]
extends holomorphically to $\mathbb P^2$ and has no wandering Fatou component there.
No dynamical assumption is imposed on the finite critical points of the fiber polynomials.
\end{corollary}

\begin{proof}
Put $p(z)=\lambda z+z^d$.
The Brjuno linearization theorem \cite{Yoccoz1995} gives a fixed Siegel disk $\Delta$ about $0$.
If $\zeta^{d-1}=1$, then
\[
 p(\zeta z)=\zeta p(z).
\]
The finite critical points, given by $z^{d-1}=-\lambda/d$, form a single orbit under this rotation symmetry.
We claim that they all lie in $\Jul(p)$.

Suppose instead that one finite critical point lies in $\Fat(p)$.
By symmetry all of them do.
Sullivan's one-dimensional no-wandering theorem \cite{Sullivan} and the classification of periodic polynomial Fatou components imply that the closure of each of their forward orbits meets $\Jul(p)$ in at most a finite parabolic cycle:
an escaping orbit converges to infinity, an attracting orbit converges inside its basin, and an orbit eventually in a Siegel disk remains on a compact rotation curve inside that disk.
Therefore the closure of the union of the finite critical orbits meets $\Jul(p)$ in only finitely many points.
But the boundary of $\Delta$ is an infinite subset of that postcritical closure \cite{Mane1993}, a contradiction.
This proves the claim.

An immediate finite attracting or parabolic basin of a polynomial contains a finite critical point.
There are therefore no such basins for $p$.
Polynomial maps have no Herman rings, so every bounded periodic Fatou component is a Siegel disk.
Sullivan's theorem then places every bounded Fatou component eventually in a periodic Siegel disk.

Finally the highest homogeneous terms of $F$ are $(z^d,q_d(z,w))$.
When $z=0$ and $w\ne0$, $q_d(0,w)=c w^d\ne0$, so they have no common nonzero zero.
Thus $F$ is regular and Theorem~\ref{thm:elliptic-global} applies.
\end{proof}

\begin{example}\label{ex:double-siegel-global}
Choose irrational Brjuno numbers $\theta,\eta$, put $\lambda=e^{2\pi i\theta}$ and $\mu=e^{2\pi i\eta}$, and take $\beta\in\C\setminus\{0\}$.
For every $d\ge2$,
\[
 F(z,w)=\bigl(\lambda z+z^d,\,
                  \mu w+w^d+\beta z\bigr)
\]
has no wandering Fatou component on $\mathbb P^2$ by Corollary~\ref{cor:brjuno-base-global}.
Both the base map and the center-fiber map $w\mapsto\mu w+w^d$ have Siegel disks.
The same critical-orbit symmetry argument as above places all finite critical points of the center-fiber map in its Julia set.
Thus this family does not satisfy the critical-point attraction or parabolic-capture hypothesis of Peters and Raissy \cite[Theorem~1]{PetersRaissy}.

For an explicit quadratic member take $\lambda=e^{2\pi i\sqrt2}$, $\mu=e^{2\pi i\sqrt3}$, and $\beta=1$.
After translating $w$ by $\mu/2$, it may be written
\[
 F(z,w)=\left(\lambda z+z^2,\,
             w^2+\frac{\mu}{2}-\frac{\mu^2}{4}+z\right).
\]
Its center fiber has a Siegel fixed point at $w=\mu/2$ with multiplier $\mu$, and its unique finite critical point $w=0$ lies in the Julia set.
The example also has a nonempty open set of bounded orbits.
At $(0,\mu/2)$ the derivative is diagonalizable, with eigenvalues $\lambda$ and $\mu$.
For integers $(a,b)\ne(0,0)$, choose the nearest integer $k$ to $a\sqrt2+b\sqrt3$.
The nonzero algebraic integer $a\sqrt2+b\sqrt3-k$ has norm of absolute value at least one in $\mathbb Q(\sqrt2,\sqrt3)$, while each of its other three conjugates is $O(1+|a|+|b|)$.
Consequently
\[
 |\lambda^a\mu^b-1|\ge
 c(1+|a|+|b|)^{-3}.
\]
The spectrum is nonresonant and satisfies the multidimensional Brjuno condition.
Brjuno's linearization theorem, in the form recorded in \cite[Theorem~1.2]{RaissyBrjuno}, therefore conjugates the map locally to $(u,v)\mapsto(\lambda u,\mu v)$.
The image of a sufficiently small invariant bidisk is a two-dimensional open subset of the bounded-orbit Fatou set.
\end{example}

\begin{example}\label{ex:bfp-global}
Boc Thaler, Forn\ae ss and Peters introduced the regular cubic skew product \cite[Section~2]{BocThalerFornaessPeters2015}
\[
 F_{\mathrm{BFP}}(z,w)
 =\bigl(\lambda z+z^3,\,
        \lambda^{-1}(w+zw^2)+w^3\bigr),
 \qquad \lambda=e^{2\pi i\sqrt2}.
\]
Their Theorem~7 constructs an invariant Fatou component whose orbits have a punctured-disk limit set in the invariant axis $\{w=0\}$.
Their Remark~2.1 shows that the Siegel disk in the distinct central fiber $\{z=0\}$ is entirely contained in $\Jul(F_{\mathrm{BFP}})$ \cite[Theorem~7 and Remark~2.1]{BocThalerFornaessPeters2015}.
The bounded-type choice of $\sqrt2$ makes both one-dimensional maps $z\mapsto\lambda z+z^3$ and $w\mapsto\lambda^{-1}w+w^3$ linearizable at $0$, as required for their construction.
Since the fiber polynomial of $F_{\mathrm{BFP}}$ has total degree three and $w^3$ coefficient one, Corollary~\ref{cor:brjuno-base-global} implies that this same map has no wandering Fatou components on $\mathbb P^2$.
The global no-wandering conclusion is the application of our criterion; the invariant component and non-bulging Siegel disk are the results of Boc Thaler, Forn\ae ss and Peters.
\end{example}

\section*{Acknowledgments}
I am deeply grateful to my advisor, Zhangchi Chen, for his generous support and encouragement throughout this work.
I also thank him and Professor Guizhen Cui for their thoughtful comments and constructive suggestions, which have greatly improved the manuscript.

\section*{Use of artificial intelligence}
The author formulated the initial idea and guiding intuition and established the research strategy.
Working within this strategy and under the author's direction, ChatGPT (OpenAI) identified and developed the concrete arguments used in the proofs.
The author also used ChatGPT to assist with organizing and revising the manuscript and checking the mathematical arguments.
The author takes full responsibility for the mathematical content, the references, and the final presentation of the paper.

\end{document}